\documentclass[nohyperref]{article}

\usepackage{microtype}
\usepackage{graphicx}
\usepackage{subfigure}
\usepackage{booktabs} %

\usepackage{hyperref}

\usepackage[accepted]{icml2023}

\usepackage{amsmath}
\usepackage{amssymb}
\usepackage{mathtools}
\usepackage{amsthm}

\usepackage[capitalize,noabbrev]{cleveref}

\theoremstyle{plain}
\newtheorem{theorem}{Theorem}[section]
\newtheorem{proposition}[theorem]{Proposition}
\newtheorem{lemma}[theorem]{Lemma}
\newtheorem{corollary}[theorem]{Corollary}

\theoremstyle{definition}
\newtheorem{definition}[theorem]{Definition}
\newtheorem{assumption}{Assumption}

\theoremstyle{remark}
\newtheorem{remark}[theorem]{Remark}
\newtheorem{example}[theorem]{Example}

\usepackage[textsize=tiny]{todonotes}

\icmltitlerunning{Policy Mirror Ascent for Efficient and Independent Learning in Mean Field Games}

\usepackage{mathrsfs}
\usepackage{xspace}
\usepackage{bm}
\usepackage{upgreek}
\usepackage{bbm}
\usepackage{xfrac}

\usepackage{aligned-overset}

\newcommand{\safemath}[2]{\newcommand{#1}{\ensuremath{#2}\xspace}}

\safemath{\bma}{\mathbf{a}}
\safemath{\bmb}{\mathbf{b}}
\safemath{\bmc}{\mathbf{c}}
\safemath{\bmd}{\mathbf{d}}
\safemath{\bme}{\mathbf{e}}
\safemath{\bmf}{\mathbf{f}}
\safemath{\bmg}{\mathbf{g}}
\safemath{\bmh}{\mathbf{h}}
\safemath{\bmi}{\mathbf{i}}
\safemath{\bmj}{\mathbf{j}}
\safemath{\bmk}{\mathbf{k}}
\safemath{\bml}{\mathbf{l}}
\safemath{\bmm}{\mathbf{m}}
\safemath{\bmn}{\mathbf{n}}
\safemath{\bmo}{\mathbf{o}}
\safemath{\bmp}{\mathbf{p}}
\safemath{\bmq}{\mathbf{q}}
\safemath{\bmr}{\mathbf{r}}
\safemath{\bms}{\mathbf{s}}
\safemath{\bmt}{\mathbf{t}}
\safemath{\bmu}{\mathbf{u}}
\safemath{\bmv}{\mathbf{v}}
\safemath{\bmw}{\mathbf{w}}
\safemath{\bmx}{\mathbf{x}}
\safemath{\bmy}{\mathbf{y}}
\safemath{\bmz}{\mathbf{z}}
\safemath{\bmzero}{\mathbf{0}}
\safemath{\bmone}{\mathbf{1}}
\safemath{\bmpi}{\pmb{\pi}}
\safemath{\bmalpha}{\pmb{\alpha}}

\bmdefine{\biad}{a}
\bmdefine{\bibd}{b}
\bmdefine{\bicd}{c}
\bmdefine{\bidd}{d}
\bmdefine{\bied}{e}
\bmdefine{\bifd}{f}
\bmdefine{\bigd}{g}
\bmdefine{\bihd}{h}
\bmdefine{\biid}{i}
\bmdefine{\bijd}{j}
\bmdefine{\bikd}{k}
\bmdefine{\bild}{l}
\bmdefine{\bimd}{m}
\bmdefine{\bind}{n}
\bmdefine{\biod}{o}
\bmdefine{\bipd}{p}
\bmdefine{\biqd}{q}
\bmdefine{\bird}{r}
\bmdefine{\bisd}{s}
\bmdefine{\bitd}{t}
\bmdefine{\biud}{u}
\bmdefine{\bivd}{v}
\bmdefine{\biwd}{w}
\bmdefine{\bixd}{x}
\bmdefine{\biyd}{y}
\bmdefine{\bizd}{z}

\bmdefine{\bixid}{\xi}
\bmdefine{\bilambdad}{\lambda}
\bmdefine{\bimud}{\mu}
\bmdefine{\binud}{\nu}
\bmdefine{\bithetad}{\theta}
\bmdefine{\biomegad}{\omega}
\bmdefine{\biphid}{\phi}

\safemath{\bmia}{\biad}
\safemath{\bmib}{\bibd}
\safemath{\bmic}{\bicd}
\safemath{\bmid}{\bidd}
\safemath{\bmie}{\bied}
\safemath{\bmif}{\bifd}
\safemath{\bmig}{\bigd}
\safemath{\bmih}{\bihd}
\safemath{\bmii}{\biid}
\safemath{\bmij}{\bijd}
\safemath{\bmik}{\bikd}
\safemath{\bmil}{\bild}
\safemath{\bmim}{\bimd}
\safemath{\bmin}{\bind}
\safemath{\bmio}{\biod}
\safemath{\bmip}{\bipd}
\safemath{\bmiq}{\biqd}
\safemath{\bmir}{\bird}
\safemath{\bmis}{\bisd}
\safemath{\bmit}{\bitd}
\safemath{\bmiu}{\biud}
\safemath{\bmiv}{\bivd}
\safemath{\bmiw}{\biwd}
\safemath{\bmix}{\bixd}
\safemath{\bmiy}{\biyd}
\safemath{\bmiz}{\bizd}

\safemath{\bmxi}{\bixid}
\safemath{\bmlambda}{\bilambdad}
\safemath{\bmmu}{\bimud}
\safemath{\bmnu}{\binud}
\safemath{\bmtheta}{\bithetad}
\safemath{\bmomega}{\biomegad}
\safemath{\bmphi}{\biphid}

\safemath{\bA}{\mathbf{A}}
\safemath{\bB}{\mathbf{B}}
\safemath{\bC}{\mathbf{C}}
\safemath{\bD}{\mathbf{D}}
\safemath{\bE}{\mathbf{E}}
\safemath{\bF}{\mathbf{F}}
\safemath{\bG}{\mathbf{G}}
\safemath{\bH}{\mathbf{H}}
\safemath{\bI}{\mathbf{I}}
\safemath{\bJ}{\mathbf{J}}
\safemath{\bK}{\mathbf{K}}
\safemath{\bL}{\mathbf{L}}
\safemath{\bM}{\mathbf{M}}
\safemath{\bN}{\mathbf{N}}
\safemath{\bO}{\mathbf{O}}
\safemath{\bP}{\mathbf{P}}
\safemath{\bQ}{\mathbf{Q}}
\safemath{\bR}{\mathbf{R}}
\safemath{\bS}{\mathbf{S}}
\safemath{\bT}{\mathbf{T}}
\safemath{\bU}{\mathbf{U}}
\safemath{\bV}{\mathbf{V}}
\safemath{\bW}{\mathbf{W}}
\safemath{\bX}{\mathbf{X}}
\safemath{\bY}{\mathbf{Y}}
\safemath{\bZ}{\mathbf{Z}}

\safemath{\bZero}{\mathbf{0}}
\safemath{\bOne}{\mathbf{1}}
\safemath{\bDelta}{\mathbf{\Delta}}
\safemath{\bLambda}{\mathbf{\UpLambda}}
\safemath{\bPhi}{\mathbf{\Upphi}}
\safemath{\bSigma}{\mathbf{\Upsigma}}
\safemath{\bOmega}{\mathbf{\Upomega}}
\safemath{\bTheta}{\mathbf{\Uptheta}}

\bmdefine{\biAd}{A}
\bmdefine{\biBd}{B}
\bmdefine{\biCd}{C}
\bmdefine{\biDd}{D}
\bmdefine{\biEd}{E}
\bmdefine{\biFd}{F}
\bmdefine{\biGd}{G}
\bmdefine{\biHd}{H}
\bmdefine{\biId}{I}
\bmdefine{\biJd}{J}
\bmdefine{\biKd}{K}
\bmdefine{\biLd}{L}
\bmdefine{\biMd}{M}
\bmdefine{\biOd}{N}
\bmdefine{\biPd}{O}
\bmdefine{\biQd}{P}
\bmdefine{\biRd}{R}
\bmdefine{\biSd}{S}
\bmdefine{\biTd}{T}
\bmdefine{\biUd}{U}
\bmdefine{\biVd}{V}
\bmdefine{\biWd}{W}
\bmdefine{\biXd}{X}
\bmdefine{\biYd}{Y}
\bmdefine{\biZd}{Z}

\bmdefine{\biDelta}{\Delta}
\bmdefine{\biLambda}{\Lambda}
\bmdefine{\biPhi}{\Phi}
\bmdefine{\biSigma}{\Sigma}
\bmdefine{\biOmega}{\Omega}
\bmdefine{\biTheta}{\Theta}

\safemath{\bimA}{\biAd}
\safemath{\bimB}{\biBd}
\safemath{\bimC}{\biCd}
\safemath{\bimD}{\biDd}
\safemath{\bimE}{\biEd}
\safemath{\bimF}{\biFd}
\safemath{\bimG}{\biGd}
\safemath{\bimH}{\biHd}
\safemath{\bimI}{\biId}
\safemath{\bimJ}{\biJd}
\safemath{\bimK}{\biKd}
\safemath{\bimL}{\biLd}
\safemath{\bimM}{\biMd}
\safemath{\bimN}{\biNd}
\safemath{\bimO}{\biOd}
\safemath{\bimP}{\biPd}
\safemath{\bimQ}{\biQd}
\safemath{\bimR}{\biRd}
\safemath{\bimS}{\biSd}
\safemath{\bimT}{\biTd}
\safemath{\bimU}{\biUd}
\safemath{\bimV}{\biVd}
\safemath{\bimW}{\biWd}
\safemath{\bimX}{\biXd}
\safemath{\bimY}{\biYd}
\safemath{\bimZ}{\biZd}

\safemath{\bimDelta}{\biDelta}
\safemath{\bimLambda}{\biLambda}
\safemath{\bimPhi}{\biPhi}
\safemath{\bimSigma}{\biSigma}
\safemath{\bimOmega}{\biOmega}
\safemath{\bimTheta}{\biTheta}

\safemath{\setA}{\mathcal{A}}
\safemath{\setB}{\mathcal{B}}
\safemath{\setC}{\mathcal{C}}
\safemath{\setD}{\mathcal{D}}
\safemath{\setE}{\mathcal{E}}
\safemath{\setF}{\mathcal{F}}
\safemath{\setG}{\mathcal{G}}
\safemath{\setH}{\mathcal{H}}
\safemath{\setI}{\mathcal{I}}
\safemath{\setJ}{\mathcal{J}}
\safemath{\setK}{\mathcal{K}}
\safemath{\setL}{\mathcal{L}}
\safemath{\setM}{\mathcal{M}}
\safemath{\setN}{\mathcal{N}}
\safemath{\setO}{\mathcal{O}}
\safemath{\setP}{\mathcal{P}}
\safemath{\setQ}{\mathcal{Q}}
\safemath{\setR}{\mathcal{R}}
\safemath{\setS}{\mathcal{S}}
\safemath{\setT}{\mathcal{T}}
\safemath{\setU}{\mathcal{U}}
\safemath{\setV}{\mathcal{V}}
\safemath{\setW}{\mathcal{W}}
\safemath{\setX}{\mathcal{X}}
\safemath{\setY}{\mathcal{Y}}
\safemath{\setZ}{\mathcal{Z}}
\safemath{\emptySet}{\varnothing}

\safemath{\colA}{\mathscr{A}}
\safemath{\colB}{\mathscr{B}}
\safemath{\colC}{\mathscr{C}}
\safemath{\colD}{\mathscr{D}}
\safemath{\colE}{\mathscr{E}}
\safemath{\colF}{\mathscr{F}}
\safemath{\colG}{\mathscr{G}}
\safemath{\colH}{\mathscr{H}}
\safemath{\colI}{\mathscr{I}}
\safemath{\colJ}{\mathscr{J}}
\safemath{\colK}{\mathscr{K}}
\safemath{\colL}{\mathscr{L}}
\safemath{\colM}{\mathscr{M}}
\safemath{\colN}{\mathscr{N}}
\safemath{\colO}{\mathscr{O}}
\safemath{\colP}{\mathscr{P}}
\safemath{\colQ}{\mathscr{Q}}
\safemath{\colR}{\mathscr{R}}
\safemath{\colS}{\mathscr{S}}
\safemath{\colT}{\mathscr{T}}
\safemath{\colU}{\mathscr{U}}
\safemath{\colV}{\mathscr{V}}
\safemath{\colW}{\mathscr{W}}
\safemath{\colX}{\mathscr{X}}
\safemath{\colY}{\mathscr{Y}}
\safemath{\colZ}{\mathscr{Z}}

\safemath{\opA}{\mathbb{A}}
\safemath{\opB}{\mathbb{B}}
\safemath{\opC}{\mathbb{C}}
\safemath{\opD}{\mathbb{D}}
\safemath{\opE}{\mathbb{E}}
\safemath{\opF}{\mathbb{F}}
\safemath{\opG}{\mathbb{G}}
\safemath{\opH}{\mathbb{H}}
\safemath{\opI}{\mathbb{I}}
\safemath{\opJ}{\mathbb{J}}
\safemath{\opK}{\mathbb{K}}
\safemath{\opL}{\mathbb{L}}
\safemath{\opM}{\mathbb{M}}
\safemath{\opN}{\mathbb{N}}
\safemath{\opO}{\mathbb{O}}
\safemath{\opP}{\mathbb{P}}
\safemath{\opQ}{\mathbb{Q}}
\safemath{\opR}{\mathbb{R}}
\safemath{\opS}{\mathbb{S}}
\safemath{\opT}{\mathbb{T}}
\safemath{\opU}{\mathbb{U}}
\safemath{\opV}{\mathbb{V}}
\safemath{\opW}{\mathbb{W}}
\safemath{\opX}{\mathbb{X}}
\safemath{\opY}{\mathbb{Y}}
\safemath{\opZ}{\mathbb{Z}}
\safemath{\opZero}{\mathbb{O}}
\safemath{\identityop}{\opI}

\safemath{\veca}{\bma}
\safemath{\vecb}{\bmb}
\safemath{\vecc}{\bmc}
\safemath{\vecd}{\bmd}
\safemath{\vece}{\bme}
\safemath{\vecf}{\bmf}
\safemath{\vecg}{\bmg}
\safemath{\vech}{\bmh}
\safemath{\veci}{\bmi}
\safemath{\vecj}{\bmj}
\safemath{\veck}{\bmk}
\safemath{\vecl}{\bml}
\safemath{\vecm}{\bmm}
\safemath{\vecn}{\bmn}
\safemath{\veco}{\bmo}
\safemath{\vecp}{\bmmp}
\safemath{\vecq}{\bmq}
\safemath{\vecr}{\bmr}
\safemath{\vecs}{\bms}
\safemath{\vect}{\bmt}
\safemath{\vecu}{\bmu}
\safemath{\vecv}{\bmv}
\safemath{\vecw}{\bmw}
\safemath{\vecx}{\bmx}
\safemath{\vecy}{\bmy}
\safemath{\vecz}{\bmz}

\safemath{\veczero}{\bmzero}
\safemath{\vecone}{\bmone}
\safemath{\vecxi}{\bmxi}
\safemath{\veclambda}{\bmlambda}
\safemath{\vecmu}{\bmmu}
\safemath{\vecnu}{\bmnu}

\safemath{\vecomega}{\bmomega}
\safemath{\vectheta}{\bmtheta}
\safemath{\vecphi}{\bmphi}
\safemath{\vecpi}{\bmpi}
\safemath{\vecalpha}{\bmalpha}

\safemath{\matA}{\bA}
\safemath{\matB}{\bB}
\safemath{\matC}{\bC}
\safemath{\matD}{\bD}
\safemath{\matE}{\bE}
\safemath{\matF}{\bF}
\safemath{\matG}{\bG}
\safemath{\matH}{\bH}
\safemath{\matI}{\bI}
\safemath{\matJ}{\bJ}
\safemath{\matK}{\bK}
\safemath{\matL}{\bL}
\safemath{\matM}{\bM}
\safemath{\matN}{\bN}
\safemath{\matO}{\bO}
\safemath{\matP}{\bP}
\safemath{\matQ}{\bQ}
\safemath{\matR}{\bR}
\safemath{\matS}{\bS}
\safemath{\matT}{\bT}
\safemath{\matU}{\bU}
\safemath{\matV}{\bV}
\safemath{\matW}{\bW}
\safemath{\matX}{\bX}
\safemath{\matY}{\bY}
\safemath{\matZ}{\bZ}
\safemath{\matzero}{\bmzero}

\safemath{\matDelta}{\bDelta}
\safemath{\matLambda}{\bLambda}
\safemath{\matPhi}{\bPhi}
\safemath{\matSigma}{\bSigma}
\safemath{\matOmega}{\bOmega}
\safemath{\matTheta}{\bTheta}

\safemath{\matidentity}{\matI}
\safemath{\matone}{\matO}

\safemath{\rnda}{A}
\safemath{\rndb}{B}
\safemath{\rndc}{C}
\safemath{\rndd}{D}
\safemath{\rnde}{E}
\safemath{\rndf}{F}
\safemath{\rndg}{G}
\safemath{\rndh}{H}
\safemath{\rndi}{I}
\safemath{\rndj}{J}
\safemath{\rndk}{K}
\safemath{\rndl}{L}
\safemath{\rndm}{M}
\safemath{\rndn}{N}
\safemath{\rndo}{O}
\safemath{\rndp}{P}
\safemath{\rndq}{Q}
\safemath{\rndr}{R}
\safemath{\rnds}{S}
\safemath{\rndt}{T}
\safemath{\rndu}{U}
\safemath{\rndv}{V}
\safemath{\rndw}{W}
\safemath{\rndx}{X}
\safemath{\rndy}{Y}
\safemath{\rndz}{Z}

\safemath{\rveca}{\bimA}
\safemath{\rvecb}{\bimB}
\safemath{\rvecc}{\bimC}
\safemath{\rvecd}{\bimD}
\safemath{\rvece}{\bimE}
\safemath{\rvecf}{\bimF}
\safemath{\rvecg}{\bimG}
\safemath{\rvech}{\bimH}
\safemath{\rveci}{\bimI}
\safemath{\rvecj}{\bimJ}
\safemath{\rveck}{\bimK}
\safemath{\rvecl}{\bimL}
\safemath{\rvecm}{\bimM}
\safemath{\rvecn}{\bimN}
\safemath{\rveco}{\bomO}
\safemath{\rvecp}{\bimP}
\safemath{\rvecq}{\bimQ}
\safemath{\rvecr}{\bimR}
\safemath{\rvecs}{\bimS}
\safemath{\rvect}{\bimT}
\safemath{\rvecu}{\bimU}
\safemath{\rvecv}{\bimV}
\safemath{\rvecw}{\bimW}
\safemath{\rvecx}{\bimX}
\safemath{\rvecy}{\bimY}
\safemath{\rvecz}{\bimZ}

\safemath{\rvecxi}{\bmxi}
\safemath{\rveclambda}{\bmlambda}
\safemath{\rvecmu}{\bmmu}
\safemath{\rvectheta}{\bmtheta}
\safemath{\rvecphi}{\bmphi}

\safemath{\rmatA}{\bimA}
\safemath{\rmatB}{\bimB}
\safemath{\rmatC}{\bimC}
\safemath{\rmatD}{\bimD}
\safemath{\rmatE}{\bimE}
\safemath{\rmatF}{\bimF}
\safemath{\rmatG}{\bimG}
\safemath{\rmatH}{\bimH}
\safemath{\rmatI}{\bimI}
\safemath{\rmatJ}{\bimJ}
\safemath{\rmatK}{\bimK}
\safemath{\rmatL}{\bimL}
\safemath{\rmatM}{\bimM}
\safemath{\rmatN}{\bimN}
\safemath{\rmatO}{\bimO}
\safemath{\rmatP}{\bimP}
\safemath{\rmatQ}{\bimQ}
\safemath{\rmatR}{\bimR}
\safemath{\rmatS}{\bimS}
\safemath{\rmatT}{\bimT}
\safemath{\rmatU}{\bimU}
\safemath{\rmatV}{\bimV}
\safemath{\rmatW}{\bimW}
\safemath{\rmatX}{\bimX}
\safemath{\rmatY}{\bimY}
\safemath{\rmatZ}{\bimZ}

\safemath{\rmatDelta}{\bimDelta}
\safemath{\rmatLambda}{\bimLambda}
\safemath{\rmatPhi}{\bimPhi}
\safemath{\rmatSigma}{\bimSigma}
\safemath{\rmatOmega}{\bimOmega}
\safemath{\rmatTheta}{\bimTheta}

\newenvironment{textbmatrix}{	\setlength{\arraycolsep}{2.5pt}%
								\big[\begin{matrix}}{\end{matrix}\big]%
								\raisebox{0.08ex}{\vphantom{M}}}

\def\be{\begin{equation}}
\def\ee{\end{equation}}
\def\een{\nonumber \end{equation}}
\def\mat{\begin{bmatrix}}
\def\emat{\end{bmatrix}}
\def\btm{\begin{textbmatrix}}
\def\etm{\end{textbmatrix}}

\def\ba#1\ea{\begin{align}#1\end{align}}
\def\bas#1\eas{\begin{align*}#1\end{align*}}
\def\bs#1\es{\begin{split}#1\end{split}} 
\def\bg#1\eg{\begin{gather}#1\end{gather}} 
\def\bi#1\ei{\begin{itemize}#1\end{itemize}}

\DeclareMathOperator*{\argmin}{arg\;min}		%
\DeclareMathOperator*{\argmax}{arg\;max}		%
\DeclareMathOperator{\Prob}{\opP}			%
\DeclareMathOperator{\Exop}{\opE}			%
\DeclareMathOperator{\grad}{\nabla}			%
\DeclareMathOperator{\supp}{supp}			%

\newcommand{\ind}[1]{\mathbbm{1}_{#1}}				%

\safemath{\dirac}{\delta}					%
\safemath{\krond}{\dirac}					%
\safemath{\upto}{\uparrow}
\safemath{\downto}{\downarrow}
\safemath{\iu}{j}							%
\safemath{\ev}{\lambda}						%
\safemath{\hilseqspace}{l^{2}}				%
\newcommand{\banachfunspace}[1]{\setL^{#1}}	%
\safemath{\hilfunspace}{\banachfunspace{2}}	%

\safemath{\SNR}{\text{\sc snr}} 				%
\safemath{\No}{N_0}							%
\safemath{\Es}{E_s}							%
\safemath{\Eb}{E_b}							%
\safemath{\EbNo}{\frac{\Eb}{\No}}
\safemath{\EsNo}{\frac{\Es}{\No}}

\DeclareMathOperator{\CHop}{\ensuremath{\opH}} %
\safemath{\tvir}{\rndh_{\CHop}}				%
\safemath{\tvtf}{\rndl_{\CHop}}				%
\safemath{\spf}{\rnds_{\CHop}}				%
\safemath{\bff}{H_{\CHop}}					%

\safemath{\ircf}{r_{h}}						%
\safemath{\tftvcf}{r_{s}}					%
\safemath{\tfcf}{r_{l}}						%
\safemath{\bfcf}{r_{H}}						%

\safemath{\tcorr}{c_h}						%
\safemath{\scf}{c_{s}}						%
\safemath{\tfcorr}{c_{l}}					%
\safemath{\fcorr}{c_{H}}						%

\safemath{\mi}{I}							%
\safemath{\capacity}{C}						%

\safemath{\normal}{\mathcal{N}}			%
\safemath{\jpg}{\mathcal{CN}}			%
\safemath{\mchain}{\leftrightarrow}		%

\safemath{\dB}{\,\mathrm{dB}}
\safemath{\dBm}{\,\mathrm{dBm}}
\safemath{\Hz}{\,\mathrm{Hz}}
\safemath{\kHz}{\,\mathrm{kHz}}
\safemath{\MHz}{\,\mathrm{MHz}}
\safemath{\GHz}{\,\mathrm{GHz}}
\safemath{\s}{\,\mathrm{s}}
\safemath{\ms}{\,\mathrm{ms}}
\safemath{\mus}{\,\mathrm{\mu s}}
\safemath{\ns}{\,\mathrm{ns}}
\safemath{\meter}{\,\mathrm{m}}
\safemath{\mm}{\,\mathrm{mm}}
\safemath{\cm}{\,\mathrm{cm}}
\safemath{\m}{\,\mathrm{m}}
\safemath{\W}{\,\mathrm{W}}
\safemath{\J}{\,\mathrm{J}}
\safemath{\K}{\,\mathrm{K}}
\safemath{\bit}{\,\mathrm{bit}}

\safemath{\define}{=}			%

\safemath{\equivalent}{\sim}
\safemath{\distas}{\sim}					%
\safemath{\sdiff}{\Delta}				%

\safemath{\reals}{\mathbb{R}}
\safemath{\positivereals}{\reals_{+}}
\safemath{\integers}{\mathbb{Z}}
\safemath{\posint}{\integers_{+}}
\safemath{\naturals}{\mathbb{N}}
\safemath{\posnaturals}{\naturals_{+}}
\safemath{\complexset}{\mathbb{C}}
\safemath{\rationals}{\mathbb{Q}}

\newcommand{\vertiii}[1]{{\left\vert\kern-0.25ex\left\vert\kern-0.25ex\left\vert #1 
    \right\vert\kern-0.25ex\right\vert\kern-0.25ex\right\vert}}

\begin{document}

\twocolumn[
\icmltitle{Policy Mirror Ascent for\\ Efficient and Independent 
Learning in Mean Field Games}

\icmlsetsymbol{equal}{*}

\begin{icmlauthorlist}
\icmlauthor{Batuhan Yardim}{eth}
\icmlauthor{Semih Cayci}{rwth}
\icmlauthor{Matthieu Geist}{google}
\icmlauthor{Niao He}{eth}
\end{icmlauthorlist}

\icmlaffiliation{eth}{Department of Computer Science, ETH Zürich, Zürich, Switzerland}
\icmlaffiliation{google}{Google Research, Brain team}
\icmlaffiliation{rwth}{Department of Mathematics, RWTH Aachen University, Aachen, Germany}

\icmlcorrespondingauthor{Batuhan Yardim}{alibatuhan.yardim@inf.ethz.ch}

\icmlkeywords{mean field games, policy mirror ascent, independent learning}

\vskip 0.3in
]

\printAffiliationsAndNotice{}  %

\begin{abstract}
Mean-field games have been used as a theoretical tool to obtain an approximate Nash equilibrium for symmetric and anonymous $N$-player games.
However, limiting applicability, existing theoretical results assume variations of a ``population generative model'', which allows arbitrary modifications of the population distribution by the learning algorithm.
Moreover, learning algorithms typically work on abstract simulators with population instead of the $N$-player game.
Instead, we show that $N$ agents running policy mirror ascent converge to the Nash equilibrium of the regularized game within $\widetilde{\mathcal{O}}(\varepsilon^{-2})$ samples from a single sample trajectory without a population generative model, up to a standard $\mathcal{O}(\frac{1}{\sqrt{N}})$ error due to the mean field.
Taking a divergent approach from the literature, instead of working with the best-response map we first show that a policy mirror ascent map can be used to construct a contractive operator having the Nash equilibrium as its fixed point.
We analyze single-path TD learning for $N$-agent games, proving sample complexity guarantees by only using a sample path from the $N$-agent simulator without a population generative model.
Furthermore, we demonstrate that our methodology allows for independent learning by $N$ agents with finite sample guarantees.
\end{abstract}

\section{Introduction}

Multiagent reinforcement learning (MARL) is a fundamentally challenging problem, despite this with a wide spectrum of applications, for instance in finance \citep{shavandi2022multi}, civil engineering \citep{wiering2000multi}, multi-player games \citep{samvelyan2019starcraft}, energy markets \citep{rashedi2016markov}, robotic control \citep{matignon2007hysteretic}, and cloud resource management \citep{mao2022mean}.

The mean field game (MFG) framework, first proposed by \citet{lasry2007mean} and \citet{huang2006large}, is a useful theoretical tool for analyzing a specific class of MARL problems for the $N$-player case when $N$ is large.
As its main insight, MFG analyzes the limiting game when $N\rightarrow\infty$, under the condition that the rewards and transition dynamics of the game are \emph{symmetric} for each agent and depend only on the distribution of agents' states (i.e., the agents are \emph{anonymous}).
Conceptually, MFG formulates a representative agent playing against a \emph{distribution} of agents.
This abstract framework makes it tractable to theoretically characterize an approximate Nash equilibrium for the $N$-agent MARL, coinciding with the true Nash equilibrium as $N$ grows \citep{anahtarci2022q,saldi2018markov}.
Moreover, one can efficiently learn such ``MFG Nash equilibria'' from repeated plays \citep[for instance, see][]{perrin2020fictitious}.
The MFG formalism is useful in analyzing for instance financial systems, auctions, and city planning.

While sharing similar ideas in principle, there have been various mathematical formalizations of MFGs.
In the finite horizon case, the Lasry-Lions conditions have been analyzed to obtain Nash equilibria with time-dependent policies \citep{perrin2020fictitious, perrin2021mean}.
In the infinite horizon setting, the time-dependent evolution of the population becomes a challenge.
One can either consider policies depending on the population distribution \citep{yang2018learning,carmona2019model}, or consider Nash equilibria induced by stationary population distributions \citep{anahtarci2022q,xie2021learning,koppel2022oracle}.
This paper studies this ``stationary MFG'' problem, formalized in the following section.
The results of our work are juxtaposed with existing theory in Table~\ref{table:juxtaposition}.

In parallel to MFG literature, there exists a plethora of results regarding policy gradient (PG) methods in single agent RL \citep{agarwal2021theory, lan2022policy, cen2020fast, mei2020global, li2022homotopic, tomar2021mirror}.
Such methods are typically on-policy algorithms, as opposed to Q-learning which has been employed in MFG literature \citep[for instance in][]{anahtarci2022q, koppel2022oracle}.
One recent result in single-agent RL suggests that using a mirror descent style operator can achieve $\mathcal{O}(\varepsilon^{-1})$ sample complexity in regularized MDPs \citep{lan2022policy}.
The algorithm builds on conditional TD learning \citep{kotsalis2022simple}, which establishes that the value estimation problem can be solved with samples from a single trajectory from the Markov chain.
In the context of linear quadratic MFGs (LQ-MFG), policy gradient methods with entropy regularization \citep{guo2022entropy} as well as actor-critic methods \citet{fu2019actor} have been analyzed.
For general MFGs, policy-based methods combined with TD learning have been also analyzed as a method for approximating the optimal Q-value
\citep{guo2022general}.
A similar mirror descent operator to ours was considered for finite horizon MFGs with continuous time analysis in the exact case by \citet{perolat2022scaling}, with a heuristic extension to deep learning by \citet{lauriere2022scalable}.
We use insights from MFG and PG theory to obtain a $\widetilde{\mathcal{O}}(\varepsilon^{-2})$ time step complexity in the infinite horizon, stationary MFG setting. \looseness=-1

Moreover, existing methods for solving stationary MFGs have strong oracle assumptions regarding the population distribution.
For instance, \citet{anahtarci2022q} assume a generative model oracle that can produce samples from the game dynamics for any population distribution.
Similarly, the $N$-player weak simulator oracle proposed by \citet{guo2022general} can produce samples of state transitions and rewards for \emph{any} population distribution.
The oracle-free Sandbox Learning algorithm by \citet{koppel2022oracle} (as well as the actor-critic method of \citet{mao2022mean}) uses a two-timescale update of the population distribution based on model-based estimates of the state dynamics.
However, this result still requires that the population distribution at each time step can be manipulated by the algorithm.
In real-world problems, however, the $N$-player simulator might not allow the algorithm to modify or control the population of agents explicitly.
For example, when learning to control a transportation infrastructure with a large number of drivers, it might be challenging and costly to force the drivers into a particular configuration before simulating the system.
Or it might be impossible to run simulations starting from arbitrary population configurations, as is typical in strategic video games with unknown/complicated dynamics.
One of our main goals is to propose an algorithm that interacts with the simulator only by controlling the \emph{policies} of the agents, as done in the single-agent RL literature.\looseness=-1

Finally, in this paper, we tackle the question of independent learning.
Existing algorithms for MFGs rely on a centralized controller running simulations and coordinating the policies of agents.
However, we show that our policy mirror ascent based algorithm can be extended to an independent learning algorithm run by $N$ agents without additional knowledge of the population and environment, only observing their own actions, state transitions, and rewards.
Independent learning in $N$-player mean-field games have been studied by \citet{yongacoglu2022independent}, although they analyze asymptotic convergence to subjective (rather than objective) equilibria without finite sample bounds.
A similar idea of independent learning has been analyzed in the two agent setting by \citet{daskalakis2020independent} and \citet{sayin2021decentralized}.
In the more restrictive context of localized learning on graphs, \citet{gu2021mean} analyze independent learning where agent states are only affected by neighboring nodes on an undirected graph.
Independent learning in the case of large-scale Markov potential games has been analyzed by \citet{ding2022independent}.
For two-player zero-sum games, independent learning algorithms have been reviewed under various settings by \citet{ozdaglar2021independent}.

\begin{table*}
\centering
\begin{tabular}{c p{2.5cm} p{1.6cm}  p{1.8cm}  p{2cm}  p{2cm} p{1.7cm} p{2cm}} 
\hline & No population manipulation & Single path  &   $N$-agent simulator & Independent learning \\
\hline \cite{guo2019learning} & No & No &  No & No \\
\cite{anahtarci2022q} & No & No & No & No \\
\cite{subramanian2019reinforcement} & No & No & No & No \\
\cite{xie2021learning} & No & Yes & No & No \\
\cite{koppel2022oracle} & No & Yes & No & No \\
\textbf{This work-1}  &\textbf{Yes} &\textbf{Yes} &  \textbf{Yes} & \textbf{No} \\
\textbf{This work-2}  &\textbf{Yes} &\textbf{Yes} &  \textbf{Yes} & \textbf{Yes} \\
\hline
\end{tabular}
\caption{Theoretical results in the literature for computing stationary MFG-NE in discrete state-action spaces and their requirements.
}
\label{table:juxtaposition}
\end{table*}
To summarize, our contributions are the following ones. 

\textbf{Policy mirror ascent operators.}
To the best of our knowledge, practically all results in the \emph{stationary} MFG literature with finite state-action spaces rely on the conditions developed by \citet{anahtarci2022q} or \citet{cui2021approximately} for the best response map (from the set of populations to the set of policies) to be Lipschitz continuous or take it as a blanket assumption, yielding an operator contracting to the Nash equilibrium.
We take a different approach and \emph{prove} that policy mirror ascent (PMA) can yield a contraction under comparable conditions on the strong concavity of the regularizer.
This difference in the operators allows us to develop an algorithm that does not need to compute the best response at each step.

\textbf{No population manipulation.}
In this work, we only assume that we can simulate a single episode of $N$ agents playing a symmetric and anonymous game.
The proposed algorithm can only interact with the environment by manipulating the policies of agents and can not directly manipulate the states of each agent.
This is in contrast with past work where the empirical state distribution of agents can be arbitrarily set \citep{anahtarci2022q, guo2019learning}, mixed \citep{koppel2022oracle, mao2022mean} or projected \citep{koppel2022oracle, guo2019learning}. \looseness=-1

\textbf{TD learning with population.}
We show that by extending the single-agent conditional TD learning results of \citet{kotsalis2022simple}, one can efficiently perform TD learning in $N$-player games without a generative model.
In the absence of a population oracle, TD learning is conducted by simulating $N$ agents, introducing several complexities including a population bias and non-homogeneous dynamics due to the evolving population.
We establish that TD learning can be performed with single path simulations with $\widetilde{\mathcal{O}}(\varepsilon^{-2})$ samples, up to a standard $\mathcal{O}(\frac{1}{{\sqrt{N}}})$ error.

\textbf{Sample efficiency.}
Our algorithms differ from several past works in that the best response does not need to be recomputed for each population distribution (unlike for instance \cite{guo2019learning, anahtarci2022q}), and only a value function estimation is necessary for policy mirror ascent at each iteration.
This approach yields a time step complexity of $\widetilde{\mathcal{O}}(\varepsilon^{-2})$ as opposed to for example $\mathcal{O}(\varepsilon^{-4})$ in \citep{koppel2022oracle} and $\mathcal{O}(\varepsilon^{-4|\setA|})$ in \citep{anahtarci2022q}.

\paragraph{Independent learning.}
A fundamental question in competitive multi-agent learning is whether independent learning can be achieved as opposed to learning in the presence of a centralized controller.
This question is especially significant for mean-field games, where the added complexity of having a very large number of agents might not allow centralized learning in practice. 
To the best of our knowledge, we establish the first MFG algorithm for independent learning with $N$ agents with finite sample bounds.

\section{Mean Field Game Formalization}

Firstly, we introduce the (stationary) MFG problem.
We assume $\setS$ is a finite state space and $\setA$ is a finite action space.
We denote the set of probability measures on a  finite set $\setX$ by $\Delta_\setX$.
We denote the set of policies as $\Pi := \left\{ \pi: \setS \rightarrow \Delta_\setA \right\}$.
Let $h:\Delta_\setA \rightarrow \mathbb{R}_{\geq 0}$ be a given function.
\looseness=-1

We formally define the symmetric anonymous game with states (SAGS) with $N$ players, which is the main object of interest of this work.
In SAGS, the state and reward dynamics depend only on the empirical distribution of states among agents (hence anonymity) and are the same for each agent (hence the symmetry).
\begin{definition}[Symmetric anonymous games]
\label{def:SAG}
A $N$-player symmetric anonymous game is a tuple $(N, \setS, \setA, P, R, \gamma)$ for which any initial states of players $\{s_0^i\}_{i=1}^N\in \setS^N$ and policies $\{\pi^i\}_{i=1}^N \in \Pi^N$ induce a random sequence of states and rewards $\{s^i_t\}_{i,t}, \{r^i_t\}_{i,t}$, so that at each timestep $t\geq 0, \forall i = 1,\ldots,N$,
\begin{align*}
    a^i_{t} \sim \pi^i(s_t^i ), \,
    r^i_t = R( s^i_t, a^i_t, \widehat{\mu}_t), \,
    s^i_{t+1} \sim P(\cdot|  s^i_t, a^i_t, \widehat{\mu}_t), 
\end{align*}
where $P: \setS \times \setA \times \Delta_\setS \rightarrow \Delta_\setS$ and $R: \setS \times \setA \times \Delta_\setS \rightarrow [0,1]$ map state, action and population distribution tuples to transition probabilities and (bounded) rewards, and $\widehat{\mu}_t = \frac{1}{N} \sum_{i=1}^N \sum_{s\in\setS} \ind{s^i_t = s} \vece_s \in \Delta_\setS$ is the $\mathcal{F}(\{s^i_{t}\}_i)$-measurable random vector of the empirical state distribution at time $t$ of the agents.\looseness=-1
\end{definition}

With SAGS dynamics induced by sampling initial states from an initial distribution $\mu_0 \in \Delta_\setS$ and policies $\vecpi:=(\pi^1, \ldots, \pi^N) \in\Pi^N$, we can formalize the expected discounted returns of each agent. \looseness=-1

\begin{definition}[$N$-player discounted reward]
For the $N$-player SAGS $(N, \setS, \setA, P, R, \gamma)$, we define the (regularized) discounted return of player $i$ for initial state distribution $\mu_0\in\Delta_\setS$ and Markov policies $\vecpi\in\Pi^N$ as
$J_h^i(\vecpi,\mu_0) := \Exop \big[ \sum_{t=0}^\infty \gamma^t \left(R(s^i_t, a^i_t, \widehat{\mu}_{t}) + h(\pi^i(s^i_t))\right) | 
s_0^j \sim \mu_0,
a_t^j \sim \pi^j(s_t^j),
s_{t+1}^j \sim P(\cdot|  s^j_t, a^j_t, \widehat{\mu}_t), \forall t \geq 0, j \in 1, \ldots, N \big].$
\end{definition}

Based on the above definition, %
we introduce the notion of a Nash equilibrium (NE) to the $N$-player SAGS.
The NE is the natural solution concept for an $N$-player competitive game, therefore our main objective will be to compute an approximate NE for the SAGS.
\looseness=-1
\begin{definition}[$\delta$-Nash equilibrium]
For $\delta \geq 0$, an $N$-tuple of policies $\vecpi = (\pi^1, \ldots, \pi^N) \in \Pi^N$ and initial distribution $\mu_0 \in \Delta_\setS$ constitute a $\delta$-Nash equilibrium $(\vecpi, \mu_0)$ if for all $i=1, \ldots, N$ we have $J_h^i(\vecpi, \mu_0) \geq \max_{\pi\in\Pi} J_h^i\left((\pi, \vecpi^{-i}), \mu_0 \right) - \delta$, where $(\pi, \vecpi^{-i}) := (\pi^1, \ldots,\pi^{i-1}, \pi,\pi^{i+1},\ldots \pi^N) \in \Pi^N$.
\end{definition}

As a useful theoretical analysis tool one can approximate the $N$ player game with the infinite agent limit ($N\rightarrow\infty$).
This is the main insight of MFG: At the limit, one observes a single representative agent's policy playing against a population of infinitely many infinitesimal opponents characterized with a fixed distribution $\mu \in \Delta_\setS$.
Observing the limit $N\rightarrow\infty$, one can define the discounted expected reward with respect to the single-agent MDP parameterized (or induced) by the population.
We next define the discounted expected reward of a representative agent against such a population.
We use the letter $V$ to distinguish the expected ``infinite-agent'' reward from the expected rewards of the $N$-agent SAGS, denoted with the letter $J$.

\begin{definition}[Mean-field discounted reward]
We define the expected mean-field reward for a population-policy pair $(\pi,\mu)\in \Pi\times \Delta_\setS$ as $V_h(\pi, \mu) := \Exop \big[ \sum_{t=0}^\infty \gamma^t \left(R(s_t, a_t, \mu) + h(\pi(s_t))\right) | 
s_0 \sim \mu,
a_t \sim \pi(s_t),
s_{t+1} \sim P(\cdot | s_t, a_t, \mu)
\big]$.
\end{definition}

A Nash equilibrium concept can be also introduced at the MFG limit.
The so-called stationary MFG-NE is formulated with an optimality condition and a stability condition on $\pi, \mu$.
Intuitively, the stability condition ensures that the population distribution remains consistent, removing the need to consider time-varying distributions \citep[see][]{guo2019learning}.\looseness=-1
\begin{definition}[MFG-NE]
A policy $\pi^*\in \Pi$ and population distribution $\mu^* \in \Delta_\setS$ pair $(\pi^*, \mu^*)$ is called an MFG-NE if the following conditions are satisfied:
\begin{align*}
    &\text{Stability: } \mu^*(s) = \sum_{s', a'} \mu^*(s')\pi^*(a'|s')P(s|s',a', \mu^*),\\
    &\text{Optimality: } V_h(\pi^*, \mu^*) = \max_{\pi} V_h(\pi, \mu^*).
\end{align*}
If the optimality condition is replaced with $V_h(\pi^*_\delta, \mu^*_\delta) \geq \max_\pi V_h(\pi, \mu^*_\delta) - \delta$, we call $(\pi^*_\delta, \mu^*_\delta)$ a $\delta$-MFG-NE.
\end{definition}

While the optimality condition above corresponds to the objective of single-agent RL, incorporating a mean field game dynamic in Markov decision processes yields an \emph{evolving} MDP depending on the population distribution (in other words, induces an infinite family of MDPs to be solved).
Conceptually this challenge in MFGs mirrors that in the case of multi-agent MDPs, where each agent plays against an evolving environment.
We also comment on the effect of regularization (due to $h$) on the MFG-NE in the appendix; see Section~\ref{sec:regularization_and_bias}.

The main motivation to study the abstract MFG-NE concept is that it corresponds to an approximate Nash equilibrium of the $N$-player SAGS.
We re-iterate this well-known result.
\looseness=-1
\begin{proposition}[MFG-NE and NE (Theorem 1 of \citet{anahtarci2022q})]\label{proposition:finite_N}
Assume that the pair $(\pi^*, \mu^*)$ is an MFG-NE.
Under technical conditions, for each $\delta > 0$, there exists an $N = N(\delta) \in \mathbb{N}_{>0}$ such that $(\pi^*, \mu^*)$ is a $\delta$-Nash equilibrium for the $N$-player SAGS.
\end{proposition}
In fact, it can be shown by standard techniques (used also in our paper) that $(\pi^*, \mu^*)$ is a $\mathcal{O}\left(\frac{1}{\sqrt{N}}\right)$-Nash equilibrium for the $N$-player SAGS.
With this reduction, the main objective of this paper will be to learn the MFG-NE.
Our goal is to expand on (and relax) this existential result and show that a similar bound on bias can be shown when also \emph{learning} completely occurs in the finite agent setting, removing the abstract MFG formalism completely from the algorithm.

\section{Operators and the Exact PMA Case}\label{sec:operators}

In this section, we show that finding the MFG-NE in the infinite agent limit can be formulated as a fixed point iteration of an appropriately defined policy mirror ascent operator.
The results of this section will ultimately show convergence in the so-called ``exact'' setting, where value functions can be computed exactly.
These results will also be instrumental in establishing convergence later in the $N$-agent stochastic case.
We state the operators and major ideas, postponing the proofs to the appendices.
Theorems with omitted constants are restated in the appendix.
This section can be best compared to the work of \citet{anahtarci2022q}.
We also compare in greater detail existing assumptions in the literature to establish a contraction to the MFG-NE in the appendix (Section~\ref{section:appendix_discussion_assumptions}).

\paragraph{Definitions.}
We equip sets $\setS, \setA$ with the discrete metric $d(x, y) = \ind{x \neq y}$.
We denote the set of state-action value functions as $\setQ := \{ q:\setS \times \setA \rightarrow \mathbb{R} \}$.
We equip both of the spaces $\Delta_\setS \subset \mathbb{R}^\setS$ and $\Delta_\setA\subset \mathbb{R}^\setA$ with the norm $\| \cdot \|_1$.
For $\pi, \pi' \in \Pi, q, q'\in \setQ$ we use the norms $\| \pi - \pi'\|_1 := \sup_{s \in \setS} \| \pi(s) - \pi'(s) \|_1,
    \| q - q'\|_\infty := \sup_{s\in \setS,a\in \setA} |q(s,a) - q'(s,a)|,
    \| q - q'\|_2 := \sqrt{\sum_{s,a} |q(s,a) - q'(s,a)|^2}$.
Finally, we assume that $h:\Delta_\setA \rightarrow \mathbb{R}_{\geq 0}$ is a $\rho$-strongly concave function on $\Delta_\setA$ with respect to norm $\|\cdot\|_1$ (see Definition~\ref{def:strong_convexity} in the appendix). 
We define $u_{max} := \argmax_{u\in \Delta_\setA} h(u)$, $h_{max} := h(u_{max})$, $\pi_{max} \in \Pi$ such that $\pi_{max}(s) = u_{max}$ for all $s\in\setS$ and $Q_{max} := \frac{1 + h_{max}}{1 - \gamma}$. %
We further assume $h$ is continuously differentiable for simplicity, although the results yield a straightforward generalization to the case $h$ is not everywhere differentiable.
Finally, for any $\Delta h \in \mathbb{R}_{>0}$ we define the convex sets 
\begin{align}
\label{eq:definition_pi_deltah}
    \setU_{\Delta h} &:= \{ u \in \Delta_\setA : h(u) \geq h_{max} - \Delta h\}, \\
    \Pi_{\Delta h} &:= \{ \pi \in \Pi : \pi(s) \in \setU_{\Delta h}, \forall s\in\setS\}.
\end{align}
As standard in previous work, we require the following smoothness assumptions on $P, R$.
\begin{assumption}[Lipschitz continuity of $P, R$]\label{assumption:lipschitz}
There exists constants $K_\mu, K_s, K_a, L_\mu, L_s,$ $L_a \in \mathbb{R}_{\geq 0}$ such that $\forall s, s'\in \setS, a,a' \in \setA, \mu, \mu' \in \Delta_\setS $,
\small{
\begin{align*}
\| P(\cdot| s, a, \mu) - P(\cdot| s', a', \mu')\| _ 1 \leq &K_\mu \| \mu - \mu'\|_1 + K_s d(s, s') \\
    & + K_a d(a, a'), \\
| R(s, a, \mu) - R(s', a', \mu')| \leq &L_\mu \|\mu - \mu' \|_1 + L_s d(s, s') \\
    & + L_a d(a,a').
\end{align*}}
\end{assumption}
Without loss of generality, we can assume $K_s, K_a \leq 2$ and $L_s, L_a \leq 1$ since it holds that $\| P(\cdot| s, a, \mu) - P(\cdot| s', a', \mu')\| _ 1 \leq 2$ and $| r(s, a, \mu) - r(s', a', \mu')| \leq 1$.

\subsection{Population Update Operators}

One critical goal in the MFG framework is understanding the evolution of the population.
We define an operator to characterize the single-step change of the population.
\begin{definition}[Population update]
The population update operator $\Gamma_{pop}: \Delta_\setS \times \Pi \rightarrow \Delta_\setS$ is defined as
\begin{align*}
    \Gamma_{pop}(\mu, \pi) (s) := \sum_{s'\in\setS} \sum_{a'\in\setA} \mu(s') \pi(a'|s') P(s|s',a', \mu),
\end{align*}
for all $s\in\setS$.
We also introduce the shorthand notation $\Gamma_{pop}^n(\mu, \pi) := \underbrace{\Gamma_{pop}( \ldots \Gamma_{pop}(\Gamma_{pop}(\mu, \pi), \pi), \ldots\pi)}_{\text{$n$ times}}$.
\end{definition}
Re-iterating known results (for instance from \citet{anahtarci2022q} or \citet{guo2019learning}) in our notation and definition of constants, we can show that $\Gamma_{pop}$ is Lipschitz.
\begin{lemma}[Lipschitz population updates]\label{lemma:lipschitz_pop_update}
The population update operator $\Gamma_{pop}$ is Lipschitz with $\|\Gamma_{pop}(\mu, \pi) - \Gamma_{pop}(\mu', \pi')\|_1 \leq L_{pop, \mu} \|\mu - \mu'\|_1 + \frac{K_a}{2} \| \pi - \pi' \|_1,$
where $L_{pop, \mu}:=(\frac{K_s}{2} + \frac{K_a}{2} + K_\mu)$, for all $\pi \in \Pi, \mu \in \Delta_\setS$.
\end{lemma}

In the stationary MFG framework, we would like to compute a unique population distribution that is stable with respect to a policy.
This requires $\Gamma_{pop}(\cdot, \pi)$ to be contractive for all $\pi$.
Hence, we state our second assumption, also implied by assumptions in past work (see \citet{anahtarci2022q, koppel2022oracle, guo2019learning}, also Section~\ref{section:appendix_discussion_assumptions}).
\begin{assumption}[Stable population]\label{assumption:stable_pop}
Population updates are stable, i.e., $L_{pop, \mu} < 1$.
\end{assumption}
We can now formalize the operator that maps policies to their stable distributions, which is well-defined under Assumption~\ref{assumption:stable_pop}.
\begin{definition}[Stable population operator $\Gamma_{pop}^\infty$]
Under Assumption~\ref{assumption:stable_pop}, the stable population operator $\Gamma_{pop}^\infty: \Pi \rightarrow \Delta_\setS$ is defined as the unique population distribution such that $\Gamma_{pop}( \Gamma_{pop}^\infty (\pi), \pi) = \Gamma_{pop}^\infty(\pi)$,
that is, the (unique) fixed point of $\Gamma_{pop}(\cdot, \pi): \Delta_\setS \rightarrow \Delta_\setS$.
\end{definition}
It is straightforward to see that $\Gamma_{pop}^\infty = \lim_{n\rightarrow\infty} \Gamma_{pop}^n$ and that $\Gamma_{pop}^\infty$ is Lipschitz with constant $L_{pop,\infty} := \frac{K_a}{2(1 - L_{pop, \mu})}$, proven in the appendices (Lemma~\ref{lemma:lipschitz_pop_infty}).
The operators $\Gamma_{pop}$, $\Gamma_{pop}^\infty$ are sometimes called population oracles.\looseness=-1

\subsection{Policy Mirror Ascent Operator}

For policy updates, we take a diverging approach from literature and demonstrate that a policy mirror ascent (PMA) operator is Lipschitz, similar to the best response operator.
We define the $Q_h$ and $q_h$ functions for each state-action pair as
$Q_h(s,a|\pi, \mu) := \Exop \big[ \sum_{t=0}^\infty \gamma^t \left(R(s_t, a_t, \mu) + h(\pi(s_t))\right)| 
s_0 = s,
a_0 = a,
s_{t+1} \sim P(\cdot | s_t, a_t, \mu)
a_{t+1} \sim \pi(\cdot|s_{t+1}), \forall t\geq 0
\big]$
and
$q_h(s,a|\pi, \mu) := R(s,a,\mu) + \gamma\sum_{s', a'}  P(s'|s,a,\mu) \pi(a'|s') Q_h(s', a'|\pi, \mu)$.
We also define $V_h(s|\pi,\mu) := \sum_a \pi(a|s) Q_h(s,a|\pi, \mu)$. 
Note that $Q_h(s,a|\pi,\mu) = q_h(s,a|\pi,\mu) + h(\pi(s))$.
With these definitions in place, we analyze the operator that maps population-policy pairs to Q-functions.
This operator is crucial for the online learning algorithm, as it can be approximated purely from trajectories of the current policy.

\begin{definition}[$\Gamma_q$ operator]
We define $\Gamma_q: \Pi \times \Delta_\setS \rightarrow \setQ$ as $\Gamma_{q}(\pi, \mu) = q_h(\cdot, \cdot|\pi, \mu) \in \setQ, \quad \forall \pi\in\Pi, \mu\in\Delta_\setS$.
\end{definition}

As expected, the map $\Gamma_q$ is also Lipschitz continuous.
The Lipschitz properties of $\Gamma_q$ are shown in the appendix, Lemma~\ref{lemma:lipschitz:q}.
Next, we define the policy improvement operator.

\begin{definition}[Policy mirror ascent operator]
\label{definition:pma_operator}
Let $\eta > 0$ and $L_h := L_a + \gamma  \frac{L_s K_a}{2 - \gamma K_s}$.
We define the PMA operator $\Gamma_{\eta}^{md}: \setQ \times \Pi \rightarrow \Pi$ as $\forall s\in\setS, q\in\setQ, \pi\in\Pi$,
\small{
\begin{align*}
     \Gamma_{\eta}^{md} (q, \pi)(s)= \argmax_{u\in\setU_{L_h}} \langle u, q(s, \cdot)\rangle + h(u) - \frac{1}{2\eta} \| u - \pi(s) \|_2^2.
\end{align*}
}
\end{definition}
We establish the Lipschitz continuity of $\Gamma_{\eta}^{md}$ as a result of Fenchel duality and the strong monotonicity of the gradient operator of a strongly convex function, the full proof presented in Lemma~\ref{lemma:lipschitz_md}.

With the building blocks defined above, we now define the main learning operator of interest.
For a learning rate $\eta>0$, we define $\Gamma_\eta:\Pi\rightarrow\Pi$ as
\begin{align*}
    \Gamma_\eta(\pi) := \Gamma_{\eta}^{md}\left(\Gamma_q( \pi, \Gamma_{pop}^\infty(\pi)), \pi\right).
\end{align*}
Intuitively, the operator $\Gamma_\eta$ takes a PMA step with respect to the MDP induced by the stationary distribution $\Gamma_{pop}^\infty(\pi)$.
This operator will be used in the fixed-point iteration process, to find $\pi$ such that $\Gamma_\eta(\pi) = \pi$.
The main justification for this is the following lemma, which demonstrates that the fixed points of $\Gamma_\eta$ are MFG-NE policies for any $\eta>0$.

\begin{lemma}[Fixed points of $\Gamma_\eta$ are MFG-NE]\label{lemma:fixed:points}
Let $\eta>0$ be arbitrary.
A pair $(\pi^*, \mu^*)$ is a MFG-NE if and only if $\pi^* = \Gamma_\eta(\pi^*)$ and $\mu^* = \Gamma^\infty_{pop}(\pi^*)$.
\end{lemma}

Finally, we present the Lipschitz continuity result of $\Gamma_\eta$ and establish the conditions that make it contractive.
\begin{lemma}[Lipschitz continuity of $\Gamma_\eta$]\label{lemma:lipschitz_gamma}
For any $\eta > 0$, the operator $\Gamma_\eta: \Pi \rightarrow \Pi$ is Lipschitz with constant $L_{\Gamma_\eta}$ on $(\Pi, \|\cdot\|_1)$, where
\begin{align*}
    L_{\Gamma_\eta} := \frac{ L_{\Gamma,q} \eta |\setA|}{1 + \rho \eta |\setA| } + \frac{1}{|\setA|^{-1}+ \eta \rho } < \frac{L_{\Gamma,q}}{\rho} + \frac{1}{\eta\rho},
\end{align*}
where $L_{\Gamma,q}$ is a problem-dependent constant.
\end{lemma}
\begin{proof}
See Lemma~\ref{lemma:lipschitz_gamma:full}.
\end{proof}
We point out that $L_{\Gamma,q}$ only depends on $L_\mu,L_a,L_s,K_\mu,K_s,K_a,\gamma$.
While $L_{\Gamma_\eta}$ depends on the dimensionality of the problem (namely $|\setA|$), this dependence will disappear when the learning rate $\eta$ is large enough since at the limit $\eta \rightarrow \infty$ we obtain $L_{\Gamma_\eta} \rightarrow \frac{L_{\Gamma, q}}{\rho}$.
Moreover, for any SAGS, if we take the regularizer to be $\lambda h(\cdot)$, for sufficiently large $\lambda > 0$ it will always be possible to obtain $L_{\Gamma_\eta} < 1$.
We also point out that for the contraction condition $L_{\Gamma_\eta} < 1$ to hold, it must hold that $\rho > L_{\Gamma,q}$.
Conversely, whenever $\rho > L_{\Gamma,q}$ holds, a contraction can be obtained by a sufficiently large learning rate, for instance, $\eta > (\rho - L_{\Gamma,q})^{-1}$.
These results mirror the conditions for contraction developed for the best response operator in \cite{anahtarci2022q} without the need of computing the best response at each stage.

It is worth noting that the mirror ascent operator $\Gamma_\eta^{md}$ only considers the Bregman divergence induced by the squared $\ell_2$ norm.
A natural alternative would be the divergence induced by $h$, leading to closed form solutions for $\Gamma_\eta^{md}$ in certain cases (e.g. when $h$ is the entropy regularizer).
However, it is difficult to establish Lipschitz continuity for general divergences with respect to $\pi$ and we postpone this consideration as future work.
Finally, as expected, computing $\Gamma_\eta^{md}$ in practice will require approximating the solution of a strongly concave maximization problem, which can be solved efficiently with linear convergence.
The computational aspects of $\Gamma_{\eta}^{md}$ are discussed further by \citet{lan2022policy}.

\subsection{Learning in the Exact Policy Mirror Ascent Case}

Concluding this section, we prove the linear convergence to the MFG-NE policy in the exact case, assuming that value functions are known exactly (i.e., we can compute $\Gamma_q$).
\begin{proposition}[Learning MFG-NE, exact case]\label{theorem:deterministic_learning}
Assume that $(\pi^*, \mu^*)$ is the MFG-NE and $L_{\Gamma_\eta} < 1$ for the learning rate $\eta>0$.
Assume $\pi_0 = \pi_{max}$ and consider the updates $\pi_{t+1} = \Gamma_\eta(\pi_t)$ for all $t\geq 0$.
For any $T\geq 1$, we have $\| \pi_T - \pi^*\|_1 \leq L_{\Gamma_\eta}^T \| \pi_0 - \pi^*\|_1\leq 2L_{\Gamma_\eta}^T$.
\end{proposition}
In case we have access to generative-model based samples of the SAGS as in literature (i.e., at any time $t$ we can sample from $P(\cdot|s_t,a_t,\mu)$ and $R(s_t,a_t,\mu)$ for arbitrary $\mu$), Proposition~\ref{theorem:deterministic_learning} readily implies a sample complexity by plugging in a method to estimate value functions and the mean field $\Gamma_{pop}^\infty(\cdot)$.

\begin{remark}
The contraction property of $\Gamma_\eta$ requires the MFG to be sufficiently regularized, as in the case for best-response operators \citep{anahtarci2022q}.
A natural question is if learning of stationary MFG is possible for unregularized games.
While we do not have a universal proof of intractability, we show in the appendices (Section~\ref{sec:impossibility}) that a large class of best response operators employed in literature can not be continuous in $\mu$, unless they are trivial (i.e., the same best-response for all $\mu$).
This result is similar to \cite{cui2021approximately} with a different characterization.
\end{remark}

\begin{remark}
Since the results do not allow $\rho\rightarrow 0$, the regularized MFG-NE will have non-vanishing bias when we are interested in the unregularized NE as discussed in Section~\ref{sec:regularization_and_bias} (see also \cite{geist2019theory}).
The idea of regularization to obtain true NE has been used by \cite{perolat2021poincare} to obtain the unbiased NE of a two player imperfect information game by updating a reference policy.
We leave it as an open question if unregularized stationary MFG-NE can be computed by a similar scheme, for instance using a Bregman divergence regularizer $-D_{-h}(\pi(s)||\pi_{old}(s))$ where $\pi_{old}$ is periodically updated.
\end{remark}

\section{Sample-Based Learning with $N$ Agents}

After establishing the deterministic operator to be computed, we now move to the case where we learn from samples/simulations.
The main goal is to show that $\Gamma_\eta$ can be repeatedly estimated from the simulation of a single path with $N$ agents.
The difficulty will be establishing that the map $\Gamma_q$ can be approximated by taking simulation steps in a single trajectory and that the accumulation of error from past iterations can be controlled.

\paragraph{Definitions.}
We provide the probabilistic setup for single-path learning with $N$-agents.
Let $\{s_0^i\}_{i=1}^N \subset \setS^N$ be arbitrary initial states.
At each timestep $t$, we denote the policy followed by agent $i$ as $\pi_t^i$ for $i=1,\ldots,N$, yielding the transitions $a_{t}^i \sim \pi_t^i(s_{t}^i), s_{t+1}^i \sim P(\cdot|s_{t}^i, a_{t}^i, \widehat{\mu}_t),$ for $i = 1, \ldots, N$, where $\widehat{\mu}_t$ is the $\mathcal{F}(\{s_t^i\}_{i=1}^N)$-measurable empirical state distribution as before.
For any $T>0$, we denote by $\mathcal{F}_{T}$ the sigma algebra $\mathcal{F}_{T} := \mathcal{F}(\{ s_{t}^i, a_t^i, r_t^i\}_{i=1, t=0}^{N,T})$.
We note that any learning algorithm in our context can only specify $\pi_t^i$ at time $t$ for any agent $i$, and only using past observations (i.e., $\pi_t^i$ must be $\mathcal{F}_{t-1}$-measurable).
We define $\setZ := \setS \times \setA \times [0,1] \times \setS \times \setA$, and set $\zeta_t^i$ to be the random transition observations of agent $i$ at time $t$, given by $\zeta_t^i = (s_t^i, a_t^i, r^i_t, s_{t+1}^i, a_{t+1}^i)$ in set $\setZ$.

In general, a mixing assumption is required for online learning on a single trajectory.
The mixing condition can be reduced to a persistence of excitation condition coupled with technical conditions on the probability transition function.
Rather than a blanket assumption, we formalize this as a ``persistence of excitation'' and ``sufficient mixing'' assumption.
The first assumption imposes that the policies throughout the entire training process take each action in each state with probability bounded away from zero. 
\begin{assumption}[Persistence of excitation]\label{assumption:exploration}
Assume there exists $p_{inf}>0$ such that 
\begin{enumerate}
    \item $\pi_{max}(a|s) \geq p_{inf}$ for any $s \in \setS, a\in\setA$, 
    \item For any $\pi \in \Pi, q \in \setQ$ satisfying $\pi(a|s) \geq p_{inf}, 0 \leq q(s,a) \leq Q_{max}, \forall (s,a)\in\setS\times\setA$, it holds that $\Gamma_{\eta}^{md}(q, \pi)(a|s) \geq p_{inf}, \forall (s,a)\in\setS\times\setA$.
\end{enumerate}
\end{assumption}
Assumption~\ref{assumption:exploration} is a property of the policy update operator $\Gamma_{\eta}^{md}$ and therefore of the regularizer $h$.
That is, the regularizer $h$ must ensure that all actions at all states are explored with non-zero probability.

\paragraph{Equivalent conditions on $h$.}
The persistence of excitation (PE) assumption above is in fact achieved for a large class of strongly concave $h$, for instance with entropy regularization.
Sufficient conditions on the gradient $\nabla h$ at the boundary of $\Delta_\setA$ for PE to hold are characterized in Section~\ref{section:conditions_on_h}.

\begin{assumption}[Sufficient mixing]\label{assumption:mixing}
For any $\pi \in \Pi$ satisfying $\pi(a|s) \geq p_{inf} > 0, \forall s\in\setS,a\in\setA$, and any initial states $\{s_0^i\}_i \in \setS^N$,
there exists $T_{mix} > 0, \delta_{mix} > 0$ such that $\Prob(s_{T_{mix}}^j = s' | \{s_0^i\}_i) \geq \delta_{mix}, \quad \forall s' \in \setS, j \in [N]$.
\end{assumption}

Assumptions~\ref{assumption:exploration} and \ref{assumption:mixing} are equivalent to Assumption 2 of \cite{koppel2022oracle}, where a policy is mixed with a uniform policy to achieve PE.
However, since we employ a strongly concave regularizer $h$, we would not typically need to explicitly mix policies, under certain conditions on $h$.
The mixing condition is implied by aperiodicity and irreducability, and have also been explored outside of MFG, for instance by \citet{lan2022policy, tsitsiklis1996analysis}.

\subsection{TD Learning Under Population Dynamics}

One of the main results of this paper is that a variant of TD learning has a $\widetilde{\mathcal{O}}(\varepsilon^{-2})$ sample complexity in SAGS.
Specifically, we extend the results of conditional TD learning (CTD) \citep{kotsalis2022simple}, where the learner waits several time steps between TD learning updates to ensure sufficient convergence to the steady state.
Following the variational inequality approach of \cite{kotsalis2022simple}, we introduce the relevant operators.

\begin{definition}[Operators for CTD]\label{def:operators:td}
Let $\pi \in \Pi$ and $\mu_\pi := \Gamma_{pop}^\infty(\pi)$.
We define the Bellman operator $T^{\pi}:\setQ \rightarrow \setQ$ as $(T^{\pi} Q)(s, a):= R(s, a, \mu_\pi)+h(\pi(s))+\gamma \sum_{s^{\prime}, a'} P\left(s^{\prime} \mid s, a, \mu_\pi\right) \pi\left(a^{\prime} \mid s\right) Q\left(s^{\prime}, a^{\prime}\right)$
for each $Q\in\setQ$. 
We also define the corresponding TD learning operator as $F^{\pi}(Q):=\matM^{\pi} \left(Q-T^\pi Q\right)$,
where $\matM^{\pi} := \operatorname{diag}(\{\mu_\pi(s)\pi(a|s)\}_{s,a}) \in \mathbb{R}^{|\setS||\setA|\times |\setS||\setA|}$ is the state-action distribution matrix induced by $\pi$ at the limiting mean field distribution.
Finally, we define the stochastic TD learning operator $\widetilde{F}^\pi: \setQ \times \setZ \rightarrow \setQ$ with
\begin{align*}
    \widetilde{F}^\pi\left( Q, \zeta \right)= \Big(Q(s,a)-r-h\left(\pi(s)\right)-\gamma Q(s',a')\Big) \vece_{s,a},
\end{align*}
for any $Q\in\setQ, \zeta := (s, a, r, s', a') \in \setZ$.
\end{definition}

To keep the notation clean, we omit the dependence of $T^\pi, F^\pi, \widetilde{F}^\pi$ on the regularizer $h$.
Intuitively, Definition~\ref{def:operators:td} defines $F^{\pi}$ with respect to the \emph{abstract} (i.e., non-observable) MDP induced by the limiting stable distribution $\Gamma_{pop}^\infty(\pi)$; however, at time $t$ we can only observe $\widetilde{F}^\pi\left( Q, \zeta_{t-2}^i\right)$ for $i\in1,\ldots,N$ from simulations with $N$ agents.
The main challenge is establishing that $\widetilde{F}^\pi$ can estimate $F^\pi$ well after waiting for the population and the (population-dependent) Markov chain to mix between each evaluation of $\widetilde{F}^\pi$.
CTD (presented in Algorithm~\ref{alg:ctd}) waits several time steps between each TD step to utilize this mixing.
Note that the algorithm performs TD learning for the first agent ($i=1$). \looseness=-1

\begin{algorithm}
\caption{Conditional TD learning with population}\label{alg:ctd}
\begin{algorithmic}
\REQUIRE $M$, $M_{td}$, learning rates $\{\beta_m\}_m$, policies $\{\pi^i\}_i$, initial states $\{s_0^i\}_i$.
\STATE Set $t\leftarrow 0$ and $\widehat{Q}_0(s,a) = Q_{max}, \forall s\in\setS,a\in\setA$
\FOR{$m \in 0, 1, \ldots, M-1$}
    \FOR{$M_{td}$ iterations}
        \STATE Take simulation step ($\forall i$): $a_t^i \sim \pi^i(s_t^i)$, $r_t^i = R(s_{t}^i, a_{t}^i, \widehat{\mu}_t)$, $s_{t+1}^i \sim P(\cdot|s_{t}^i, a_{t}^i,\widehat{\mu}_t)$.
        \STATE $t \leftarrow t+1$
    \ENDFOR
    \STATE Set $\zeta_{t-2}^1 = (s_{t-2}^1, a_{t-2}^1, r_{t-2}^1, s_{t-1}^1, a_{t-1}^1)$.
    \STATE CTD Update: $\widehat{Q}_{m+1} = \widehat{Q}_{m} - \beta_m \widetilde{F}^{\pi^1}\left( \widehat{Q}_{m}, \zeta_{t-2}^1\right).$
\ENDFOR
\STATE Return $\widehat{Q}_{M}$.
\end{algorithmic}
\end{algorithm}
It is known that $T^\pi$ is contractive with $\gamma$ and $F^\pi$ is Lipschitz with $L_F := (1+\gamma)$ with respect to the $\| \cdot\|_2$ norm on $\setQ$ for any $\pi\in\Pi$. 
$F^\pi$ is also generalized strongly monotone \citep{kotsalis2022simple} with modulus $\mu_F := (1-\gamma)\delta_{mix} p_{inf}$, that is, $\forall Q \in \setQ$, it holds that
\begin{align*}
    \langle F^\pi(Q), Q - Q_h(\cdot, \cdot|\pi,\mu_\pi) \rangle \geq \mu_F \|Q - Q_h(\cdot, \cdot|\pi,\mu_\pi) \|_2^2,
\end{align*}
where $Q_h(\cdot, \cdot|\pi,\mu_\pi)$ is the true value function of the policy $\pi$ at the mean-field $\mu_\pi = \Gamma_{pop}^\infty(\pi)$ as defined in Section~\ref{sec:operators}.

Our analysis consists of two steps: (1) we prove that the CTD algorithm of \cite{kotsalis2022simple} can be used when the samples are biased and state the explicit bound (see Theorem~\ref{theorem:ctd_opt} in the appendices), and (2) quantify the bias and mixing rate to the limiting distribution in $N$-player SAGS (in Section~\ref{section:mixing_bounds}) to prove the main result, Theorem~\ref{theorem:ctd_with_pop}.

\begin{theorem}[CTD learning with population]\label{theorem:ctd_with_pop}
Assume Assumption \ref{assumption:mixing} holds and let policies $\{\pi^i\}_i$ be given so that $\pi^i(a|s)\geq p_{inf}$ for all $i$.
Assume Algorithm~\ref{alg:ctd} is run with policies $\{\pi^i\}_i$, \emph{arbitrary} initial agent states $\{s_0^i\}_i$, learning rates $\beta_m=\frac{2}{(1 - \gamma)\left(t_0+m-1\right)}, \forall m \geq 0$ and $M>\mathcal{O}(\varepsilon^{-2})$, $M_{td} > \mathcal{O}(\log \varepsilon^{-1})$.
If $\bar{\pi} \in \Pi$ is an arbitrary policy, $\Delta_{\bar{\pi}}:= \frac{1}{N}\sum_{i} \| \pi^i - \bar{\pi}\|_1$ and $Q^* := Q_h(\cdot,\cdot|\bar{\pi},\mu_{\bar{\pi}})$, then the (random) output $\widehat{Q}_M$ of Algorithm~\ref{alg:ctd} satisfies
\begin{align*}
    \Exop[ \|\widehat{Q}_{M} - Q^* \|_\infty] \leq &
    \varepsilon + \mathcal{O}\left(\frac{1}{\sqrt{N}} + \Delta_{\bar{\pi}} + \| \pi^1 - \bar{\pi}\|_1\right).
\end{align*} 
\end{theorem}
\begin{proof}
See Theorem~\ref{theorem:ctd_with_pop:full}.
\end{proof}

If the policies of all agents are equal (i.e., $\pi^i = \bar{\pi}$ for all $i$), then Theorem~\ref{theorem:ctd_with_pop} suggests an expected $\ell_\infty$ accuracy of $\varepsilon > 0$ can be achieved with respect to the ``limiting mean field'' $\Gamma_{pop}^\infty(\bar{\pi})$ in $\mathcal{O}(\varepsilon^{-2} \log \varepsilon^{-1})$ time steps, up to a finite population bias $\mathcal{O}\left(\frac{1}{\sqrt{N}}\right)$ as expected.

\subsection{Main Results: Learning MFG-NE from Samples}

In this section, we analyze two algorithms to estimate MFG-NE from a single sample path using $N$ agents.
In the first (Algorithm~\ref{alg:centralized}), the policies are synchronized  by a centralized controller, allowing each agent to follow the same policy yielding a superior sampling complexity.
In the second (Algorithm~\ref{alg:decentralized}), learning is performed completely independently by each agent utilizing their local observations only.
The algorithms are based on repeating the following iteration:
\begin{enumerate}
    \item \textbf{Estimate values:} Each agent, keeping their policies constant, performs CTD learning for $M_{pg} > 0$ steps, each time waiting $M_{td}$ steps between TD updates,
    \item \textbf{Policy update:} Afterwards, agents simultaneously perform \emph{one} policy mirror ascent update using $Q$-value estimates (the same update in the centralized case).
\end{enumerate}
Both algorithms use samples from the SAGS in the form $s_{t+1}^i \sim P(\cdot|s_{t}^i, a_{t}^i,\widehat{\mu}_t)$, and $r_t^i = R(s_{t}^i, a_{t}^i, \widehat{\mu}_t)$, where $\widehat{\mu}_t$ is the \emph{empirical} state distribution of the $N$ agents at time $t$ which the algorithm can not control, hence they are single-path and generative-model-free algorithms.

\begin{algorithm}[tbh]%
\caption{Centralized MFG-PMA}\label{alg:centralized}
\begin{algorithmic}
\REQUIRE parameters $K, M_{td}, M_{pg}, \eta, \{\beta_m\}_m$.
\REQUIRE initial states $\{s_0^i\}_i$.
\STATE Set $\pi_0 = \pi_{max}$ and $t\leftarrow 0$.
\FOR{$k \in 0, \ldots, K-1$}
    \STATE $\forall s,a: \widehat{Q}_0(s,a) = Q_{max}$
    \FOR{$m \in 0, \ldots, M_{pg}-1$}
        \FOR{$M_{td}$ iterations}
            \STATE Take step $\forall i$: $a_t^i \sim \pi_k(s_t^i)$, $r_t^i = R(s_{t}^i, a_{t}^i, \widehat{\mu}_t)$, $s_{t+1}^i \sim P(\cdot|s_{t}^i, a_{t}^i,\widehat{\mu}_t)$.
            \STATE $t \leftarrow t+1$
        \ENDFOR
        \STATE Set $\zeta_{t-2}^1 = (s_{t-2}^1, a_{t-2}^1, r_{t-2}^1, s_{t-1}^1, a_{t-1}^1)$.
        \STATE CTD step: $\widehat{Q}_{m+1} = \widehat{Q}_{m} - \beta_{m} \widetilde{F}^{\pi_k}\left( \widehat{Q}_{m}, \zeta_{t-2}^1\right)$
    \ENDFOR
    \STATE PMA step: $\pi_{k+1} = \Gamma_{\eta}^{md}(\widehat{Q}_{M_{pg}}, \pi_{k})$
\ENDFOR
\STATE Return policy $\pi_{K}$
\end{algorithmic}
\end{algorithm}

\begin{algorithm}[tbh]%
\caption{Independent MFG-PMA}\label{alg:decentralized}
\begin{algorithmic}
\REQUIRE parameters $K, M_{td}, M_{pg}, \eta, \{\beta_m\}_m$.
\REQUIRE initial states $\{s_0^i\}_i$.
\STATE Set $\pi^i_0 = \pi_{max}, \forall i$ and $t\leftarrow 0$.
\FOR{$k \in 0, \ldots, K-1$}
    \STATE $\forall s,a,i: \widehat{Q}^i_0(s,a) = Q_{max}$
    \FOR{$m \in 0, \ldots, M_{pg}-1$}
        \FOR{$M_{td}$ iterations}
            \STATE Take step $\forall i$: $a_t^i \sim \pi^i_k(s_t^i)$, $r_t^i = R(s_{t}^i, a_{t}^i, \widehat{\mu}_t)$, $s_{t+1}^i \sim P(\cdot|s_{t}^i, a_{t}^i,\widehat{\mu}_t)$.
            \STATE $t \leftarrow t+1$
        \ENDFOR
        \STATE Set $\forall i: \zeta_{t-2}^i = (s_{t-2}^i, a_{t-2}^i, r_{t-2}^i, s_{t-1}^i, a_{t-1}^i)$.
        \STATE CTD step $\forall i$: $\widehat{Q}^i_{m+1} = \widehat{Q}^i_{m} - \beta_{m} \widetilde{F}^{\pi^i_k}\left( \widehat{Q}^i_{m}, \zeta_{t-2}^i\right)$
    \ENDFOR
    \STATE PMA step $\forall i: \pi^i_{k+1} = \Gamma_{\eta}^{md}(\widehat{Q}^i_{M_{pg}}, \pi^i_{k})$
\ENDFOR
\STATE Return policies $\{\pi^i_{K}\}_i$
\end{algorithmic}
\end{algorithm}

We first present the main result for centralized learning.
The centralized algorithm uses samples from the path of a single agent among $N$ (for the first agent $i=1$) and updates the policies of all agents accordingly. 
For simplicity, algorithms are presented in a form that requires the number of iterations to be set beforehand but can be trivially converted into algorithms that take the sample budget as input.\looseness=-1

\begin{theorem}[Centralized learning]\label{theorem:convergence_centralized}
Assume that $\eta > 0$ an arbitrary learning rate which satisfies $L_{\Gamma_\eta} < 1$,  
Assumptions~\ref{assumption:lipschitz}, \ref{assumption:stable_pop}, \ref{assumption:exploration} and \ref{assumption:mixing} hold and $\pi^*$ is the unique MFG-NE.
Let $\varepsilon > 0$ be arbitrary.
If the learning rates $\{\beta_m\}$ are as defined in Lemma~\ref{theorem:ctd_with_pop} and $K > \mathcal{O}(\log \varepsilon^{-1}), M_{td} > \mathcal{O}(\log \varepsilon^{-1}), M_{pg} > \mathcal{O}(\varepsilon^{-2})$ then the (random) output $\pi_K$ of Algorithm~\ref{alg:centralized} satisfies $\Exop\left[ \| \pi_K - \pi^*\|_1 \right] \leq \varepsilon + \mathcal{O}\left(\frac{1}{\sqrt{N}}\right).$
\end{theorem}
\begin{proof}
\emph{(sketch)}
The proof builds up on Theorem~\ref{theorem:ctd_with_pop} to build estimates of the true Q-values $Q(\cdot, \cdot|\pi_k, \Gamma_{pop}^\infty(\pi_k))$ at each outer loop.
The estimates allow the approximate evaluation of the $\Gamma_\eta$ operator which was shown to contract to the MFG-NE.
See Theorem~\ref{theorem:convergence_centralized:full} for the full statement and proof.
\end{proof}

\begin{remark}
In Algorithm~\ref{alg:centralized}, the total number of time steps simulated from the SAGS is equal to $K \times M_{pg} \times M_{td}$; hence, Theorem~\ref{theorem:convergence_centralized} implies a $\mathcal{O}(\varepsilon^{-2} \log^2 \varepsilon^{-1})$ sample complexity.
In fact, since MFG-NE are $\mathcal{O}\left( \frac{1}{\sqrt{N}}\right)$-NE of the $N$-agent SAGS, Theorem~\ref{theorem:convergence_centralized} shows that a Algorithm~\ref{alg:centralized} finds a $\varepsilon + \mathcal{O}\left( \frac{1}{\sqrt{N}}\right)$-NE of the SAGS in $\widetilde{\mathcal{O}}(\varepsilon^{-2})$ time steps.
\end{remark}

Finally, we present the case where each agent learns separate policies only by observing their own state-action transitions, presented in Algorithm~\ref{alg:decentralized}.
It only uses the observations of agent $i$ to update policy $\pi_t^i$, hence learning is independent.
Moreover, the algorithm does not employ multiple timescales and it has the desirable \emph{symmetry} property that each learner follows the same protocol, meaning no handshake procedure between learners is required before learning starts (other than agreement on the number of epochs, $K$).
The proof utilizes the bounds in Theorem~\ref{theorem:ctd_with_pop} in terms of policy deviations $\Delta_{\bar{\pi}}$ of agents.

\begin{theorem}[Independent learning]\label{theorem:convergence_decentralized}
Assume that $\eta > 0$ satisfies $L_{\Gamma_\eta} < 1$, 
Assumptions~\ref{assumption:lipschitz}, \ref{assumption:stable_pop}, \ref{assumption:exploration} and \ref{assumption:mixing} hold and $\pi^*$ is the unique MFG-NE.
Let $\varepsilon > 0$ be arbitrary.
Let the learning rates $\{\beta_m\}$ for CTD be as defined in Lemma~\ref{theorem:ctd_with_pop}.
There exists a problem dependent constant $a \in [0,\infty)$ such that if $K = \lceil \frac{\log 8\varepsilon^{-1}}{\log L_{\Gamma_\eta}^{-1}} \rceil$, $M_{td} > \mathcal{O}(\log^2 \varepsilon^{-1})$, and $M_{pg} > \mathcal{O}(\varepsilon^{-2-a})$,
then the (random) output $\{\pi_K^i\}_i$ of Algorithm~\ref{alg:decentralized} satisfies for all agents $i=1,\ldots,N$, $\Exop\left[ \| \pi_K^i - \pi^*\|_1 \right] \leq \varepsilon + \mathcal{O}\left(\frac{1}{\sqrt{N}}\right)$.
\end{theorem}
\begin{proof}
\emph{(sketch)}
The proof in the independent learning case also builds up on Theorem~\ref{theorem:ctd_with_pop}, however, as the policies of agents diverge due to stochasticity, the terms $\Delta_{\bar{\pi}}$ become significant as they do not vanish.
Therefore, the choice of $K$ and consequently $M_{pg}$ ensure the final error can be made to be less than $\varepsilon$ up to finite population bias, which disappears as $N\rightarrow\infty$.
See Theorem~\ref{theorem:convergence_decentralized:full} for the full statement and proof.
\end{proof}

\begin{remark}
The constant $a$ in Theorem~\ref{theorem:convergence_decentralized} is $0$ for smooth enough problems.
Hence the theorem implies a sample complexity of $\widetilde{\mathcal{O}}(\varepsilon^{-2})$ for SAGS with sufficiently smooth dynamics (e.g., large enough $\rho$, small enough constants $K_a, L_\mu$), even though these smoothness conditions might be too restrictive.
In the general case however, the sampling complexity will be $\widetilde{\mathcal{O}}(\varepsilon^{-2-a})$, where $a$ is affected by the learning rate $\eta$ and has a logarithmic dependency in general on the parameters $T_{mix}, \delta_{mix}^{-1}$, and $p_{inf}^{-1}$.
\end{remark}

Concluding this section, we comment on the differences in sample complexity upper bounds between the centralized and independent learning cases.
In the independent learning case, the worse $\widetilde{\mathcal{O}}(\varepsilon^{-2-a})$ sample complexity arises due to the fact that the MFG-NE formalism requires the policies of all $N$ agents to be \emph{exactly} the same.
As variances in the agents' updates accumulate a non-vanishing bias in the later stages of the outer loop (due to the terms $\Delta_{\bar{\pi}}$), the CTD learning loop must be executed for a long time to keep the variance sufficiently low.
It is possible that this is an artifact of the analysis technique, and we leave possible improvements as future work.

\section{Discussion}
We prove that policy mirror ascent can achieve an improved sample complexity for computing the approximate Nash equilibrium of an $N$-player SAGS.
Our approach is complementary to existing MFG literature in that we analyze mirror ascent type operators allowing us to remove strong generative model assumptions to obtain an algorithm for the truly $N$-player game.
Future work could involve a finer (e.g., potential-based) analysis using results in policy mirror descent literature.\looseness=-1

\section*{Acknowledgements}

This project has been carried out with funding from the Swiss National Science Foundation under the framework of NCCR Automation.

\bibliography{references}
\bibliographystyle{icml2023}

\newpage
\appendix
\onecolumn

\section{Clarification on Notation and Constants}

The topological interior and boundary of a set $\setX \subset \mathbb{R}^D$ are denoted by $\setX^\circ$ and $\partial\setX$ respectively.
We define the point-set distance $d(x, \setX) := \inf_{x'\in\setX} \| x - x'\|_2$.
$\setX^{m\times n}$ denotes the set of $m \times n$ matrices with entries from the set $\setX$.
For $\matA \in \setX^{m\times n}$, the $i$-th row vector is denoted as $\matA_{i\cdot}$ and the $j$-th column vector as $\matA_{\cdot j}$.
Stochastic matrices are defined as matrices with positive entries and with row sums equal to 1.
$\vece_x \in \Delta_\setX$ for $x\in\setX$ is the vector with only the entry corresponding to $x$ set to 1.
$\|\cdot\|_*$ denotes the dual of a norm $\|\cdot\|$, given by the standard definition
\begin{align*}
    \| \vecx \|_* := \sup_{} \{ \vecy ^\top \vecx : \vecy \in \mathbb{R}^D, \| \vecy \| \leq 1 \}, \forall \vecx \in \mathbb{R}^D.
\end{align*}

\paragraph{A note on notation for functions on $\setA$.}
In certain cases, it will be useful to alternate between (equivalent) definitions of certain functions on $\Delta_\setA$ and $\setA$ \citep[see][]{anahtarci2022q}.
For instance, one can define state transition probabilities as functions $P: \setS \times \setA\times \Delta_\setS \rightarrow \Delta_\setS$ taking state-action pairs as input, or equivalently, as a function $\bar{P}:\setS \times \Delta_\setA \times \Delta_\setS \rightarrow \Delta_\setS$ taking state-probability distribution of actions with the relation $\bar{P}(s'|s, u,\mu) := \sum_{a\in\setA} u(a) P(s'|s, a,\mu)$.
To ease reading we follow the convention that for a function on $f$ on $\setA$, we denote the equivalent function on $\Delta_\setA$ with $\bar{f}$.

\paragraph{Notation for constants.}
Our analysis will involve multiple Lipschitz constants.
For simplicity and readability, we use the convention that for a function of multiple variables say $f(x,y, \mu)$, the Lipschitz constant of $f$ with respect to $x,y,\mu$ are denoted by $L_{f,x}, L_{f,y}, L_{f,\mu}$ respectively.

\section{Basic Lemmas}

In this section, we present several general lemmas re-used throughout the paper.

\begin{lemma}[p. 141, \cite{georgii2011gibbs}]\label{lemma:exptv_inequality}
Assume $E$ a finite set, $F: E \rightarrow \mathbb{R}$ a real valued function, and $\nu, \mu$ two probability measures on $E$.
Then, 
\begin{align*}
    \left|\sum_{e} F(e) \mu(e)-\sum_{e} F(e) \nu(e)\right| \leq \frac{\sup_e F(e) - \inf_e F(e)}{2}\|\mu-\nu\|_{1}.
\end{align*}
\end{lemma}

We provide two generalizations of this lemma, the first one adapted from  \citet[Lemma~A2]{kontorovich2008concentration}.

\begin{lemma}
\label{lemma:expvector_inequality}
Assume $E$ a finite set, $g: E \rightarrow \mathbb{R}^p$ a vector value function, and $\nu, \mu$ two probability measures on $E$.
Then, 
\begin{align*}
    \left\|\sum_{e} g(e) \mu(e)-\sum_{e} g(e) \nu(e)\right\|_1 \leq \frac{\lambda_g}{2} \|\mu-\nu\|_{1},
\end{align*}
where $\lambda_g := \sup_{e, e'} \|g(e) - g(e')\|_1$.
\end{lemma}
\begin{proof}
We use the fact that $\|u\|_1 = \sup_{\|v\|_\infty \leq 1} u^\top v$ by duality of the norms.
Applying this on the vector $\sum_{e} g(e) \left(\mu(e)-\nu(e)\right)$, we obtain
\begin{align*}
   \left\| \sum_{e} g(e) \left(\mu(e)-\nu(e)\right) \right\|_1 = \sup_{\|v\|_\infty \leq 1} \sum_{e}  \left(\mu(e)-\nu(e)\right) g(e)^\top v.
\end{align*}

Take any $v \in \mathbb{R}^d$ such that $\|v\|_\infty \leq 1$ and apply Lemma~\ref{lemma:exptv_inequality},
\begin{align}
    \left|\sum_{e}  \left(\mu(e)-\nu(e)\right) g(e)^\top v \right| \leq & \| \mu - \nu \|_1 \frac{\sup_e g(e)^\top v - \inf_e g(e)^\top v }{2} \label{eq:expvecineqv} \\
    \leq & \| \mu - \nu \|_1 \frac{\sup_{e, e'} \|g(e') - g(e) \|_1}{2}
\end{align}
where the last line follows from the fact that $E$ is finite and for any $e, e'\in E$,
\begin{align*}
    |g(e)^\top v - g(e')^\top v | \leq  \|g(e') - g(e) \|_1 \|v\|_\infty = \|g(e') - g(e) \|_1,
\end{align*}
by an application of Hölder's inequality.
Then taking the supremum of the left hand side in (\ref{eq:expvecineqv}), the result follows.
\end{proof}
We further generalize this lemma below.
\begin{lemma}\label{lemma:stochastic_matrices_l1}
Assume $\vecv_1, \vecv_2 \in \Delta_D$ and $\matA^{(1)}, \ldots, \matA^{(K)}, \matB^{(1)}, \ldots, \matB^{(K)} \in \mathbb{R}_{\geq 0}^{D\times D}$ are stochastic matrices.
Then, it holds that
\begin{align*}
    \| \vecv_1^\top \matA^{(1)} \matA^{(2)} \ldots \matA^{(K)}  - \vecv_2^\top \matB^{(1)} \matB^{(2)} \ldots \matB^{(K)} \|_1 \leq \| \vecv_1 - \vecv_2\|_1 + \sum_{k=1}^K \sup_j \|\matA_{ j \cdot}^{(k)} - \matB_{ j \cdot}^{(k)}\|_1.
\end{align*}
\end{lemma}
\begin{proof}
We prove by induction on $k$.
For $k=1$, we have
\begin{align*}
    \|\vecv_1^\top \matA^{(1)}  - \vecv_2^\top \matB^{(1)} \|_1 &\leq \| ( \vecv_1 - \vecv_2)^\top \matA^{(1)} \|_1 + \|\vecv_2^\top (\matA^{(1)} - \matB^{(1)})  \|_1 \\
    \overset{\text{Lemma~\ref{lemma:expvector_inequality}}}&{\leq} \| \vecv_1 - \vecv_2\|_1 + \|\vecv_2 ^\top (\matA^{(1)} - \matB^{(1)}) \|_1 \\
    \overset{\text{Jensen's}}&{\leq} \| \vecv_1 - \vecv_2\|_1 + \sum_{j=1}^n \|\matA^{(1)}_{j \cdot } - \matB^{(1)}_{j \cdot }\|_1 \vecv_2(j)\\
    &\leq \| \vecv_1 - \vecv_2\|_1 + \sup_j \|\matA^{(1)}_{j \cdot } - \matB^{(1)}_{j \cdot }\|_1.
\end{align*}
Assuming it holds for some $K \geq 1$, the statement also holds for $K+1$ since
\begin{align*}
\| \vecv_1^\top \matA^{(1)} \matA^{(2)} \ldots \matA^{(K+1)}  - \vecv_2^\top \matB^{(1)} \matB^{(2)} \ldots \matB^{(K+1)} \|_1 \leq &\| \vecv_1^\top \matA^{(1)} \ldots \matA^{(K)}  - \vecv_2^\top \matB^{(1)} \ldots \matB^{(K)} \|_1 \\
    &+ \sup_j \|\matA^{(K+1)}_{ j \cdot} - \matB^{(K+1)}_{j \cdot}\|_1.
\end{align*}
Hence the general statement follows.
\end{proof}

Finally, we use the widely known Fenchel duality to establish certain Lipschitz regularity results.
We re-iterate the following definition of strong convexity and strong concavity with respect to arbitrary norms.
\begin{definition}[Strong convexity]\label{def:strong_convexity}
Let $f:\mathbb{R}^d \rightarrow \mathbb{R} \cup \{ \infty \}$ be a convex differentiable function with domain $S := \{x \in \mathbb{R}^d: f(x) \in \mathbb{R}\}$ and let $\| \cdot \| : \mathbb{R}^d \rightarrow \mathbb{R}_{\geq 0}$ be an arbitrary norm on $\mathbb{R}^d$.
If for any $x,y \in S$ it holds that
\begin{align*}
    f(y) \geq f(x)+\langle\nabla f(x), y-x\rangle+\frac{1}{2} \rho\|y-x\|^2,
\end{align*}
then $f$ is called a strongly convex function with modulus $\rho$ with respect to norm $\|\cdot\|$.
\end{definition}
As expected, if $-f$ is a strongly convex function, $f$ is called a strongly concave function with respect to norm $\|\cdot\|$.
We list standard properties of strong convexity, which hold with respect to arbitrary norms.
If $f_1, f_2$ are $\rho_1, \rho_2$ strongly convex with respect to $\|\cdot\|$ and $\alpha_1, \alpha_2 > 0$, then $\alpha_1 f_1+ \alpha_2 f_2$ is $\alpha_1 \rho_1 + \alpha_2 \rho_2$ strongly convex with respect to $\|\cdot\|$.
If $\vertiii{\cdot}$ and $\| \cdot \|$ are equivalent norms so that $\vertiii{\cdot} \geq c \| \cdot \|$ for some constant $c$, and if $f$ is $\rho$-strongly convex with respect to $\vertiii{\cdot}$, then it is $c^2 \rho$-strongly convex with respect to $\| \cdot \|$.
Finally, $\rho$-strong convexity of $f$ with respect to $\| \cdot\|$ implies that for all $x_1, x_2$ in the domain of $f$:
\begin{align*}
    \left[ \nabla f(x_1) - \nabla f(x_2)\right]^\top (x_1 - x_2) \geq \rho \| x_1 - x_2\|^2.
\end{align*}

We will also need the following standard concept of a Fenchel conjugate.
\begin{definition}[Fenchel conjugate]
Assume that $f: \mathbb{R}^d \rightarrow \mathbb{R}\cup \{\infty\}$ is a convex function, with domain $S \subset \mathbb{R}^d$.
The Fenchel conjugate $f^*: \mathbb{R}^d \rightarrow \mathbb{R}\cup \{\infty\}$ is defined as 
\begin{align*}
    f^*(y) = \sup_{x \in S} \langle x, y\rangle - f(x).
\end{align*}
\end{definition}
For further details regarding the Fenchel conjugate, see \citet{nesterov2018lectures}.
The Fenchel conjugate is useful due to the following well-known duality result.
\begin{lemma}\label{lemma:fenchel_duality}
Assume that $f: \mathbb{R}^d \rightarrow \mathbb{R}\cup \{\infty\}$ is differentiable and $\tau$-strongly convex with respect to a norm $\|\cdot\|$ and has domain $S \subset \mathbb{R}^d$.
Then,
\begin{enumerate}
    \item $f^*$ is differentiable on $\mathbb{R}^d$,
    \item $\grad f^*(y) = \argmax_{x\in S} \langle x,y\rangle - f(x)$,
    \item $f^*$ is $\frac{1}{\tau}$-smooth with respect to $\|\cdot\|_*$, i.e., $\|\grad f^*(y) - \grad f^*(y')\| \leq \frac{1}{\tau} \| y - y'\|_*, \forall y,y' \in \mathbb{R}^d$.
\end{enumerate}
\end{lemma}
\begin{proof}
See Lemma 15 of \cite{shalev2007online} or Lemma 6.1.2 of \cite{nesterov2018lectures}.
\end{proof}
\section{Operators and Lipschitz Continuity}

In this section, we provide the proofs regarding the Lipschitz continuity of  operators mentioned in the main text.
Some of these results are organized and generalized from \cite{anahtarci2022q} to our setting, while others relevant to mirror ascent are unique to our case.

\subsection{Discussion of Assumptions for Contractivity}\label{section:appendix_discussion_assumptions}

The construction or assumption of an operator that contracts to the MFG-NE is a common strategy in stationary MFG literature.
Before we move on to presenting our results and proofs for obtaining a contractive operator, we compare our assumptions with several other relevant works in literature.
We present Table~\ref{table:assumptions_contractivity} as a summary.

\begin{table*}[h]
\centering
\begin{tabular}{|p{4cm} | p{6.5cm} | p{3.5cm}|} 
\hline 
\textbf{Contractivity assumption} & \textbf{Additional/implied assumption on $\Gamma_{pop}$} & \textbf{Works}\\
\hline 
\hline
Blanket contractivity & $\Gamma_{pop}(\cdot, \pi)$ contraction (Assumption~\ref{assumption:stable_pop}) & \citet{guo2019learning}, \citet{guo2022general}, \citet{koppel2022oracle}, \citet{xie2021learning}\\[3ex]
\hline
Lipschitz continuous $P,R$ (Assumption~\ref{assumption:lipschitz}) + regularization & Access to mean-field oracle $\Gamma_{pop}^\infty$ + Lipschitz continuous $\Gamma_{pop}^\infty$ & \citet{cui2021approximately}\\[3ex]
\hline
Lipschitz continuous $P,R$ (Assumption~\ref{assumption:lipschitz}) + regularization & $\Gamma_{pop}(\cdot, \pi)$ contraction (Assumption~\ref{assumption:stable_pop})  & \textbf{Our work}, \citet{anahtarci2022q}\\[3ex]
\hline 
\end{tabular}
\caption{A summary of existing assumptions in stationary MFGs in literature for obtaining a contraction.
We note that the works by \citet{xie2021learning} and \citet{koppel2022oracle} refer to past works for justification of the contraction assumption.
}
\label{table:assumptions_contractivity}
\end{table*}

\paragraph{A note on assumptions on $\Gamma_{pop}$.}
The assumption on the population distribution $
\Gamma_{pop}(\cdot, \pi)$ being a contraction for all $\pi$ is implicit in most past works as shown in Table~\ref{table:assumptions_contractivity}.
For instance the assumption in Theorem~4 of in \citet{guo2019learning} and Theorem~4 of \citet{guo2022general} that $d_1 d_2 + d_3 < 1$ for positive Lipschitz constants $d_1, d_2, d_3$ implies $d_3 < 1$, where $d_3$ satisfies (in the notation of the mentioned paper)
\begin{align*}
    \mathcal{W}_1\left(\Gamma_2\left(\pi, \mathcal{L}^1\right), \Gamma_2\left(\pi, \mathcal{L}^2\right)\right) \leq d_3 \mathcal{W}_1\left(\mathcal{L}^1, \mathcal{L}^2\right),
\end{align*}
where $\Gamma_2$ is equivalent to $\Gamma_{pop}$ in our notation, $\mathcal{L}^1, \mathcal{L}^2 \in \Delta_\setS$ and $\mathcal{W}_1$ is the $\ell_1$ Wasserstein distance.
Likewise, Assumption~2 of \citet{xie2021learning} imposes that there exists $d_3$ satisfying $\left\|\Gamma_2(\pi, \mu)-\Gamma_2\left(\pi, \mu^{\prime}\right)\right\|_{\mathcal{H}} \leq d_3\left\|\mu-\mu^{\prime}\right\|_{\mathcal{H}}$ for all $\pi, \mu$ with the additional restriction $d_1 d_2 + d_3 < 1$, where $\| \cdot \|_\mathcal{H}$ is norm associated with the RHKS of a bounded and universal kernel $k$.
Similarly, Assumption~1 of \citet{koppel2022oracle} states that (once again, $\Gamma_2$ being equivalent to $\Gamma_{pop}$ in our notation):
\begin{align*}
    \left\|\Gamma_2(\pi, \mu)-\Gamma_2\left(\pi, \mu^{\prime}\right)\right\|_1 \leq d_3\left\|\mu-\mu^{\prime}\right\|_1,
\end{align*}
with the additional restriction $d_1 d_2 + d_3 < 1$ for positive $d_1, d_2, d_3$ also implying $d_3 < 1$.
In the case of \citet{anahtarci2022q}, it is implied that $\Gamma_{pop}(\cdot, \pi)$ is a contraction for all $\pi \in \Pi' \subset \Pi$, where $\Pi'$ is a set of possible best response policies which satisfy a smoothness condition as their Lemma~3 suggests.
This is not immediate since they do not define $\Gamma_{pop}$ explicitly, but a careful analysis of their proof of Proposition~3, their assumption that $K_H < 1$ and their population update operator $H_2: \Delta_\setS \times \setQ \rightarrow \Delta_\setS$ defined on value functions rather than policies yields this result.
Finally, \citet{cui2021approximately} assume instead access to the mean-field population oracle $\Gamma_{pop}^\infty$ (in the notation of their paper, $\Psi$).
This is in general different from assuming $\Gamma_{pop}(\cdot, \pi)$ is contractive for all $\pi$, as this instead implies that the fixed point equation $\Gamma_{pop}(\mu, \pi) = \mu$ can be solved for $\mu$ with a solution Lipschitz continuous in $\pi$.

In fact, we point out that it is not granted (to the best of our knowledge) that the infinite agent game is a good approximation for the $N$-agent game when $\Gamma_{pop}(\cdot, \pi)$ is not a contraction in some metric space for all $\pi$, for instance, Theorem~1 of \citet{anahtarci2022q} (summarized in Proposition~\ref{proposition:finite_N} in our paper) requires a similar contraction constraint (in the notation of the paper) that $C_1 < 1$.
This might be intuitive: if $\Gamma_{pop}(\cdot, \pi)$ is not a contraction, in general, the finite population bias might be amplified between time steps of the $N$-agent SAGS.
This would make the stationarity condition of the population in the definition of MFG-NE irrelevant in the case of $N$ agents and when there is stochasticity in the transition dynamics: the empirical distributions $\widehat{\mu}_t$ will not be close to $\mu^*$ for the MFG-NE population distribution $\mu^*$.
We leave it as an interesting question for future work if the finite agent SAGS is well approximated by the MFG without explicitly assuming contraction of $\Gamma_{pop}(\cdot, \pi)$ (for instance, when $\Gamma_{pop}(\cdot, \pi)$ is only non-expansive).

\subsection{Lipschitz Continuity in $(\Delta_\setA, \|\cdot\|_1)$ of $P, R$}

Firstly, we define 
$\bar{R}: \setS \times \Delta_\setA \times \Delta_\setS \rightarrow [0,1]$ and $\bar{P}: \setS \times \Delta_\setA \times \Delta_\setS \rightarrow \Delta_\setS$
as the rewards and action probabilities on probability distributions over actions as
$\bar{R}(s, u, \mu) := \sum_{a\in\setA} u(a) R(s, a, \mu)$ and
$\bar{P}(\cdot|s,u,\mu) := \sum_{a\in\setA} u(a) P(\cdot|s,a,\mu)$ for all $s\in\setS, u\in \Delta_\setA, \mu \in \Delta_\setS$.
These alternative definitions will be practical when establishing certain Lipschitz identities later.
As expected, $\bar{P}, \bar{R}$ are Lipschitz in their arguments.

\begin{lemma}[Lipschitz continuity of $\bar{P}, \bar{R}$]
For all $s,s'\in\setS, u,u'\in\Delta_\setA, \mu, \mu' \in \Delta_\setS$,
\begin{align*}
    |\bar{R}(s,u, \mu) - \bar{R}(s',u', \mu')| \leq & L_\mu\|\mu - \mu'\|_1 + L_s d(s, s') + \frac{L_a}{2} \|u - u'\|_1  \\
    \|\bar{P}(\cdot|s,u, \mu) - \bar{P}(\cdot|s',u', \mu')\|_1  \leq &K_\mu \|\mu - \mu'\|_1 + K_s d(s, s') + \frac{K_a}{2} \| u - u'\|_1
\end{align*}
\end{lemma}
\begin{proof}
The lemma simply follows from the Lemma~\ref{lemma:exptv_inequality}, with the inequalities,
\begin{align*}
    |\bar{R}(s,u, \mu) - \bar{R}(s',u', \mu')|  
    \leq &\left|\sum_{a\in\setA} u(a) R(s, a, \mu) - \sum_{a\in\setA} u(a) R(s', a, \mu')\right| \\
         & + \left|\sum_{a\in\setA} u(a) R(s', a, \mu') - \sum_{a\in\setA} u'(a) R(s', a, \mu')\right| \\
    \leq & \sum_{a\in\setA} u(a) |R(s, a, \mu) - R(s', a, \mu')|  \\
         & + \left |\sum_{a\in\setA} u(a) R(s', a, \mu') - \sum_{a\in\setA} u'(a) R(s', a, \mu') \right|,
\end{align*}
where the last line follows from Jensen's inequality.
Using Lemma~\ref{lemma:exptv_inequality} on the second summand we obtain the statement of the lemma.

Similarly for $\bar{P}$, we have
\begin{align*}
    \|\bar{P}(\cdot|s,u, \mu) - \bar{P}(\cdot|s',u', \mu')\|_1  
    \leq &\left\|\sum_{a\in\setA} u(a) P(\cdot|s, a, \mu) - \sum_{a\in\setA} u(a) P(\cdot|s', a, \mu')\right\|_1 \\
         &+ \left\| \sum_{a\in\setA} u(a) P(\cdot|s', a, \mu') - \sum_{a\in\setA} u'(a) P(\cdot|s', a, \mu')\right\|_1 .
\end{align*}
Using Jensen's inequality for the first term and Lemma~\ref{lemma:expvector_inequality} for the second, we conclude.
\end{proof}

\subsection{Lipschitz Continuity of Policy Mirror Ascent}

As opposed to \cite{anahtarci2022q}, in this work we focus on a policy mirror ascent scheme, instead of a best response operator.
This section includes relevant proofs for the continuity of the policy mirror descent response.
Firstly, we prove the following useful lemma.
\looseness=-1

\begin{lemma}[Lipschitz continuity of mirror ascent]\label{lemma:lipschitz_g}
Assume that $u_0 \in \Delta_\setA, q \in \mathbb{R}^\setA$, $\eta >0$, $K\subset \Delta_\setA$ a convex closed set and define $g_\eta: \mathbb{R}^\setA \times \Delta_\setA \rightarrow  K$ as
\begin{align*}
    g_\eta(q, u_0) = \argmax_{u \in K} q^\top u + h(u) - \frac{1}{2\eta} \| u - u_0 \|_2^2.
\end{align*}
Then, $g$ is Lipschitz in both $q, u_0$, that is,
\begin{align*}
    \| g_\eta(q, u_0) - g_\eta(q', u'_0)\|_1 \leq L_{g, q} \| q - q' \|_\infty + L_{g, u} \| u_0 - u'_0 \|_1,
\end{align*}
where $L_{g, q} = \frac{\eta|\setA|}{1 + \rho\eta|\setA|}, L_{g, u} = \frac{1}{|\setA|^{-1} + \eta\rho}$.
\end{lemma}
\begin{proof}
The Lipschitz continuity with respect to $q$ simply follows from Lemma~\ref{lemma:fenchel_duality}, and the fact that $- h(u) + \frac{ 1}{\eta} \| u - u_0 \|_2^2$ is $\rho + \frac{1}{\eta|\setA|}$ strongly convex in $u$ with respect to the norm $\| \cdot\|_1$.
For the continuity with respect to $u_0$, we first define the function
\begin{align*}
f(u) = q^\top u + h(u) - \frac{1}{2\eta} \| u - u_0 \|_2^2.
\end{align*}
Computing the gradient, we obtain:
\begin{align*}
\grad f(u) = q + \grad h(u) - \frac{1}{\eta} (u - u_0).
\end{align*}
By strong concavity of $h$, $-\grad h$ is a strongly monotone operator with parameter $\rho$.
Define by $N_K$ the normal cone operator corresponding to the set $K$.
Then $N_K$ is (maximal) monotone since $K$ is closed and convex.
By first order optimality conditions, we have
\begin{align*}
    u^* = g_\eta(q, u_0) &\iff 0 \in \grad f(u^*) - N_{K}(u^*) \\
                    &\iff 0 \in q + \grad h(u^*) - \frac{1}{\eta} (u^* - u_0) - N_{K}(u^*) \\
                    &\iff 0 \in q + \frac{ 1}{\eta} u_0 - \left( \frac{ 1}{\eta} I - \grad h + N_{K} \right)(u^*) \\
                    & \iff u^* = \left( \frac{ 1}{\eta} I - \grad h + N_{K} \right)^{-1}(q + \frac{ 1}{ \eta}u_0),
\end{align*}
where the last line follows since the resolvent $\left( \frac{ 1}{\eta} I - \grad h + N_{K} \right)^{-1}$ is a function.
Since $N_K$ is monotone, $\frac{1}{\eta}I$ is $\frac{1}{|\setA|\eta}$-strongly monotone in $\|\cdot\|_1$ and $-\grad h$ is $\rho$-strongly monotone in $\|\cdot\|_1$, $G := \left( \frac{ 1}{\eta} I - \grad h + N_{K} \right)$ is a (set-valued) strongly monotone operator with parameter $\frac{1}{|\setA|\eta} + \rho$.
Hence by definition, it holds that for any $y_1 \in G(x_1), y_2 \in G(x_2)$
\begin{align*}
    \|y_1 - y_2 \|_1 \|x_1 - x_2 \|_1 &\geq \|y_1 - y_2 \|_\infty \|x_1 - x_2 \|_1 \\
        &\geq (y_1 - y_2 )^\top (x_1 - x_2)    \\
        &\geq \left( \frac{1}{|\setA|\eta} + \rho\right) \| x_1 - x_2\|_1^2
\end{align*}
Then, $G^{-1}$ is a Lipschitz function with constant $\frac{\eta}{|\setA|^{-1} + \eta\rho}$ (between normed spaces $(\mathbb{R}^\setA, \|\cdot\|_1) \rightarrow (\mathbb{R}^\setA, \|\cdot\|_1)$), hence $G^{-1}(q + \frac{1}{\eta}u_0)$ is $\frac{1}{|\setA|^{-1} + \eta\rho}$ Lipschitz continuous in $u_0$.
It follows that $g$ is Lipschitz continuous in $u_0$ with respect to norm $\|\cdot\|_1$ with constant $L_{g,u}$ defined in the theorem.
\end{proof}

With the above, we can prove Lemma~\ref{lemma:lipschitz_md} below that states Lipschitz continuity in terms of the norms defined on $\Pi, \setQ$ in the main text.

\begin{lemma}[Lipschitz continuity of $\Gamma_{\eta}^{md}$]\label{lemma:lipschitz_md}
$\Gamma_{\eta}^{md}$ is Lipschitz continuous in both of its arguments, so that for all $q, q'\in \setQ, \pi, \pi' \in \Pi$, it holds that $\| \Gamma_{\eta}^{md}(q, \pi) - \Gamma_{\eta}^{md}(q', \pi')   \|_1 \leq L_{md,\pi}\| \pi - \pi'\|_1 + L_{md,q} \| q - q'\|_\infty$,
where $L_{md, q} = \frac{\eta|\setA|}{1 + \eta\rho|\setA|}$ and $L_{md, \pi} = \frac{ 1}{|\setA|^{-1} + \eta\rho}$.
\end{lemma}
\begin{proof}%
For each state $s \in \setS$, we have $\Gamma_{\eta}^{md}(q, \pi)(s) = g_\eta(q(\cdot|s), \pi(s))$ with the set $K:=\setU_{L_h}$ in Lemma~\ref{lemma:lipschitz_g}.
By the Lipschitz continuity of $g_\eta$ shown in Lemma~\ref{lemma:lipschitz_g}, we can write
\begin{align*}
    \| \Gamma_{\eta}^{md}(q, \pi)(\cdot|s) - \Gamma_{\eta}^{md}(q', \pi')(\cdot|s)\|_1 \leq &L_{g, q} \| q(s, \cdot ) - q'(s, \cdot) \|_\infty + L_{g, u} \| \pi(\cdot|s) - \pi'(\cdot|s)\|_1. 
\end{align*}
So taking the supremum of both sides with respect to $s$ and using the definition of norms on $\Pi$, we obtain
\begin{align*}
    \| \Gamma_{\eta}^{md}(q, \pi) - \Gamma_{\eta}^{md}(q', \pi')\|_1 \leq &L_{g, q} \| q - q' \|_\infty + L_{g, u} \| \pi - \pi'\|_1. 
\end{align*}
So the desired Lipschitz constants are given by $L_{md, q} = L_{g, q}, L_{md, \pi} = L_{g, u}$.
\end{proof}

\subsection{Definitions of Value Functions}

In this section, we define various value functions that depend on the population distribution.
The definitions are standard from single-agent RL literature.
\begin{definition}[$V_h(\cdot|\pi, \mu), Q_h(\cdot|\pi, \mu), q_h(\cdot|\pi, \mu)$]
The value function $V$ and the Q-functions $Q, q$ are respectively defined as
\begin{align*}
     V_h(s|\pi, \mu) &:= \Exop \left[ \sum_{t=1}^\infty \gamma^t \left(R(s_t, a_t, \mu) + h(\pi(s_t))\right) \middle| \substack{
s_0 = s \\
s_{t+1} \sim P(\cdot | s_t, \pi_t^i(s_t), \mu) \\
a_t \sim \pi(\cdot|s_t)}
\right] ,\\
Q_h(s,a|\pi, \mu) &:= \Exop \left[ \sum_{t=1}^\infty \gamma^t \left(R(s_t, a_t, \mu) + h(\pi(s_t))\right) \middle| \substack{
s_0 = s, a_0 = a \\
s_{t+1} \sim P(\cdot | s_t, \pi_t^i(s_t), \mu) \\
a_t \sim \pi(\cdot|s_t)}
\right] ,\\
q_h(s,a|\pi, \mu) &:= R(s,a,\mu) + \gamma\sum_{s'\in \setS} \sum_{a'\in \setA} P(s'|s,a,\mu) \pi(a'|s') Q_h(s', a'|\pi, \mu).
\end{align*}
Similarly, Q functions with argument in $u \in \Delta_\setS$ in are defined as for all $s\in\setS$, $\bar{Q}_h(s, u|\pi, \mu) := \sum_a Q_h(s, a|\pi, \mu) u(a)$ and $\bar{q}_h(s, u|\pi, \mu) := \sum_a q_h(s, a|\pi, \mu) u(a)$.
\end{definition}

Likewise, we define the standard optimal value functions for the MDP as follows.

\begin{definition}[$V_h^*(\cdot| \mu), Q_h^*(\cdot|\mu), q_h^*(\cdot|\mu)$]
The optimal value function $V$ and the Q-functions $Q, q$ are respectively defined as
\begin{align*}
     V_h^*(s| \mu) := \max_{\pi\in\Pi} V_h(s|\pi, \mu), \quad
     &Q_h^*(s,a| \mu) := \max_{\pi\in\Pi} Q_h(s,a|\pi, \mu), 
     \\
     q_h^*(s,a|\mu) &:= \max_{\pi\in\Pi} q_h(s,a|\pi, \mu) .
\end{align*}
Similarly, Q functions with argument in $u \in \Delta_\setS$ are defined as for all $s\in\setS$, $\bar{Q}_h^*(s, u|\mu) := \sum_a Q_h^*(s, a| \mu) u(a)$ and $\bar{q}_h^*(s, u|\pi, \mu) := \sum_a q_h^*(s, a|\mu) u(a)$.
\end{definition}

Finally, we state the very useful characterization of value functions as the fixed points of a certain Bellman operator.
The next lemma states this standard result without proof.

\begin{lemma}[Value functions as fixed points]
For any $\pi\in\Pi, \mu\in\Delta_\setS$, the value functions $V_h, Q_h, q_h$ (uniquely) satisfy
\begin{align*}
    V_h(s| \pi, \mu) &= \sum_a \pi(a|s) \left( R(s,a,\mu) + h(\pi(s)) + \gamma \sum_{s'\in \setS}  P(s'|s,a,\mu) V_h(s| \pi, \mu)\right), \\
    Q_h(s,a|\pi, \mu)& = R(s,a,\mu) +h(\pi(s))+ \gamma\sum_{s'\in \setS} \sum_{a'\in \setA} P(s'|s,a,\mu) \pi(a'|s') Q_h(s', a'|\pi, \mu), \\
    q_h(s,a|\pi, \mu) &:= R(s,a,\mu) + \gamma\sum_{s'\in \setS} \sum_{a'\in \setA} P(s'|s,a,\mu) \pi(a'|s') \left(q_h(s',a'|\pi, \mu) + h(\pi(s')) \right).
\end{align*}
Likewise, the optimal value functions are uniquely defined as the fixed points satisfying
\begin{align*}
    V_h^*(s| \mu)& = \max_{u \in \Delta_\setA} \left[\sum_a u(a) \left(R(s,a,\mu) + h(u) + \gamma \sum_{s'\in \setS} P(s'|s,a,\mu) V_h^*(s| \mu) \right) \right],\\
    q_h^*(s, a| \mu)& = R(s,a,\mu) + \gamma\sum_{s'\in \setS} P(s'|s,a,\mu) \max_{u \in \Delta_\setA} \left[ h(u) + \sum_{a'} u(a')  q_h^*(s', a'| \mu) \right].\\
\end{align*}
\end{lemma}

\subsection{Lipschitz Continuity of Value Functions}

In this section, we establish that $\Gamma_q$ is Lipschitz continuous on a well-defined convex subset of $\Pi$.
The main difficulty will be avoiding dependence on the Lipschitz continuity of $h$.
We first prove two technical lemmas.
\begin{lemma}\label{lemma:LVs}
Assume that $\pi \in \Pi_{\Delta h}$ and $\mu \in \Delta_\setS$ arbitrary.
Then, for any $s_1, s_2\in\setS$,
\begin{align*}
|V_h(s_1|\pi, \mu) &- V_h(s_2|\pi, \mu)| \leq L_{V,s} := \frac{L_s + L_a + \Delta h }{1 - \gamma \min\{1, \frac{K_s + K_a}{2}\}}.
\end{align*}
\end{lemma}
\begin{proof}
Using the fixed point definition of $V_h$, Lemma~\ref{lemma:exptv_inequality}, and the definition of $\Pi_{\Delta h}$ in Eq.~\ref{eq:definition_pi_deltah},
\begin{align*}
|V_h(s_1|\pi, \mu) &- V_h(s_2|\pi, \mu)| \\
\leq &|\bar{R}(s_1,\pi(s_1), \mu) - \bar{R}(s_2,\pi(s_2), \mu)| + |h(\pi(s_1)) - h(\pi(s_2))| \\
    &+\gamma \sum_{s'} \left( \bar{P}(s'|s_1, \pi(s_1), \mu) - \bar{P}(s'|s_2, \pi(s_2), \mu)\right) V_h(s|\pi, \mu), \\
\leq&  L_s + L_a + \Delta h \\
    &+\frac{\gamma \sup_{s,s'} |V_h(s|\pi, \mu) - V_h(s'|\pi, \mu)|}{2} \| P(s'|s, \pi(s_1), \mu) - P(s'|s_2, \pi(s_2), \mu)\|_1 ,\\
\leq &L_s + L_a + \Delta h + \frac{\gamma \min\{2, K_s + K_a\}}{2} \sup_{s,s'} |V_h(s|\pi, \mu) - V_h(s'|\pi, \mu)| ,
\end{align*}
which yields the lemma by taking the supremum of the left hand side over $s_1, s_2$.
\end{proof}
\begin{lemma}[Lipschitz continuity of $V_h$]\label{lemma:Vh_lipschitz}
Assume that $\Delta h >0$ arbitrary. For any $\pi, \pi' \in \Pi_{\Delta h}$ and $\mu, \mu' \in \Delta_\setS$,
\begin{align*}
\| V_h(\cdot|\pi,\mu)- V_h(\cdot|\pi',\mu') \|_\infty \leq &L_{V, \pi} \|\pi - \pi'\|_1 + L_{V, \mu} \| \mu - \mu'\|_1,
\end{align*}
for $L_{V, \pi} = \frac{4L_a + \gamma K_a L_{V,s}}{4(1 - \gamma)}, L_{V, \mu} = \frac{2L_\mu + \gamma K_\mu L_{V,s}}{2(1 - \gamma)}$ and $L_{V,s}$ is as defined in Lemma~\ref{lemma:LVs}.
\end{lemma}
\begin{proof}
Similar to previous computations,
\begin{align*}
| V_h(s|\pi,\mu)- V_h(s|\pi',\mu') | \leq &|\bar{R}(s,\pi(s), \mu)-\bar{R}(s,\pi'(s), \mu')| \\
    & + \gamma \left|\sum_{s'} \left(P(s'|s,\pi(s),\mu) V_h(s'|\pi,\mu) - P(s'|s,\pi'(s),\mu') V_h(s'|\pi',\mu')\right) \right| \\
\leq & L_a \| \pi - \pi'\|_1 + L_\mu \|\mu - \mu' \|_1 +  \gamma \frac{L_{V,s}}{2} (K_\mu  \| \mu - \mu'\|_1 + \frac{K_a}{2} \| \pi - \pi' \|_1) \\
    & + \gamma \sup_s | V_h(s|\pi,\mu)- V_h(s|\pi',\mu') |,
\end{align*}
where the last line follows from an application of the triangle inequality and the previous lemma.
\end{proof}

The key result is that $\Gamma_q$ is Lipschitz on a subset of policies given by $\Pi_{\Delta h}$ (Eq.~\eqref{eq:definition_pi_deltah}).
\begin{lemma}[Lipschitz continuity of $\Gamma_q$]\label{lemma:lipschitz:q}
Let $\Delta h > 0$ be arbitrary. 
There exists $L_{q, \pi}, L_{q, \mu}$ depending on $\Delta h$ such that for all $\pi, \pi' \in \Pi_{\Delta h}$ and $\mu, \mu' \in \Delta_\setS$,
\begin{align*}
\|\Gamma_{q}(\pi, \mu) - \Gamma_{q}(\pi', \mu') \|_\infty \leq &L_{q, \pi} \|\pi - \pi'\|_1 + L_{q, \mu} \| \mu - \mu'\|_1.
\end{align*}
\end{lemma}
\begin{proof}%
The result follows from the definition of $q_h$, in terms of $V_h$ since the Lipschitz continuity of $V_h$ has been shown in Lemma~\ref{lemma:Vh_lipschitz}.
Specifically, we have
\begin{align*}
    L_{q,\mu} = L_\mu + \gamma L_{V,\mu} + \gamma \frac{L_{V,s} K_\mu}{2} ,  \quad L_{q,\pi} =\gamma L_{V,\pi} + \gamma L_{V,s} K_a.
\end{align*}
\end{proof}

\subsection{Lipschitz Continuity of Population Update}

\begin{proof}[Proof of Lemma~\ref{lemma:lipschitz_pop_update}]
The proof relies on Lemma~\ref{lemma:expvector_inequality}.
\begin{align*}
    \left\|\Gamma_{pop}(\mu, \pi) - \Gamma_{pop}(\mu', \pi')\right\|_1 = & \left\| \sum_s \mu(s) \bar{P}(\cdot | s, \pi(s), \mu) - \sum_s \mu'(s) \bar{P}(\cdot | s, \pi'(s), \mu')\right\|_1 \\
    \leq & \underbrace{\left\| \sum_s \mu(s) \bar{P}(\cdot | s, \pi(s), \mu) - \sum_s \mu(s) \bar{P}(\cdot | s, \pi'(s), \mu')\right\|_1}_{A}  \\
    &+ \underbrace{\left\| \sum_s \mu(s) \bar{P}(\cdot | s, \pi'(s), \mu') - \sum_s \mu'(s) \bar{P}(\cdot | s, \pi'(s), \mu')\right\|_1}_{B}.
\end{align*}
The first term can be bounded by using the Jensen's inequality:
\begin{align*}
    A \leq& \sum_s \mu(s) \left\| \bar{P}(\cdot | s, \pi(s), \mu) - \bar{P}(\cdot | s, \pi'(s), \mu')\right\|_1 
    \leq  \frac{K_a}{2} \left\| \pi - \pi' \right\|_1 + K_\mu \|\mu - \mu'\|_1.
\end{align*}

For the second term, using Lemma~\ref{lemma:expvector_inequality}, we obtain
\begin{align*}
    B \leq & \| \mu - \mu' \|_1 \frac{\sup_{s, s'\in \setS} \|\bar{P}(\cdot | s, \pi'(s), \mu')-\bar{P}(\cdot | s', \pi'(s'), \mu') \|_1}{2}.
\end{align*}
To bound the supremum, we use Lipschitz continuity of $\bar{P}$, to obtain for $s,s'\in \setS$,
\begin{align*}
    \|\bar{P}(\cdot | s, \pi'(s), \mu')-\bar{P}(\cdot | s', \pi'(s'), \mu') \|_1 \leq &K_s d(s, s') + K_a\|\pi'(s) - \pi'(s') \|_1 \\
    \leq &(K_s + 2K_a ) d(s, s'), 
\end{align*}
from which the lemma follows.
\end{proof}

Finally, we characterize the Lipschitz continuity of $\Gamma_{pop}^\infty$ as mentioned in the main text.
\begin{lemma}[Lipschitz continuity of $\Gamma_{pop}^\infty$]\label{lemma:lipschitz_pop_infty}
Assume that Assumption~\ref{assumption:stable_pop} holds, that is, $L_{pop, \mu} < 1$.
The mapping $\Gamma_{pop}^\infty: \Pi \rightarrow \Delta_\setS$ is then Lipschitz with constant $L_{pop,\infty} := \frac{K_a}{2 \left(1 - L_{pop, \mu}\right)}$.
\end{lemma}
\begin{proof}
Let $\pi, \pi'\in\Pi_L$, then by definition
\begin{align*}
    \|\Gamma_{pop}^\infty (\pi) - \Gamma_{pop}^\infty (\pi')\|_1 = &\| \Gamma_{pop}( \Gamma_{pop}^\infty (\pi), \pi) - \Gamma_{pop}( \Gamma_{pop}^\infty (\pi'), \pi') \|_1 \\
    \leq & L_{pop, \mu} \|\Gamma_{pop}^\infty (\pi) - \Gamma_{pop}^\infty (\pi')\|_1 + \frac{K_a}{2} \| \pi - \pi' \|_1,
\end{align*}
hence the result follows.
\end{proof}

\subsection{Fixed Points and Continuity of $\Gamma_\eta$}

Lemma~\ref{lemma:lipschitz:q} proves the Lipschitz continuity of $\Gamma_q$ only on a restricted subclass of policies.
However, the next result shows that the MFG-NE policy can be contained in a subset $\Pi_{\Delta h}$ for some well-defined $\Delta h > 0$.
\begin{lemma}\label{lemma:optimal_pi_smooth}
Let $\mu\in\Delta_\setS$ arbitrary, and $\pi^*\in\Pi$ the optimal response such that for all $s\in\setS$, $V_h(s|\pi^*,\mu) = \max_{\pi\in\Pi} V_h(s|\pi,\mu)$.
Then, $\pi^* \in \Pi_{L_h}$ where $L_h := L_a + \gamma  \frac{L_s K_a}{2 - \gamma K_s}$.
\end{lemma}
\begin{proof}%
Firstly, using the fixed point definition of the optimal value function, we have for all $s_1, s_2\in\setS$:
\begin{align*}
|V_h^*(s_1|\mu) &- V_h^*(s_2|\mu)| \\
\leq &\Bigg| \sup_{u\in\Delta_\setA} \left(\bar{R}(s_1,u,\mu) + h(u)+ \gamma \sum_{s} \bar{P}(s|s_1,u,\mu) V_h^*(s|\mu)\right)  \\
    &- \sup_{u\in\Delta_\setA} \left(\bar{R}(s_2,u,\mu) + h(u) +\gamma \sum_{s} \bar{P}(s|s_2,u,\mu) V_h^*(s|\mu)\right)\Bigg| \\
\leq &  \sup_{u\in\Delta_\setA} \Bigg| \bar{R}(s_1,u,\mu) - \bar{R}(s_2,u,\mu) + \gamma \sum_{s} (\bar{P}(s|s_1,u,\mu)-\bar{P}(s|s_2,u,\mu)) V_h^*(s|\mu) \Bigg| \\
\leq & L_s + \frac{\gamma K_s}{2} \sup_{s,s'} |V_h^*(s|\mu) - V_h^*(s'|\mu)|,
\end{align*}
hence $|V_h^*(s_1|\mu) - V_h^*(s_2|\mu)|\leq \frac{L_s}{1 - \gamma K_s / 2}$.
Since $q_h^*(s,a|\mu) = R(s,a,\mu)+\gamma\sum_{s'} P(s'|s,a,\mu)V_h^*(s'|\mu)$, we can also conclude that for any $s\in\setS $ that
\begin{align*}
    \sup_{a,a'\in\setA} |q_h^*(s,a|\mu) - q_h^*(s,a'|\mu)| \leq L_a + \gamma \frac{K_a}{2} \frac{L_s}{1 - \gamma K_s / 2}.
\end{align*}

Finally, by optimality conditions of the policy $\pi^*$, for any $s\in\setS$ we have
\begin{align*}
    \max_{u\in\Delta_\setA} \langle q_h^*(s,\cdot|\mu) , u \rangle + h(u) = \langle q_h^*(s,\cdot|\mu) , \pi^*(s) \rangle + h(\pi^*(s))\geq \langle q_h^*(s,\cdot|\mu) , u_{max} \rangle + h_{max}.
\end{align*}
This implies that $h_{max} - h(\pi^*(s)) \leq \langle q_h^*(s,\cdot|\mu) , \pi(s) - u_{max} \rangle$.
\end{proof}

We note that the constant $L_h$ obtained above is comparable to the constant $K_{H_1}$ of \citet{anahtarci2022q}.
Hence the results above generalize the known Lipschitz continuity of optimal Q-values to general Q-values with respect to a policy $\pi$ without introducing stringent assumptions.

\begin{proof}[Proof of Lemma~\ref{lemma:fixed:points}]
If $\pi^*, \mu^*$ satisfy the MFG-NE conditions, the two equalities follow from MDP optimality and Lemma~\ref{lemma:optimal_pi_smooth}.
Conversely, assume $\pi^* = \Gamma_\eta(\pi^*)$ and $\mu^* = \Gamma^\infty_{pop}(\pi^*)$.
Note that this implies $\pi^* = \Gamma_\eta^{md}(q_h(\cdot|\pi^*, \mu^*), \pi^*)$.
By optimality conditions on MDPs, it follows that $\pi^*$ is optimal with respect to the MDP induced by $\mu^*$.
\end{proof}

\begin{lemma}[Lipschitz continuity of $\Gamma_\eta$]\label{lemma:lipschitz_gamma:full}
For any $\eta > 0$, the operator $\Gamma_\eta: \Pi \rightarrow \Pi$ is Lipschitz with constant $L_{\Gamma_\eta}$ on $(\Pi, \|\cdot\|_1)$, where
\begin{align*}
    L_{\Gamma_\eta} := \frac{L_{\Gamma,q} \eta |\setA|}{1 + \eta \rho|\setA| } + \frac{1}{|\setA|^{-1}+ \eta \rho } < \frac{L_{\Gamma,q}}{\rho} + \frac{1}{\eta\rho},
\end{align*}
with the constant $L_{\Gamma,q}$ defined as $L_{\Gamma,q} := L_{pop,\infty} L_{q, \mu} + L_{q,\pi}$.
\end{lemma}
\begin{proof}
The result follows from combining previously established Lipschitz continuity results of $\Gamma_\eta^{md}$ (Lemma~\ref{lemma:lipschitz_md}) and $\Gamma_q$ (Lemma~\ref{lemma:lipschitz:q}) with the definition of $\Gamma_\eta$.
\end{proof}

\begin{proof}[Proof of Theorem~\ref{theorem:deterministic_learning}]
$\pi^*$ is a fixed point of $\Gamma_\eta$ and $\Gamma_\eta$ is a contraction by Lemma~\ref{lemma:lipschitz_gamma}.
\end{proof}

\subsection{Continuity and Best Response}\label{sec:impossibility}

In this self-contained section, we present a short proof that a large class of best response operators can not be continuous and non-trivial in $\mu$ if $\setS,\setA$ are finite, despite what is typically assumed in stationary MFG literature.
Firstly, we define the unregularized optimal action values $Q^*(s, a|\mu)$ for each $\mu\in\Delta_\setS$ as
\begin{align*}
Q^*(s,a|\mu) &:= \max_{\pi \in \Pi} \Exop \left[ \sum_{t=0}^\infty \gamma^t R(s_t, a_t, \mu) \middle|
\substack{
s_0 = s\\
a_0 = a},
\substack{
s_{t+1} \sim P(\cdot | s_t, a_t, \mu) \\
a_{t+1} \sim \pi(s_{t+1})}, \forall t\geq 0
\right].
\end{align*}
For any $s\in\setS$, we also define the set-valued optimal action map $\mathcal{A}_s^*: \Delta_\setS \rightarrow 2^{\setA}$ as
\begin{align*}
    \mathcal{A}^*_s( \mu) := \{ a \in \setA : Q^*(s, a|\mu) \geq Q^*(s, a'|\mu) \text{ for all } a'\in\setA\} \in 2^{\setA},
\end{align*}
for each $\mu \in \Delta_\setS$.
We call a map $\Gamma_{br}: \Delta_\setS \rightarrow \Pi$ a ``best-response operator'' (BR) if
\begin{align*}
    \supp \Gamma_{br}(\mu)(s) \subset \mathcal{A}^*_s(\mu), \forall s\in\setS, \mu\in\Delta_\setS.
\end{align*}
We also denote by $\Pi^*(\mu)$ the set of optimal policies for population distribution $\mu$, hence a valid BR operator must satisfy $\Gamma_{br}(\mu) \in \Pi^*(\mu)$ for all $\mu$.
$\Gamma_{br}$ is not unique in general, since it can assign non-zero action probabilities arbitrarily on $\mathcal{A}^*_s(\mu)$.
The question for this section is if there could be a $\Gamma_{br}$ that assigns probabilities to optimal actions so that it is continuous on $(\Delta_\setS, \|\cdot\|_1)$ (or on any equivalent norm). 
We provide a negative answer for a fairly general subclass of operators between $\Delta_\setS$ and $\Pi$.

\begin{definition}[``Optimal action stable'' policy map]\label{def:optimal_action_stable}
Let $\Gamma: \Delta_\setS \rightarrow \Pi$ be an arbitrary mapping.
Let $S := |\setS|, \setS = \{ s_1, \cdots, s_S\}$.
$\Gamma$ is called \emph{optimal action stable (OAS)} if for any given subsets of actions $A_{1}, \ldots, A_{S} \subset \setA$, on the sets of the form
\begin{align*}
    \omega(A_{1}, \cdots, A_{S}) := \bigcap_{1\leq i \leq S} \left(\setA^*_{s_i}\right)^{-1} ( A_i ) \subset \Delta_\setS
\end{align*}
that are non-empty, $\Gamma$ takes a single value.
Equivalently, for any $A_{1}, \cdots, A_{S} \subset \setA$, if $\omega(A_{1}, \cdots, A_{S}) \neq \emptyset$, the forward image of $\omega(A_{1}, \cdots, A_{S})$ is a single policy: $\Gamma\big(\omega(A_{1}, ..., A_{S}) \big) = \{ \pi_{A_{1}, \cdots, A_{S}} \}$.
\end{definition}
While being a technical condition, OAS is satisfied by practically all best response maps proposed by single-agent RL literature.
We clarify with examples.

\begin{example}[Pure best response]
Fix an ordering $a_1, a_2, \cdots, a_K$ on $\setA$.
Define $\Gamma_{det}: \Delta_\setS \rightarrow \Pi$ so that
\begin{align*}
    \Gamma_{det} (\mu)(a_k | s) =
\begin{cases}
1, \text{ if } a_1, a_2, \cdots, a_{k-1} \notin \setA^*_{s}(\mu), a_k \in \setA^*_{s}(\mu)\\
0, \text{ otherwise }
\end{cases}.
\end{align*}
$\Gamma_{det}$ assigns probability $1$ to the first optimal action in the ordering.
$\Gamma_{det}$ is OAS since the the action probabilities are completely determined by which actions are optimal.
\end{example}

\begin{example}[Uniform map to optimal actions]
Define $\Gamma_{unif}: \Delta_\setS \rightarrow \Pi$ so that $\Gamma_{unif} (\mu)(a | s) = \frac{1}{|\setA^*_{s}(\mu)|}$ if $a \in \setA^*_{s}(\mu)$, otherwise $0$.
That is, $\Gamma_{unif}$ assigns equal probability to all optimal actions at every state $s$.
$\Gamma_{unif}$ is OAS, and is the limiting operator of the Boltzman policy of \cite{guo2019learning}.
\end{example}

\begin{example}[Limit of regularized BR]
Take $h$ strongly concave as before, and consider the regularization function $\tau h(\cdot)$ for $\tau>0$.
Consider $\Gamma^{h}_\tau(\mu) := \argmax_\pi V_{\tau h}(\cdot|\pi,\mu)$.
While for fixed $\tau>0$ the $\Gamma^{BR}_\tau$ is not a valid (unregularized) best response operator, one can take the limit $\tau \rightarrow 0$ and show that
\begin{align*}
    \Gamma_{\tau\rightarrow 0}^h(\mu) := \lim_{\tau\rightarrow 0} \Gamma_{\tau}^h(\mu) = \argmax_{\pi \in \Pi^*(\mu)} \Exop\left[ \sum_{t=0}^\infty \gamma^t h(\pi(s_t)) \middle|
\substack{
s_0 = \cdot \\
a_{t} \sim \pi(s_{t}) \\
s_{t+1} \sim P(\cdot | s_t, a_t, \mu)
}, \forall t\geq 0
\right].
\end{align*}
which implies $\Gamma_{\tau\rightarrow 0}^h$ is a best response operator.
By above, it is also OAS, since it is the unique optimal policy that maximizes the regularizer term on the restricted ``sub''-MDP where at each state $s$ the only available actions are $\setA^*_{s}(\mu)$.
That is, $\Gamma_{\tau\rightarrow 0}^h(\mu)$ depends only on the optimal action sets $\setA^*_{s}(\mu)$.
\end{example}

Finally, with this definition in place, we present the main result of this subsection.

\begin{lemma}[Continuous, OAS $\Gamma$  are constant]\label{lemma:impossibility_oas}
Let $\Gamma: \Delta_\setS \rightarrow \Pi$ be a continuous, OAS map such that $\Gamma_{br}(\mu) \in \Pi^*(\mu)$ for all $\mu\in\Delta_\setS$.
Then, $\Gamma$ is constant on $\Omega$, i.e., for some $\pi_0\in\Pi$, $\Gamma(\mu) = \pi_0, \forall \mu \in \Delta_\setS$.
\end{lemma}
\begin{proof}
$\Delta_\setS$ is a connected set with the topology induced by $\|\cdot\|_1$, hence by continuity of $\Gamma$, $\Gamma(\Delta_\setS)$ must be a connected set in $\Pi$.
There are only finitely many subsets of actions $A_1, A_2, \cdots, A_N \subset \setA$, hence there are only finitely many subsets of $\Omega$ of the form $\omega(A_1, A_2, \cdots, A_N)$
and these form a partition of the domain $\Omega$.
Hence by OAS, the image $\operatorname{Im} \Gamma$ is a discrete set.
A discrete connected set must be a singleton.
\end{proof}
The above lemma shows that blanket continuity assumptions of unregularized best response might be too strong in MFG, reducing the MFG problem to the learning of a constant policy.
The OAS assumption of Lemma~\ref{lemma:impossibility_oas} hints that simply treating best response as a single-agent RL problem will lead to operators that are not continuous (or operators that have exploding Lipschitz constants as smoothing is decreased to approximate the discontinuous best response, as in the case of $\Gamma_{\tau\rightarrow 0}^h$).

\subsection{Regularization and Bias}\label{sec:regularization_and_bias}

In this subsection, we define the unregularized Nash equilibrium and quantify the relationship between the unregularized and regularized Nash equilibrium.

\begin{lemma}[Unregularized Value and MFG-NE]
We define the expected unregularized mean-field reward for a population-policy pair $(\pi,\mu)\in \Pi\times \Delta_\setS$ as
\begin{align*}
    V(\pi, \mu) := \Exop \left[ \sum_{t=0}^\infty \gamma^t R(s_t, a_t, \mu) \middle| \substack{
s_0 \sim \mu \\
a_t \sim \pi(s_t) \\
s_{t+1} \sim P(\cdot | s_t, a_t, \mu)
}
\right].
\end{align*}
A pair $(\pi^*,\mu^*)\in \Pi\times \Delta_\setS$ is called an unregularized MFG-NE if the following hold:
\begin{align*}
    &\text{Stability: } \mu^*(s) = \sum_{s', a'} \mu^*(s')\pi^*(a'|s')P(s|s',a', \mu^*),\\
    &\text{Optimality: } V(\pi^*, \mu^*) = \max_{\pi} V(\pi, \mu^*).
\end{align*}
If the optimality condition is only satisfied with $V(\pi^*_\delta, \mu^*_\delta) \geq \max_\pi V(\pi, \mu^*_\delta) - \delta$, we call $(\pi^*_\delta, \mu^*_\delta)$ an unregularized $\delta$-MFG-NE.
\end{lemma}

In general, the unregularized MFG-NE will not be unique.
As expected, the regularized MFG-NE pair $(\pi^*,\mu^*)$ also forms an unregularized $\delta$-MFG-NE.
We quantify the bias in a straightforward manner in the following lemma.

\begin{lemma}[Regularization Bias on NE]
Let $h:\Delta_\setA \rightarrow [0, h_{max}]$ for some $h_{max} > 0$ and let $(\pi^*_\delta, \mu^*_\delta)$ be a regularized $\delta$-MFG-NE with the regularizer $h$.
Then, $(\pi^*_\delta, \mu^*_\delta)$ is an unregularized $(\delta + \frac{h_{max}}{1-\gamma})$-MFG-NE.
\end{lemma}
\begin{proof}
The stability conditions for regularized and unregularized MFG-NE are identical.
For the optimality condition, we note that for any pair $(\pi, \mu) \in \Pi \times \Delta_\setS$, $|V(\pi, \mu) - V_h(\pi, \mu)| \leq \frac{h_{max}}{1 - \gamma}$.
It follows that for any $\mu\in\Delta_\setS$, $|\max_{\pi} V(\pi, \mu) - \max_{\pi} V_h(\pi, \mu)| \leq \frac{h_{max}}{1 - \gamma}$.
\end{proof}
For instance, using the scaled entropy regularizer $\tau h_{ent}(u) = - \tau \sum_a u(a) \log u(a)$ for $\tau > 0$, the bias will be bounded by $\frac{\tau \log |\setA|}{1 - \gamma}$.

\section{Sample Based Learning}

\subsection{Conditions on $h$ for Persistence of Excitation}\label{section:conditions_on_h}

The learning algorithms presented in this paper assume that persistence of excitation holds throughout training.
We show that for many choices of $h$, the persistence of excitation (PE) assumption (Assumption~\ref{assumption:exploration}) is automatically satisfied.

\begin{lemma}[PE conditions on $h$]\label{lemma:conditions_on_h}
Assume that $h$ is strongly concave with modulus $\rho$, differentiable in $\Delta_\setA^\circ$, $u_{max} \in \Delta_\setA^\circ$, and define $U_\delta := \{u\in \Delta_\setA^\circ : d(u, \partial \Delta_\setA) < \delta\}$.
Further assume that 
\begin{align*}
    \lim_{\delta \rightarrow 0} \inf_{u \in U_\delta} \grad h(u)^\top (u_{max} - u) > Q_{max} + \frac{4}{\eta}.
\end{align*}
Then, there exists $p_{inf} > 0$ such that for all $q \in \setQ, 0\leq q(\cdot,\cdot) \leq Q_{max}, \pi \in \Pi$, it holds that $\Gamma_{md}^\eta(q, \pi)(a|s) > p_{inf}$ for all $s\in\setS, a\in\setA$.
\end{lemma}
\begin{proof}
$\setU_{L_h}$ is a closed convex set (see Definition~\ref{definition:pma_operator} and Eq.~\eqref{eq:definition_pi_deltah}).
If $\setU_{L_h} \cap \partial \Delta_\setA = \varnothing$, we are done, since the image of $\Gamma^\eta_{md}$ is a compact set and $\Gamma^\eta_{md}(s)$ is contained in $\Delta_\setA^\circ$ for all $s \in \setS$.
So we assume $\setU_{L_h} \cap \partial \Delta_\setA \neq \varnothing$.

Denote $\setQ' := \{ q\in\setQ : 0 \leq q(\cdot, \cdot) \leq Q_{max}\}$.
Let $\bar{u} \in \setU_{L_h} \cap \partial \Delta_\setA$ and $q\in\setQ', \pi\in\Pi$ arbitrary.
Define the functions $f:\setU_{L_h} \rightarrow \mathbb{R}$ and $g:[0,1] \rightarrow \mathbb{R}$ as
\begin{align*}
    f(u) := \langle u, q(s, \cdot)\rangle + h(u) - \frac{1}{2\eta} \| u - \pi(s) \|_2^2, \qquad
    g(t) := f(\bar{u} + t(u_{max} - \bar{u})).
\end{align*}
Here $\bar{u} + t(u_{max} - \bar{u}) \in \setU_{L_h}$ for all $t\in[0,1]$ by convexity of $\setU_{L_h}$ and the fact that $u_{max} \in \setU_{L_h}$.
We will show that $f(\bar{u})$ can not be a maximum in $\setU_{L_h}$ by proving $g(0)$ has a direction of increase.
Since $g$ is differentiable in $(0,1)$ and continuous in $[0,1]$, for any $\varepsilon > 0$ by the mean value theorem there exists $\bar{\varepsilon} \in (0, \varepsilon) $ such that $g'(\bar{\varepsilon}) = \frac{g(\varepsilon) - g(0)}{\varepsilon}$.
It is sufficient to show that for small enough $\varepsilon$, $g'(\bar{\varepsilon}) > 0$ for any $\bar{\varepsilon} \in (0,\varepsilon)$.
Computing the derivatives, we obtain
\begin{align*}
    g'(\varepsilon) = &\grad f(\bar{u} + \varepsilon (u_{max} - \bar{u}))^\top (u_{max} - \bar{u}) \\
    = &q(s, \cdot)^\top (u_{max} - \bar{u}) + \grad h(\bar{u} + \varepsilon (u_{max} - \bar{u}))^\top (u_{max} - \bar{u}) \\
    &- \frac{1}{\eta} (\bar{u} + \varepsilon (u_{max} - \bar{u}) - \pi(s))^\top (u_{max} - \bar{u}).
\end{align*}
Note that the magnitudes of the first and last terms can be bounded by
\begin{align*}
    \left|q(s, \cdot)^\top (u_{max} - \bar{u})\right| \leq Q_{max}, \qquad
    \left|\frac{1}{\eta} (\bar{u} + \varepsilon (u_{max} - \bar{u}) - \pi(s))^\top (u_{max} - \bar{u})\right| \leq \frac{4}{\eta}.
\end{align*}
Hence, choosing $\varepsilon > 0$ small enough so that $\grad h(\bar{u} + \bar{\varepsilon} (u_{max} - \bar{u}))^\top (u_{max} - \bar{u}) > Q_{max} + \frac{4}{\eta}$ for all $\bar{\varepsilon}\in(0,\varepsilon)$ shows that $g(\varepsilon) > g(0)$, implying $\bar{u}$ can not be a maximizer of $f$.
This implies $\Gamma_{md}^\eta(q,\pi)(a|s) > 0$ for any $s,a\in\setS\times\setA$.

Since the domain set $\setQ' \times \Pi$ is compact and $\Gamma_{md}^\eta$ is continuous (see Lemma~\ref{lemma:lipschitz_md}), the image set $\Gamma_{md}^\eta(\setQ',\Pi)(a|s)$ is compact for all $s,a\in\setS\times\setA$ and there exists $p_{inf} > 0$ such that $\Gamma_{md}^\eta(q,\pi)(a|s) \geq p_{inf}$ for all $q,\pi\in \setQ' \times \Pi$ and $s,a\in\setS\times\setA$.
\end{proof}

The lemma above does not require differentiability of $h$ at the boundary $\partial \Delta_\setA$.
For instance, for the entropy regularizer $h_{ent}(u) = -\sum_a u(a) \log u(a)$ the assumption of Lemma~\ref{lemma:conditions_on_h} is satisfied for any learning rate $\eta > 0$, since $\lim_{\delta \rightarrow 0} \inf_{u \in U_\delta} \grad h(u)^\top (u_{max} - u) = \infty$.

\subsection{Generalized Monotone Variational Inequalities under Biased Markovian Noise}

In this section, we generalize the results from \citet{kotsalis2022simple} for conditional TD learning to incorporate non-vanishing bias in estimates of the operator with possibly non-Markovian sequences.
We refer the reader to their work for the full presentation of their ideas, while we provide a self-contained proof incorporating bias.

\begin{theorem}[CTD under bias]\label{theorem:ctd_opt}
Let $F: \setX \rightarrow \mathbb{R}^D$ be an operator on some bounded set $\setX\subset\mathbb{R}^d$.
Let $\left\{\xi_t: t \in \mathbb{Z}_{+}\right\}$ be a random process on space $(\Omega, \mathcal{F}, \mathbb{P})$ where $\xi_t \in \Xi \subset \mathbb{R}^n$ with probability 1.
Let $\mathcal{F}_t=\sigma\left(\xi_1, \ldots, \xi_t\right)$ and $\widetilde{F}: \setX \times \Xi \rightarrow \mathbb{R}^D$.
Let $\|\cdot\|$ be a norm on $\mathbb{R}^D$ with dual $\|\cdot\|_*$, and $V$ be a Bregman divergence satisfying $V(x,y) \geq \frac{1}{2}\| x - y\|^2$ for all $x,y \in \setX$.
We assume the following hold.
\begin{enumerate}
    \item [A1.] There exists (a unique) point $x^* \in X$ such that $\left\langle F\left(x^*\right), x-x^*\right\rangle \geq 0, \quad \forall x \in X$.
    \item [A2.] $F$ is ``generalized'' strongly monotone, hence for some $\mu > 0$,
        \begin{align}\label{eq:sgm}
            \left\langle F(x), x-x^*\right\rangle \geq 2\mu V(x, x^*), \quad \forall x \in X.
        \end{align}
    \item [A3.] $F$ is Lipschitz with constant $L$, so that $\left\|F\left(x_1\right)-F\left(x_2\right)\right\|_* \leq L\left\|x_1-x_2\right\|, \forall x_1,x_2\in\setX$.
    \item [A4.] There exists $\sigma, \varsigma > 0$ such that for all iterates $\{x_t\}$ and $t,\tau\in\mathbb{Z}_+$,
        \begin{align}\label{eq:ctdvar}
            \Exop\left[\left\|\widetilde{F}\left(x_t, \xi_{t+\tau}\right)-\mathbb{E}\left[\widetilde{F}\left(x_t, \xi_{t+\tau}\right) \mid \mathcal{F}_{t-1}\right]\right\|_*^2 \mid \mathcal{F}_{t-1}\right] \leq \frac{\sigma^2}{2}+\frac{\varsigma^2}{2}\left\|x_t-x^*\right\|^2.
        \end{align}
    \item [A5.] There exists $\delta > 0, C > 0, \rho \in (0,1) $ such that $\forall t,\tau\in\mathbb{Z}_+$ almost surely
        \begin{align}\label{eq:ctdmixcond}
            \left\|F(x)-\mathbb{E}\left[\widetilde{F}\left(x, \xi_{t+\tau}\right) \mid \mathcal{F}_{t-1}\right]\right\|_* \leq C \rho^\tau\left\|x-x^*\right\| + \delta, \forall x\in\setX.
        \end{align}
\end{enumerate}
Then, for $\tau > \frac{\log{1/\mu} + \log{20C}}{\log{1/\rho}},
    t_0=\max \left\{\frac{8 L^2}{\mu^2}, \frac{16 \varsigma^2}{\mu^2}\right\}, \beta_k=\frac{2}{\mu\left(t_0+k-1\right)}$,
the iterations given by arbitrary $x_1 \in \setX$ and
\begin{align}\label{eq:ctd_iteration}
x_{k+1}=\argmin_{x \in \setX} \beta_k\left\langle\widetilde{F}\left(x_k, \xi_{k\tau}\right), x\right\rangle+V(x_k, x), \forall k \in \mathbb{Z}_+,
\end{align}
satisfy
\begin{align*}
    \Exop[V(x_{k+1}, x^*)] \leq &\frac{2(t_0 + 1)(t_0+2)V(x_1, x^*)}{(k+t_0)(k+t_0+1)} + \frac{6k(\sigma^2+ 4\delta^2)}{\mu^2(k+t_0)(k+t_0+1)} + \frac{100\delta^2}{\mu^2}.
\end{align*}
\end{theorem}
\begin{proof}
Firstly, we note that by the triangle inequality and application of Young's inequality, for any sequence $\{ x_t\}_t$ we have that
\begin{align*}
    \|F(x_t)-\widetilde{F}(x_t, \xi_{t+\tau})\|_*^2 \leq &2\|F(x_t)-\mathbb{E}[\widetilde{F}(x_t, \xi_{t+\tau}) \mid \mathcal{F}_{t-1}]\|_*^2 \\
    &+2\|\mathbb{E}[\widetilde{F}(x_t, \xi_{t+\tau}) \mid \mathcal{F}_{t-1}]-\widetilde{F}(x_t, \xi_{t+\tau})\|_*^2.
\end{align*}
By taking expectations, applying Equations~\ref{eq:ctdmixcond} and \ref{eq:ctdvar} and Youngs inequality,
we can conclude (similar to Lemma~2.1 of \citet{kotsalis2022simple} apart from the bias term) that
\begin{align}\label{eq:boundctdl2}
    \mathbb{E}[\|F(x_t)-\widetilde{F}(x_t, \xi_{t+\tau})\|_*^2] \leq \sigma^2+4\delta^2+(4C^2 \rho^{2 \tau}+\varsigma^2) \mathbb{E}[\|x_t-x^*\|^2].
\end{align}

Now for the sequence of updates $\{x_k \}$ defined by Equation~\ref{eq:ctd_iteration}, as in \citet{kotsalis2022simple}, we define
\begin{align*}
    \Delta F_k:=F\left(x_k\right)-F\left(x_{k-1}\right) \text { and } \delta_k^\tau:=\widetilde{F}\left(x_k, \xi_{k\tau}\right)-F\left(x_k\right).
\end{align*}
Reiterating Proposition 3.7 of \cite{kotsalis2022simple}, by optimality conditions of Equation~\ref{eq:ctd_iteration} it holds again that for all $x\in\setX$,
\begin{align*}
    \beta_k\left\langle F\left(x_{k+1}\right), x_{k+1}-x\right\rangle+\left(1-2 \beta_k^2 L^2\right) V\left(x_{k+1}, x\right)+\beta_k\left\langle\delta_k^\tau, x_k-x\right\rangle \leq V\left(x_k, x\right)+\beta_k^2\left\|\delta_k^\tau\right\|_*^2.
\end{align*}
Fixing $x=x^*$ and using the strong generalized monotonicity condition of (\ref{eq:sgm}), we obtain
\begin{align*}
(1+ 2\mu \beta_k-2 \beta_k^2 L^2) \mathbb{E}[V(x_{k+1}, x^*)] \leq & \Exop[V(x_k, x^*)]+\beta_k^2 \mathbb{E}[\|\delta_t^\tau\|_*^2]+\beta_k \mathbb{E}[|\langle\delta_k^\tau, x_k-x^*\rangle|] \\
    \leq & \Exop[V(x_k, x^*)]+\beta_k^2 \mathbb{E}[\|\delta_t^\tau\|_*^2]+\beta_k \mathbb{E}[\|\delta_k^\tau\|_* \|x_k-x^*\|] \\
    \leq &\Exop[V(x_k, x^*)]
        +\beta_k^2 (\sigma^2 + 4\delta^2)
        +\delta \beta_k  \mathbb{E}[\|x_k-x^*\|]\\
        &+(4C^2 \rho^{2\tau}\beta_k^2 +\varsigma^2\beta_k^2 + C\beta_k\rho^\tau) \Exop[\| x_k - x^*\|^2],
\end{align*}
where the last inequality holds by (\ref{eq:boundctdl2}) and (\ref{eq:ctdmixcond}).
Hence by the property of divergence $V$:
\begin{align*}
(1+ 2\mu \beta_k-2 \beta_k^2 L^2) \mathbb{E}[V(x_{k+1}, x^*)] \leq &\Exop[V(x_k, x^*)] (1 + 8C^2 \rho^{2\tau}\beta_k^2 +2\varsigma^2\beta_k^2 + 2C\beta_k\rho^\tau)\\
&+ \beta_k^2 (\sigma^2 + 4\delta^2) + \beta_k \delta \mathbb{E}[\|x_k-x^*\|].
\end{align*}
Applying Young's inequality on the term $\beta_k \delta \mathbb{E}[\|x_k-x^*\|]$, we obtain:
\begin{align*}
(1+2\mu \beta_k-2 \beta_k^2 L^2) \mathbb{E}[V(x_{k+1}, x^*)] \leq &\Exop[V(x_k, x^*)] (1 + 8C^2 \rho^{2\tau}\beta_k^2 +2\varsigma^2\beta_k^2 + 2C\beta_k\rho^\tau)\nonumber\\
&+ \beta_k^2 (\sigma^2 + 4\delta^2) + \beta_k \frac{10\delta^2}{\mu} + \beta_k\frac{\mu}{10}\mathbb{E}[\|x_k-x^*\|^2].
\end{align*}
Hence, we conclude
\begin{align}\label{eq:ctd_final_pot}
(1+ 2\mu \beta_k-2 \beta_k^2 L^2) \mathbb{E}[V(x_{k+1}, x^*)] \leq &\Exop[V(x_k, x^*)] (1 + 8C^2 \rho^{2\tau}\beta_k^2 +2\varsigma^2\beta_k^2 + 2C\beta_k\rho^\tau + \frac{\beta_k\mu}{10})\nonumber\\
&+ \beta_k^2 (\sigma^2 + 4\delta^2) + \beta_k \frac{10\delta^2}{\mu}.
\end{align}

Finally, setting $\theta_k := (k+t_0)(k+t_0+1)$, we note that with the assumptions on $\tau, t_0$,
\begin{align*}
    \theta_k (1 + 8C^2 \rho^{2\tau}\beta_k^2 +2\varsigma^2\beta_k^2 + 2C\beta_k\rho^\tau + \frac{\beta_k\mu}{10})\leq \theta_{k-1} (1 + 2\mu\beta_{k-1}-2\beta_{k-1}^2 L^2).
\end{align*}
Multiplying both sides of (\ref{eq:ctd_final_pot}) by $\theta_k$ and summing over $k=1, \ldots,K$, we obtain
\begin{align*}
    \theta_K (1 + 2\mu\beta_{K}-2\beta_{K}^2 L^2) \Exop[V(x_{K+1}, x^*)] \leq & \theta_1(1 + 8C^2 \rho^{2\tau}\beta_1^2 +2\varsigma^2\beta_1^2 + 2C\beta_1\rho^\tau)V(x_1, x) \\
    &+\sum_{k=1}^K \theta_k \beta_k^2 (\sigma^2 + 4\delta^2) + \sum_{k=1}^K \theta_k \beta_k \frac{10\delta^2}{\mu}.
\end{align*}
The result follows by computing the sums and $\sum_{k=1}^K \theta_k \beta_k \leq \frac{10}{\mu}(K+t_0)(K+t_0+1)$.
\end{proof}

\subsection{Stochastic Population Update Bounds}\label{section:mixing_bounds}

We establish two useful results that (1) bound the expected deviation of the empirical population distribution from the mean field and (2) characterize mixing in terms of the state visitation probabilities of each agent.
Firstly, we show that the empirical population distribution $\widehat{\mu}_t$ approximates the mean field $\Gamma_{pop}^\infty(\pi)$ in expectation up to a bias scaling with $\mathcal{O}(\frac{1}{\sqrt{N}})$.

\begin{lemma}[Empirical population bound]\label{lemma:bounds:population:full}
Assume that at any time $t\geq 0$, each agent $i$ follows a given (arbitrary) policy $\pi^i \in \Pi$ so that,
\begin{align*}
    a_{t}^i &\sim \pi^i(s_{t}^i), \quad s_{t+1}^i \sim P(\cdot|s_{t}^i, a_{t}^i,\widehat{\mu}_t), \quad \forall t \geq 0, i=1,\ldots, N.
\end{align*}
Let $\bar{\pi}\in\Pi$ be an arbitrary policy and $\Delta_{\bar{\pi}} := \frac{1}{N}\sum_i \| \bar{\pi} - \pi^i\|_1$.
For any $\tau, t \geq 0$, it holds that
\begin{align*}
    \Exop\left[ \left\|\widehat{\mu}_{t+\tau} - \Gamma_{pop}^{\tau}(\widehat{\mu}_t, \bar{\pi})\right\|_1 \middle| \mathcal{F}_t \right]
\leq \frac{1 - L_{pop, \mu}^{\tau}}{1 - L_{pop, \mu}} \left( \frac{\sqrt{2|\setS|}}{\sqrt{N}} + \frac{\Delta_{\bar{\pi}} K_a}{2} \right).
\end{align*}
\end{lemma}
\begin{proof}
For the bias term when $\tau=1$, we compute:
\begin{align*}
    \Exop\left[ \widehat{\mu}_{t+1}  \middle| \mathcal{F}_t \right] &=  \Exop\left[ \frac{1}{N} \sum_{s'\in \setS} \sum_{i=1}^N \mathbbm{1}(s_{t+1}^i = s') \vece_{s'} \middle| \mathcal{F}_t \right]
    = \sum_{s'\in \setS} \vece_{s'} \sum_{i=1}^N \frac{1}{N}  \bar{P}(s'|s_{t}^i, \pi^i(s_{t}^i), \widehat{\mu}_t),
\end{align*}
where we use the $\mathcal{F}_t$-measurability of $\widehat{\mu}_t$ and $s^i_t$.
We then obtain
\begin{align*}
    \|\Gamma_{pop}(\pi, \widehat{\mu}_t) - \Exop\left[ \widehat{\mu}_{t+1}  \middle| \mathcal{F}_t \right] \|_1
    =&\left\| \sum_{s'\in \setS} \vece_{s'} \sum_{i=1}^N \frac{1}{N}  (\bar{P}(s'|s_{t}^i, \pi^i(s_{t}^i), \widehat{\mu}_t) - \bar{P}(s'|s_{t}^i, \bar{\pi}(s_{t}^i), \widehat{\mu}_t)) \right\|_1 \\
        \leq& \frac{1}{N}\sum_{i=1}^N\left\| (\bar{P}(\cdot|s_{t}^i, \pi^i(s_{t}^i), \widehat{\mu}_t) - \bar{P}(\cdot|s_{t}^i, \bar{\pi}(\cdot|s_{t}^i), \widehat{\mu}_t)) \right\|_1 \leq \frac{K_a \Delta_{\bar{\pi}}}{2} .
\end{align*}
We also compute the variance at timestep $t+1$.
For $\tau=1$, by independence, we can decompose the $\ell_2$-variance:
\begin{align*}
    \Exop[ \| \widehat{\mu}_{t+1} - \Exop[ \widehat{\mu}_{t+1}  | \mathcal{F}_t ] \|_2^2 | \mathcal{F}_t ] &=  \frac{1}{N^2}  \sum_{i=1}^N \Exop[ \| \vece_{s_{t+1}^i} -  \Exop[  \vece_{s_{t+1}^i}  | \mathcal{F}_t ] \|_2^2 | \mathcal{F}_t ] \leq \frac{2}{N},
\end{align*}
where we use the fact that $\| \vece_{s_{t+1}^i} -  \Exop[  \vece_{s_{t+1}^i}  | \mathcal{F}_t ] \|_2^2 \leq 2$.
In particular, we obtain
\begin{align*}
    \Exop\left[ \| \widehat{\mu}_{t+1} - \Exop\left[ \widehat{\mu}_{t+1}  \middle| \mathcal{F}_t \right] \|_1 \middle| \mathcal{F}_t \right] & = 
    \sqrt{ \Exop\left[ \| \widehat{\mu}_{t+1} - \Exop\left[ \widehat{\mu}_{t+1}  \middle| \mathcal{F}_t \right] \|_1 \middle| \mathcal{F}_t \right]^2 } \\
    &\leq \sqrt{ |\setS| \Exop\left[ \| \widehat{\mu}_{t+1} - \Exop\left[ \widehat{\mu}_{t+1}  \middle| \mathcal{F}_t \right] \|_2^2 \middle| \mathcal{F}_t \right] } \\
    &\leq \frac{\sqrt{2|\setS|}}{\sqrt{N}}
\end{align*}
using Jensen's inequality and the fact that $\|v\|_1 \leq \sqrt{d} \|v\|_2$ for any $v \in \mathbb{R}^d$.
Hence we have
\begin{align*}
    \Exop[ \| \widehat{\mu}_{t+1} - \Gamma_{pop}(\widehat{\mu}_t, \bar{\pi})\|_1  | \mathcal{F}_t ] \leq &\Exop[ \| \widehat{\mu}_{t+1} - \Exop\left[ \widehat{\mu}_{t+1}  \middle| \mathcal{F}_t \right] \|_1  | \mathcal{F}_t ] + \Exop[ \| \Exop\left[ \widehat{\mu}_{t+1}  \middle| \mathcal{F}_t \right] - \Gamma_{pop}(\widehat{\mu}_t, \bar{\pi})\|_1  | \mathcal{F}_t ] \\
       \leq &\frac{\sqrt{2|\setS|}}{\sqrt{N}} + \frac{K_a}{2} \Delta_{\bar{\pi}}.
\end{align*}

For $\tau > 1$, we inductively generalize the result.
By law of iterated expectations (and the fact that $ \mathcal{F}_{t} \subset \mathcal{F}_{t+\tau}$), we can derive
\begin{align*}
    \Exop\left[ \left\|\widehat{\mu}_{t+\tau+1} - \Gamma_{pop}^{\tau+1}(\widehat{\mu}_t, \pi)\right\|_1 \middle| \mathcal{F}_t \right] \leq &\Exop\left[ \left\|\widehat{\mu}_{t+\tau+1} - \Gamma_{pop}(\widehat{\mu}_{t+\tau}, \bar{\pi})\right\|_1 \middle| \mathcal{F}_t \right] \\
        &+ \Exop\left[ \left\|\Gamma_{pop}(\widehat{\mu}_{t+\tau}, \pi) - \Gamma_{pop}^{\tau+1}(\widehat{\mu}_t, \pi)\right\|_1 \middle| \mathcal{F}_t \right] \\
    \leq &\Exop\left[ \Exop\left[\left\|\widehat{\mu}_{t+\tau+1} - \Gamma_{pop}(\widehat{\mu}_{t+\tau}, \pi)\right\|_1 \middle| \mathcal{F}_{t + \tau} \right]\middle| \mathcal{F}_t \right] \\ 
        &+\Exop\left[ L_{pop, \mu} \left\|\widehat{\mu}_{t+\tau} - \Gamma_{pop}^{\tau}(\widehat{\mu}_t, \pi)\right\|_1 \middle| \mathcal{F}_t \right] \\
    \leq& \frac{\sqrt{2|\setS|}}{\sqrt{N}} + \frac{\Delta_{\bar{\pi}} K_a}{2} 
        +L_{pop, \mu} \Exop\left[  \left\|\widehat{\mu}_{t+\tau} - \Gamma_{pop}^{\tau}(\widehat{\mu}_t, \pi)\right\|_1 \middle| \mathcal{F}_t \right],
\end{align*}
where the last inequality follows from what has been proven above for $\tau=1$ and the Lipschitz continuity of the operator $\Gamma_{pop}(\cdot, \pi)$.

Hence, we inductively obtain the bound for any $\tau > 0$,
\begin{align*}
\Exop\left[ \left\|\widehat{\mu}_{t+\tau} - \Gamma_{pop}^{\tau}(\widehat{\mu}_t, \pi)\right\|_1 \middle| \mathcal{F}_t \right]
\leq \frac{1 - L_{pop, \mu}^{\tau}}{1 - L_{pop, \mu}} \left( \frac{\sqrt{2|\setS|}}{\sqrt{N}} + \frac{\Delta_{\bar{\pi}} K_a}{2} \right).
\end{align*}
\end{proof}

The dependence on $\Delta_{\bar{\pi}}$ above indicates that if the policies of agents deviate from each other, an additional bias will be incurred on the empirical population distribution.
In the centralized learning case, we will have $\Delta_{\bar{\pi}} = 0$, whereas the term will be significant for the independent learning case due to the variance in agents' policy updates.
As corollaries, we have the following bounds in terms of the limiting population distribution and the empirical population bound conditioned on the state of a single agent.

\begin{corollary}[Convergence to stable distribution]\label{corollary:convergence:stable:full}
Under the conditions of Lemma~\ref{lemma:bounds:population:full} for any $t, \tau\geq 0$, we have
\begin{align*}
    \Exop\left[ \left\|\widehat{\mu}_{t+\tau} - \Gamma_{pop}^\infty(\bar{\pi})\right\|_1 \middle| \mathcal{F}_t \right]
\leq \frac{1}{1 - L_{pop, \mu}} \left(\frac{\sqrt{2|\setS|}}{\sqrt{N}} + \frac{K_a\Delta_{\bar{\pi}}}{2}\right) + 2 L_{pop, \mu}^{\tau}.
\end{align*}
\end{corollary}
\begin{proof}
The proof follows from an application of triangle inequality and Lemma~\ref{lemma:bounds:population:full}.
\end{proof}

\begin{corollary}\label{corollary:conditional_population}
Assume the conditions of Lemma~\ref{lemma:bounds:population:full} and Assumption \ref{assumption:mixing}.
Then, for any $\bar{s}\in\setS$, agent $j\in[N]$, and for $\tau > T_{mix}$, we have
\begin{align*}
    \Exop\left[ \left\|\widehat{\mu}_{t+\tau} - \Gamma_{pop}^{\tau}(\widehat{\mu}_t, \pi)\right\|_1 \middle| s_{t+\tau}^j = \bar{s}, \mathcal{F}_t \right]
\leq \frac{1 - L_{pop, \mu}^{\tau}}{(1 - L_{pop, \mu})\delta_{mix}} \left( \frac{\sqrt{2|\setS|}}{\sqrt{N}} + \frac{\Delta_{\bar{\pi}} K_a}{2} \right).
\end{align*}
\end{corollary}
\begin{proof}
The corollary follows from the law of total expectation and the fact that by mixing conditions, $\Prob[s_{t+\tau}^j = \bar{s}|\mathcal{F}_t]>\delta_{mix}$.
\end{proof}

As a stronger result, we prove that any agent's state visitation probabilities also converge to the mean field up to a population bias term.
This result generalizes similar mixing theorems for Markov chains to the case where there is time dependence due to population dynamics and later facilitates proving learning bounds on a single sample path.
The proof is loosely based on the ideas from Theorem 4.9 of \cite{levin2017markov}, with the additional complication that the transitions are not homogeneous due to population effects.

\begin{proposition}[Mean field MC convergence]\label{theorem:bounds:populationMC:full}
Let $\{s_{0}^i\}_{i} \in \setS^N$ be (arbitrary) initial states of each agent.
Assume that each player follows a policy $\pi^i\in \Pi$, that is,
\begin{align*}
    a_{t}^i \sim \pi^i(s_{t}^i), \quad
    s_{t+1}^i &\sim P(\cdot|s_{t}^i, a_{t}^i,\widehat{\mu}_t), \quad \forall t \geq 0, i=1,\ldots, N .
\end{align*}
Let $\bar{\pi}\in\Pi$ be an arbitrary policy and $\Delta_{\bar{\pi}} := \frac{1}{N} \sum_{i} \| \bar{\pi} - \pi^i\|_1$.
For any $T \geq 0$ and for any $i=1,\ldots,N$, it holds that
\begin{align*}
    \|\Prob[ s_T^i = \cdot ] - \Gamma_{pop}^\infty(\bar{\pi})\|_1 \leq & C_{mix} \rho_{mix}^{T}+ \frac{2K_\mu T_{mix}}{\delta_{mix}^2}\frac{\sqrt{2|\setS|}}{\sqrt{N}} \\
    &+ \frac{T_{mix} K_a K_\mu}{\delta_{mix}^2 (1-L_{pop,\mu})} \Delta_{\bar{\pi}} + \frac{T_{mix}K_a}{\delta_{mix}} \| \pi^i - \bar{\pi}\|_1,
\end{align*}
where $\rho_{mix}:=\max\left\{L_{pop, \mu}, (1-\delta_{mix})^{1/T_{mix}}\right\}$ and $C_{mix}:=\frac{4 T_{mix} \max\{K_\mu, 1\}}{\delta_{mix} \rho_{mix}^{T_{mix}} |L_{pop, \mu} - (1-\delta_{mix})^{1/T_{mix}}|}$.
\end{proposition}
\begin{proof}
Denote $\mu_\infty = \Gamma_{pop}^\infty(\pi)$, and denote the transition matrix induced by the limiting population as $[\matP_\infty]_{s, s'} = \bar{P}(s'|s, \bar{\pi}(s), \mu_\infty)$ 
and note that by definition, $\mu_\infty$ is the limiting distribution of the Markov chain induced by transition probabilities $[\matP_\infty]_{s, s'}$.
By the irreducability and aperiodicity implied by Assumption~\ref{assumption:mixing}, it is in fact the unique stationary distribution of the stochastic matrix $\matP_\infty$.
Finally, by $\matM_\infty$ denote the (stochastic) matrix with all rows equal to $\mu_\infty$.
Then it holds that $\matX \matM_\infty = \matM_\infty$ for \emph{any} stochastic matrix $\matX$, and $ \matM_\infty \matP_\infty = \matM_\infty$.
For a sequence of matrices $\{\matA_i\}_i$, while matrix multiplication is not commutative in general, in this proof we denote $\prod_{i=1}^I \matA_i := \matA_1 \matA_2 \ldots \matA_I$ for simplicity.

We prove the result for the first agent ($i=1$), from which the theorem will follow by symmetry.
By $\matP_t$ denote the stochastic transition matrix at time $t$ given by $[\matP_t]_{s, s'} = \Prob(s^1_{t} = s'|s^1_{t-1}=s)$.
By Assumption~\ref{assumption:mixing}, there exists a $T_{mix}>0$ such that for some $\delta_{mix}>0$, the matrix defined by 
$\matP^{(j)} = \prod_{t=(j-1) T_{mix} + 1}^{jT_{mix}} \matP_t$ 
satisfies for all $j$ that $[\matP^{(j)}]_{s, s'} > \delta_{mix} \mu_\infty(s')> 0$.
Hence for each $j$, we define the stochastic matrices $\matQ^{(j)}$ implicitly given by the equation $\matP^{(j)} = (1 - \theta) \matM_\infty + \theta \matQ^{(j)}$.
where $\theta := 1 - \delta_{mix}$.
The proof will follow in three steps.

\paragraph{Step 1: Induction on $j$.}
By induction, we will prove that for all $J>0$, the following holds:
\begin{align}
    \prod_{j=1}^{J} \matP ^{(j)} = (1 - \theta^J) \matM_\infty + \theta^J \prod_{j=1}^{J} \matQ ^{(j)} 
    + \sum_{l=2}^{J}(1-\theta^{l-1})\theta^{J-l} \matM_\infty \left( \matP^{(j)} - \matP_\infty^{T_{mix}}\right) \prod_{l'=l+1}^J \matQ^{(l')}. \label{eq:matp_induction}
\end{align}
The identity can be verified in a straightforward manner for $J=1,2$.
Assuming the identity holds for $J>1$, we show it holds for $J+1$.
Denoting $\triangle := (\prod_{j=1}^{J} \matP ^{(j)}) \matP^{(J+1)}$, and distributing over the equality for the inductive assumption on $J$, we obtain
\begin{align*}
   \triangle  = 
     &(1 - \theta^J) \matM_\infty \matP^{(J+1)} + \theta^J \prod_{j=1}^{J} \matQ ^{(j)} \left( (1 - \theta) \matM_\infty + \theta \matQ^{(J+1)} \right)  \\
        & + \left(\sum_{l=2}^{J}(1-\theta^{l-1})\theta^{J-j} \matM_\infty \left( \matP^{(j)} - \matP_\infty\right) \prod_{l'=l+1}^J \matQ^{(l')} \right) \left( (1 - \theta) \matM_\infty + \theta \matQ^{(J+1)} \right).
\end{align*}
We observe that $\prod_{j=1}^{J} \matQ ^{(j)} $ is a stochastic matrix and that $\matM_\infty \matP_\infty = \matM_\infty$ to obtain:
\begin{align*}
    = & (1 - \theta^J) \matM_\infty \left(\matP^{(J+1)} - \matP_\infty^{T_{mix}} \right) + (1 - \theta^J) \matM_\infty \matP_\infty^{T_{mix}} + \theta^J(1 - \theta)  \matM_\infty + \theta^{J+1} \prod_{j=1}^{J} \matQ ^{(j)} \\
        &+\left(\sum_{l=2}^{J}(1-\theta^{l-1})\theta^{J-j} \matM_\infty \left( \matP^{(j)} - \matP_\infty^{T_{mix}}\right) \prod_{l'=l+1}^J \matQ^{(l')} \right) \left( (1 - \theta) \matM_\infty + \theta \matQ^{(J+1)} \right) \\
    = &(1 - \theta^J) \matM_\infty \left(\matP^{(J+1)} - \matP_\infty^{T_{mix}} \right) + (1 - \theta^{J+1}) \matM_\infty + \theta^{J+1} \prod_{j=1}^{J} \matQ ^{(j)} \\
        &+\left(\sum_{j=2}^{J}(1-\theta^{l-1})\theta^{J-j} \matM_\infty \left( \matP^{(j)} - \matP_\infty^{T_{mix}}\right) \prod_{l'=l+1}^J \matQ^{(l')} \right) \left( (1 - \theta) \matM_\infty + \theta \matQ^{(J+1)} \right).
\end{align*}
The result follows by the identity
\begin{align*}
    \left(\sum_{j=2}^{J}(1-\theta^{l-1})\theta^{J-j} \matM_\infty \left( \matP^{(j)} - \matP_\infty^{T_{mix}}\right) \prod_{l'=l+1}^J \matQ^{(l')} \right) (1 - \theta) \matM_\infty = 0
\end{align*}
since both $\matP^{(j)}\prod_{l'=j+1}^J \matQ^{(l')}$ and $\matP_\infty\prod_{l'=j+1}^J \matQ^{(l')}$ are stochastic matrices.

\paragraph{Step 2: Quantifying convergence.}
Using the definition $\matP^{(j)} = \prod_{t=(j-1) T_{mix} + 1}^{jT_{mix}} \matP_t$ we can rewrite (\ref{eq:matp_induction}) equivalently as 
\begin{align*}
\left(\prod_{t=1}^{J T_{mix}} \matP_t \right)= (1 - \theta^J) \matM_\infty + \theta^J \prod_{j=1}^{J} \matQ ^{(j)} 
    + \sum_{j=2}^{J}(1-\theta^{l-1})\theta^{J-j} \matM_\infty\left( \matP^{(j)} - \matP_\infty\right) \prod_{l'=l+1}^J \matQ^{(l')}.
\end{align*}
Multiplying both sides by $\matE^{(r)} := \prod_{t=1}^{r} \matP_{JT_{mix}+t}$ from the right for $r < T_{mix}$ and using $\matM_\infty \matP_\infty^r = \matM_\infty$, we obtain
\begin{align*}
    \prod_{t=1}^{J T_{mix} +r} \matP_t - \matM_\infty =& \theta^J \left( \prod_{j=1}^{J} \matQ ^{(j)} \matE^{(r)}- \matM_\infty \matE^{(r)} \right) + \matM_\infty \left(\matE^{(r)} - \matP_\infty^r\right)\\
        &+ \sum_{j=2}^{J}(1-\theta^{l-1})\theta^{J-j} \matM_\infty \left( \matP^{(j)} - \matP_\infty\right) \prod_{l'=l+1}^J \matQ^{(l')} \matE^{(r)}.
\end{align*}
Let $\vecu \in \Delta_{\setS}$ be an arbitrary column vector.
Multiplying each side by $\vecu^\top$ on the left, and taking the $\ell_1$ norm, we have for $\square:= \left\|\vecu^\top \prod_{t=1}^{J T_{mix} +r} \matP_t - \vecmu_\infty \right\|_1$,
\begin{align*}
 \square \leq &2\theta^J + \sum_{j=2}^{J}\theta^{J-j}\left\| \bar{\vecu}^\top \left( \matP^{(j)} - \matP_\infty\right) \bar{\matQ}^{(j)}\right\|_1 + \left\|\bar{\vecu}^\top \left(\matE^{(r)} - \matP_\infty^r\right)\right\|_1 \\
    \leq &2\theta^J + \sum_{j=2}^{J} \theta^{J-j}
    \sum_{t=(j-1)T_{mix}+1}^{jT_{mix}} \sup_s \| \matP_{t, s \cdot } - \matP_{\infty,  j \cdot }\|_1 + \sum_{t=jT_{mix}+1}^{jT_{mix}+r} \sup_s \| \matP_{t,s \cdot } - \matP_{\infty,  j\cdot }\|_1,
\end{align*}
where $\matP_{t, s \cdot }, \matP_{\infty,  s \cdot }$ indicate $s$-th row vectors of matrices $\matP_t, \matP_{\infty}$, and $\bar{\vecu} \in \Delta_{\setS}$ and $\bar{\matQ}^{(j)}$ is a stochastic matrix, using Lemma~\ref{lemma:stochastic_matrices_l1} in the last line.

\paragraph{Step 3: Quantifying finite population bias.}
Finally, we will quantify the error due to the non-vanishing $\sup_s \| \matP_{t, s\cdot} - \matP_{\infty, s\cdot }\|_1$ terms.
We denote by $\widehat{\Delta}_{N, \setS} \subset \Delta_\setS$ the (finite) set of possible empirical state distributions with $N$ agents.
Observe that for any $s\in\setS$,
\begin{align*}
\matP_{t, s\cdot } - \matP_{\infty, s\cdot} &= \Prob(s_{t+1}^1 = \cdot | s_{t}^1 = s) - P(\cdot | s, \bar{\pi}(s), \mu_\infty) \\
    \overset{\star}&{=} \sum_{\mu \in \widehat{\Delta}_{N, \setS} } \left(\Prob(s_{t+1}^1 = \cdot| \widehat{\mu}_t = \mu, s_{t}^1 = s) - P(\cdot | s, \bar{\pi}(s), \mu_\infty)\right) \Prob(\widehat{\mu}_t = \mu| s_{t}^1 = s) \\
    \overset{\triangle}&{=} \sum_{\mu \in \widehat{\Delta}_{N, \setS} } \left(P(\cdot| \mu, \pi^1(s), s) - P(\cdot | s, \bar{\pi}(s), \mu_\infty)\right) \Prob(\widehat{\mu}_t = \mu| s_{t}^1 = s),
\end{align*}
where $(\star)$ follows from the law of total probability and $(\triangle)$ follows from the definition of SAGS dynamics.
Assume $t \geq T_{mix}$.
We then conclude
\begin{align*}
    \|\matP_{t, s \cdot } - \matP_{\infty, s \cdot }\|_1 \leq 
     & \sum_{\mu \in \widehat{\Delta}_{N, \setS} } \left\|P(\cdot| \mu, \pi^1(s), s) - P(\cdot | s, \bar{\pi}(s), \mu_\infty)\right\|_1 \Prob(\widehat{\mu}_t = \mu| s_{t}^1 = s) \\
    \leq &\sum_{\mu \in \widehat{\Delta}_{N, \setS} } \left(K_\mu \| \mu_\infty - \mu\|_1 + \frac{K_a}{2}\| \pi^1 - \bar{\pi}\|_1\right) \Prob(\widehat{\mu}_t = \mu| s_{t}^1 = s) \\
    \leq &\frac{K_a}{2}\| \pi^1 - \bar{\pi}\|_1 + \frac{K_\mu}{\delta_{mix}} \Exop[ \| \mu_\infty - \widehat{\mu}_t\|_1] ,
\end{align*}
where the last line is a consequence of Corollary~\ref{corollary:conditional_population}.
Using Lemma~\ref{lemma:bounds:population:full}, for $t\geq T_{mix}$,
\begin{align*}
     \|\matP_{t,  s\cdot} - \matP_{\infty, s\cdot }\|_1 &\leq \left(\frac{K_\mu K_a}{(1 - L_{pop, \mu}) \delta_{mix}} \right)\frac{\Delta_{\bar{\pi}}}{2} + \frac{K_a}{2} \| \pi^1 - \bar{\pi} \|_1+ \frac{K_\mu}{\delta_{mix}}\frac{\sqrt{2|\setS|}}{\sqrt{N}} + \frac{2K_\mu}{\delta_{mix}} L_{pop, \mu}^{t}.
\end{align*}
Placing this in the inequality we obtained from step 2, we obtain result.
\end{proof}

An intuitive detail above is that the mixing constant $\rho_{mix}$ is equal to the maximum of the Markov chain mixing constant $(1-\delta_{mix})^{1/T_{mix}}$ and the population mixing (or contraction) constant $L_{pop, \mu}$.
Hence, the bound above suggests waiting for both the Markov chain \emph{and} the population is essential.
We use these convergence results to prove finite sample bounds for TD-learning in $N$-player SAGS in the next section.

\subsection{CTD with Population}\label{section:ctd_population}

For the purpose of having a clearer presentation and clarifying the dependence on problem parameters, we define the problem-dependent constants (where $L_h$ such that $\pi\in\Pi_{L_h}$):
\begin{align*}
\underbar{M}_{td} &:= \frac{\log{1/\mu_F} + \log{40C_{mix}}}{\log{1/\rho_{mix}}}, \qquad
t_0:= \frac{16 (1+\gamma)^2}{\mu_F^2}, \\
    C^{TD}_1 &:= \frac{2(1+L_h)(t_0 + 2)|\setS||\setA|}{1-\gamma}, \qquad C^{TD}_2 := \frac{16 (1 + L_h)}{(1-\gamma)^2\delta_{mix} p_{inf}}, \\
    C^{TD}_{pop, 1} &:= \frac{20(K_\mu+L_\mu)(1+L_h)}{(1-\gamma)^2 \delta_{mix}p_{inf}}, \qquad C^{TD}_{pop, 2} := \frac{10(9K_\mu + L_\mu) (1 + L_h) T_{mix}\sqrt{2|\setS|}}{(1-L_{pop,\mu})(1-\gamma)^2\delta_{mix}^3 p_{inf}}, \\
    C^{TD}_{pol,1} &:= \frac{5T_{mix}(1+L_h) (K_\mu + L_\mu + 8K_a K_\mu) }{(1-L_{pop,\mu})(1-\gamma)^2 \delta_{mix}^3 p_{inf}}, \qquad C^{TD}_{pol,2} := \frac{40(1+L_h)K_a T_{mix}}{(1-\gamma)^2 \delta_{mix}^2 p_{inf}} + \frac{20C_h + 10L_h}{(1-\gamma) \delta_{mix} p_{inf}},
\end{align*}
where $C_h$ is defined as the Lipschitz constant of $h$ on the set $\{ u \in \Delta_\setA: u(a) \geq p_{inf}, \forall a\in\setA\}$ with respect to the $\|\cdot\|_1$, which is guaranteed to exists since $\nabla h$ is continuous on the compact set $\{ u \in \Delta_\setA: u(a) \geq p_{inf}, \forall a\in\setA\}$.
We provide the full proof and restatement of Theorem~\ref{theorem:ctd_with_pop}.

\begin{theorem}[CTD learning with population]\label{theorem:ctd_with_pop:full}
Assume Assumption \ref{assumption:mixing} holds and let policies $\{\pi^i\}_i$ be given so that $\pi^i(a|s)\geq p_{inf}$ for all $i$.
Assume Algorithm~\ref{alg:ctd} is run with policies $\{\pi^i\}_i$, \emph{arbitrary} initial agent states $\{s_0^i\}_i$, learning rates $\beta_m=\frac{2}{(1 - \gamma)\left(t_0+m-1\right)}, \forall m \geq 0$ and any $M>0$, $M_{td} > \underbar{M}_{td}$.
If $\bar{\pi} \in \Pi$ is an arbitrary policy, $\Delta_{\bar{\pi}}:= \frac{1}{N}\sum_{i} \| \pi^i - \bar{\pi}\|_1$ and $Q^* := Q_h(\cdot,\cdot|\bar{\pi},\mu_{\bar{\pi}})$, then the (random) output $\widehat{Q}_M$ of Algorithm~\ref{alg:ctd} satisfies
\begin{align*}
    \Exop[ \|\widehat{Q}_{M} - \widehat{Q}^* \|_\infty] \leq &
     \frac{C^{TD}_1}{\sqrt{(M+t_0)(M+t_0+1)}} + \frac{C^{TD}_2 \sqrt{M}}{\sqrt{(M+t_0)(M+t_0+1)}} + C^{TD}_{pop,1} L_{pop,\mu} ^ {M_{td}} \\
    &+ \frac{C^{TD}_{pop,2}}{\sqrt{N}} + C^{TD}_{pol,1} \Delta_{\bar{\pi}} + C^{TD}_{pol,2} \|\pi^1 - \bar{\pi} \|_1.
\end{align*} 
\end{theorem}
\begin{proof}
We verify the assumptions of Theorem~\ref{theorem:ctd_opt}.
Take the divergence $V(q,q') =\frac{1}{2} \| q - q'\|_2^2$, and denote $\zeta_{\alpha}^{m}:=\zeta_{mM_{td}+\alpha}, \mathcal{F}^m_\alpha := \mathcal{F}_{mM_{td}+\alpha}$.
Assumptions A1-A3 have been shown, with generalized strong monotonicity modulus $\mu_F := (1-\gamma)\delta_{mix}p_{inf}$ and Lipschitz constant $L_F := (1+\gamma)$.
For A4, we bound the variance of $\widetilde{F}$ by (see \citet{lan2022policy}):
\begin{align*}
    \Exop\left[ \| \widetilde{F}^{\pi^1}(\widehat{Q}_m, \zeta_{\alpha}^{m}) - \Exop[ \widetilde{F}^{\pi^1}(\widehat{Q}_m, \zeta_{\alpha}^{m}) | \mathcal{F}^m_0]\|_2^2 | \mathcal{F}^m_0\right]
    &\leq 4(1+\gamma)^2 \Exop[ \| \widehat{Q}_m - Q^*\|^2_2 ] + \frac{4(1 + L_h)^2}{(1-\gamma)^2}.
\end{align*}
During CTD learning, the iterates $\widehat{Q}_m$ stay in the set $[\frac{h_{max}-L_h}{1-\gamma}, Q_{max}]$ since $\widehat{Q}_0(s,a) = Q_{max}$ at initialization.

Finally, we verify the (more challenging) mixing condition.
As before, let $\mu_\pi := \Gamma^\infty_{pop}(\bar{\pi})$.
For any $Q\in\setQ$, we have
\begin{align*}
    \Exop[\widetilde{F}^{\pi^1}(Q, \zeta_{t+\tau})| \mathcal{F}_{t}] = &\Exop[(Q(s_{t+\tau}, a_{t+\tau}) - R(s_{t+\tau}, a_{t+\tau}, \mu_\pi) - h(\pi^1(s)) -\gamma Q(s_{t+\tau+1}, a_{t+\tau+1})) \vece_{s_{t+\tau},a_{t+\tau}}| \mathcal{F}_{t}] \\
    &+ \Exop[(R(s_{t+\tau}, a_{t+\tau}, \mu_\pi) - R(s_{t+\tau}, a_{t+\tau}, \widehat{\mu}_{t+\tau}))\vece_{s_{t+\tau},a_{t+\tau}}| \mathcal{F}_{t}].
\end{align*}
By Lemma~\ref{lemma:bounds:population:full}, the second term can be bounded by 
\begin{align*}
    \|\Exop[(R(s_{t+\tau}, a_{t+\tau}, \mu_\pi) - R(s_{t+\tau}, a_{t+\tau}, \widehat{\mu}_{t+\tau}))\vece_{s_{t+\tau},a_{t+\tau}}| \mathcal{F}_{t}]\|_2 \leq \frac{L_\mu}{1-L_{pop,\mu}}\left(\frac{\sqrt{2|\setS|}}{\sqrt{N}} + \frac{\Delta_{\bar{\pi}} K_a}{2}\right) + 2 L_\mu L_{pop,\mu}^\tau.
\end{align*}
For the first term, abbreviating $\vecv_{s,a}^{s',a'} := \pi^1(a|s) \pi^1(a'|s')(Q(s,a) - R(s,a, \mu_\pi) - h(\pi^1(s)) -\gamma Q(s',a')) \vece_{s,a}$ and $\bar{\vecv}_{s,a}^{s',a'} := \bar{\pi}(a|s) \bar{\pi}(a'|s')(Q(s,a) - R(s,a, \mu_\pi) - h(\bar{\pi}(s)) -\gamma Q(s',a')) \vece_{s,a}$, and again denoting by $\widehat{\Delta}_{N, \setS} \subset \Delta_\setS$ the (finite) set of possible empirical state distributions with $N$ agents,
\begin{align*}
    \Exop[&(Q(s_{t+\tau}, a_{t+\tau}) - R(s_{t+\tau}, a_{t+\tau}, \mu_\pi) - h(\pi^1(s)) -\gamma Q(s_{t+\tau+1}, a_{t+\tau+1})) \vece_{s_{t+\tau},a_{t+\tau}}| \mathcal{F}_{t}] \\
= &\sum_{\substack{s,s',a,a' \\ \mu\in\widehat{\Delta}_{N,\setS}}} \Prob[s_{t+\tau} = s|\mathcal{F}_{t}]  \Prob[\widehat{\mu}_{t+\tau} = \mu|s_{t+\tau} = s,\mathcal{F}_{t}] P(s'|s,a,\mu) 
     \vecv_{s,a}^{s',a'} \\
= & \underbrace{\sum_{\substack{s,s',a,a' \\ \mu\in\widehat{\Delta}_{N,\setS}}} \Prob[s_{t+\tau} = s|\mathcal{F}_{t}] \Prob[\widehat{\mu}_{t+\tau} = \mu|s_{t+\tau} = s,\mathcal{F}_{t}] \left(P(s'|s,a,\mu)  - P(s'|s,a,\mu_\pi)\right)
     \vecv_{s,a}^{s',a'}}_{(\star)} \\
    & + \underbrace{\sum_{s,s',a,a'} \Prob[s_{t+\tau} = s|\mathcal{F}_{t}]  P(s'|s,a,\mu_\pi)
     (\vecv_{s,a}^{s',a'} - \bar{\vecv}_{s,a}^{s',a'})}_{(\square)}+ \underbrace{\sum_{s,s',a,a'} \Prob[s_{t+\tau} = s|\mathcal{F}_{t}]  P(s'|s,a,\mu_\pi)
     \bar{\vecv}_{s,a}^{s',a'}}_{(\triangle)}.
\end{align*}
We analyze the three terms above separately.
For $(\star)$, observe the inequality using Corollary~\ref{corollary:conditional_population}:
\begin{align*}
    \sum_{\mu\in\widehat{\Delta}_{N,\setS}} 
    &\Prob[\widehat{\mu}_{t+\tau} = \mu|s_{t+\tau} = s,\mathcal{F}_{t}] \left\| P(\cdot|s,a,\mu)  - P(\cdot|s,a,\mu_\pi)\right\|_1
    \leq \frac{ K_\mu}{(1-L_{pop,\mu})\delta_{mix}}\left(\frac{\sqrt{2|\setS|}}{\sqrt{N}} + \frac{\Delta_{\bar{\pi}} K_a}{2}\right) + 2 K_\mu L_{pop,\mu}^\tau.
\end{align*}
Here using Jensen's inequality and Lemma~\ref{lemma:expvector_inequality},
\begin{align*}
    \| \star \|_1 \leq \left(\frac{ K_\mu}{(1-L_{pop,\mu})\delta_{mix}}\left(\frac{\sqrt{2|\setS|}}{\sqrt{N}} + \frac{\Delta_{\bar{\pi}} K_a}{2}\right) + 2 K_\mu L_{pop,\mu}^\tau \right) \frac{1 + L_h}{1 - \gamma}.
\end{align*}
Similarly, the term $(\square)$ can be bounded by $\|\square\|_1 \leq \| \pi^1 - \bar{\pi} \|_1 (2C_h + L_h)$.

We finally analyze the term $(\triangle)$ using Proposition~\ref{theorem:bounds:populationMC:full}.
We note that
\begin{align*}
    (\triangle) - F^\pi(Q) = (\matM_{\tau} - \matM^\pi) (\matI - \gamma \matP_\infty ) (Q - Q^*),
\end{align*}
where $\matM_{\tau} := \operatorname{diag}(\{\Prob[s^1_{t+\tau}=s|\mathcal{F}_t]\bar{\pi}(a|s)\}_{s,a}) \in\mathbb{R}^{|\setS||\setA|\times|\setS||\setA|}$ and taking $Q, Q^*$ as vectors in $\mathbb{R}^{|\setS||\setA|}$.
Utilizing Proposition~\ref{theorem:bounds:populationMC:full} to bound $\sigma_{max}(\matM_{\tau} - \matM^\pi)$, we conclude
\begin{align*}
    \| \triangle - F^\pi(Q)\|_2 \leq &\Bigg( 2C_{mix} \rho_{mix}^{T}+ \frac{4K_\mu T_{mix}}{\delta_{mix}^2}\frac{\sqrt{2|\setS|}}{\sqrt{N}} \\
    &+ \frac{2T_{mix} K_a K_\mu}{\delta_{mix}^2(1-L_{pop,\mu})}\Delta_{\bar{\pi}} + \frac{2T_{mix}K_a}{\delta_{mix}} \| \pi^i - \bar{\pi}\|_1 \Bigg) \| Q - Q^* \|_2.
\end{align*}
Since we have $\| Q - Q^* \|_2 \leq \frac{2(1 + L_h)}{1 - \gamma}$ and $\|\cdot\|_2\leq\|\cdot\|_1$, the theorem follows.
\end{proof}

\subsection{Main Results}

We restate and prove Theorem~\ref{theorem:convergence_centralized} and Theorem~\ref{theorem:convergence_decentralized}.

\begin{theorem}[Centralized learning]\label{theorem:convergence_centralized:full}
Assume that $\eta > 0$ an arbitrary learning rate which satisfies $L_{\Gamma_\eta} < 1$,  
Assumptions~\ref{assumption:lipschitz}, \ref{assumption:stable_pop}, \ref{assumption:exploration} and \ref{assumption:mixing} hold and $\pi^*$ is the unique MFG-NE.
Let $\varepsilon > 0$ be arbitrary.
If the learning rates $\{\beta_m\}$ are as defined in Lemma~\ref{theorem:ctd_with_pop},
\begin{align*}
    M_{td} &> \max \left\{\frac{\log (4 (1 - L_{\Gamma_\eta})^{-1} L_{md,q}C^{TD}_{pop,1} \varepsilon^{-1})}{\log(L_{pop,\mu}^{-1})}, \underbar{M}_{td} \right\}, \quad K > \frac{\log 8\varepsilon^{-1}}{\log L_{\Gamma_\eta}^{-1}} , \quad \text{ and } \\
    &M_{pg} > \max \Big\{ 4C_1^{TD} L_{md,q}(1 - L_{\Gamma_\eta})^{-1}\varepsilon^{-1}, 16(C_2^{TD})^2 L_{md,q}^2(1 - L_{\Gamma_\eta})^{-2}\varepsilon^{-2}\Big\},
\end{align*}
then the (random) output $\pi_K$ of Algorithm~\ref{alg:centralized} satisfies
\begin{align*}
    \Exop\left[ \| \pi_K - \pi^*\|_1 \right] \leq \varepsilon + \frac{L_{md,q} C^{TD}_{pop,2}}{(1 - L_{\Gamma_\eta})\sqrt{N}}.
\end{align*}
\end{theorem}
\begin{proof}
Denote $q_k := q_h(\cdot,\cdot|\pi_k, \Gamma_{pop}^\infty(\pi_k))$.
Firstly, by Theorem~\ref{theorem:ctd_with_pop}, for any $k\in 0,\ldots,K-1$, it holds for all combinations of states $\{\bar{s}^i\}_{i} \in \setS ^ N$ that with probability 1,
\begin{align*}
    \Exop[\| \widehat{q}_k - q_k\|_\infty | s^i_{kM_{td}M_{pg}}=\bar{s}^i, \pi_k] \leq & \frac{C^{TD}_1}{\sqrt{(M_{pg}+t_0)(M_{pg}+t_0+1)}} + \frac{C^{TD}_2 \sqrt{M_{pg}}}{\sqrt{(M_{pg}+t_0)(M_{pg}+t_0+1)}}  \\
    &+ C^{TD}_{pop,1} L_{pop,\mu} ^ {M_{td}}+ \frac{C^{TD}_{pop,2}}{\sqrt{N}},
\end{align*}
since the policies followed by each agent are the same.
Hence by iterated expectations and placing the definitions of $M_{pg}, M_{td}$, we obtain with probability 1,
\begin{align*}
    \Exop [\| \widehat{q}_k - q_k \|_\infty | \pi^k] \leq \frac{3(1 - L_{\Gamma_\eta}) \varepsilon}{4L_{md,q}} + \frac{C^{TD}_{pop,2}}{\sqrt{N}}.
\end{align*}
Moreover, with probability 1, 
\begin{align*}
    \Exop[ \| \pi_{k+1} - \pi^* \|_1 | \pi_k] = &\Exop[ \|\Gamma_{\eta}^{md}(\widehat{q}_k, \pi_k) - \pi^* \|_1 | \pi_k] \\
    \leq & \Exop [\|\Gamma_{\eta}^{md}(q_k, \pi_k) - \pi^* \|_1 | \pi_k] + \Exop [\|\Gamma_{\eta}^{md}(q_k, \pi_k) - \Gamma_{\eta}^{md}(\widehat{q}_k, \pi_k) \|_1 | \pi_k] \\
    \leq &\Exop [\|\Gamma_\eta(\pi_k) - \pi^* \|_1 | \pi_k] + \Exop [\|\Gamma_{\eta}^{md}(q_k, \pi_k) - \Gamma_{\eta}^{md}(\widehat{q}_k, \pi_k) \|_1 | \pi_k] \\
    \leq & L_{\Gamma_\eta} \| \pi_{k} - \pi^*\|_1 + L_{md,q} \Exop [\| \widehat{q}_k - q_k \|_\infty | \pi_k],
\end{align*}
which implies by the law of iterated expectations:
\begin{align*}
    \Exop[\| \widehat{\pi}_{k+1} - \pi^*\|_1] \leq L_{\Gamma_\eta}\Exop[\| \widehat{\pi}_{k} - \pi^*\|_1] + \frac{3(1 - L_{\Gamma_\eta}) \varepsilon}{4} + \frac{L_{md,q} C^{TD}_{pop,2}}{\sqrt{N}}.
\end{align*}
Inductively applying the inequality for $k=0,\ldots,K-1$ implies the statement of the theorem, noting $\|\pi_0 -\pi^*\|_1 \leq 2$.
\end{proof}

Finally, we restate Theorem~\ref{theorem:convergence_decentralized} with explicit constants and provide the proof.

\begin{theorem}[Independent learning]\label{theorem:convergence_decentralized:full}
Assume that $\eta > 0$ satisfies $L_{\Gamma_\eta} < 1$, 
Assumptions~\ref{assumption:lipschitz}, \ref{assumption:stable_pop}, \ref{assumption:exploration} and \ref{assumption:mixing} hold and $\pi^*$ is the unique MFG-NE.
Let $\varepsilon > 0$ be arbitrary.
Let the learning rates $\{\beta_m\}$ for CTD be as defined in Lemma~\ref{theorem:ctd_with_pop}, and $K > \frac{\log 8\varepsilon^{-1}}{\log L_{\Gamma_\eta}^{-1}}$.
For $c_\eta := L_{md,q} (C^{TD}_{pol,1} + C^{TD}_{pol,2}) + L_{md,\pi} = \mathcal{O}(\frac{1}{\rho})$:
\begin{enumerate}
    \item If $c_\eta < 1$, let $M_{td} > \max \left\{\frac{\log [4 (1 - L_{\Gamma_\eta})^{-1}(1 - c_{\eta})^{-1} (L_{md,q} + L_{md,q}^2C^{TD}_{pop,1}) \varepsilon^{-1}]}{\log(L_{pop,\mu}^{-1})}, \underbar{M}_{td} \right\}$ and $M_{pg} > 4 (1-c_\eta)^{-1} (1 - L_{\Gamma_\eta})^{-1} (L_{md,q} + L_{md,q}^2C^{TD}_{pop,1}) \max \Big\{ C_1^{TD} \varepsilon^{-1}, (C_2^{TD})^2 (1 - L_{\Gamma_\eta})^{-1}\varepsilon^{-2}\Big\}$.
    \item If $c_\eta = 1$, let $M_{td} > \max \left\{\frac{\log [4 (1 - L_{\Gamma_\eta})^{-1} L_{md,q}C^{TD}_{pop,1} \varepsilon^{-1} (1 + C^{TD}_{pol,1} L_{md,q}K)]}{\log(L_{pop,\mu}^{-1})}, \underbar{M}_{td} \right\}$ and $M_{pg} > $ $4 \max \Big\{  \frac{L_{md,q} C_1^{TD}}{1 - L_{\Gamma_\eta}} (1+ C^{TD}_{pol,1} L_{md,q} K)\varepsilon^{-1}, \frac{4(C_2^{TD})^2 L_{md,q}^2}{(1 - L_{\Gamma_\eta})^2}(1+ C^{TD}_{pol,1} L_{md,q} K)^2\varepsilon^{-2}\Big\}$.
    \item If $c_\eta > 1$, let $M_{td} > \max \left\{\frac{\log [4 (1 - L_{\Gamma_\eta})^{-1} L_{md,q}C^{TD}_{pop,1} (1+\frac{C^{TD}_{pol,1} L_{md,q}}{c_\eta-1}c_\eta^K)\varepsilon^{-1}]}{\log(L_{pop,\mu}^{-1})}, \underbar{M}_{td} \right\}$ and $M_{pg} > 4 (1 - L_{\Gamma_\eta})^{-2} L_{md,q} \max \Big\{ C_1^{TD} \varepsilon^{-1}, 4(C_2^{TD} L_{md,q})^2\varepsilon^{-2}\Big\} (1+\frac{C^{TD}_{pol,1} L_{md,q}}{c_\eta-1}c_\eta^K)^2$.
\end{enumerate}
Then, the (random) output $\{\pi_K^i\}_i$ of Algorithm~\ref{alg:decentralized} satisfies for all agents $i=1,\ldots,N$, $\Exop\left[ \| \pi_K^i - \pi^*\|_1 \right] \leq \varepsilon + \mathcal{O}\left(\frac{1}{\sqrt{N}}\right)$.
\end{theorem}
\begin{proof}
The proof again follows by using previous error propagation results for CTD learning with population and the contractivity of $\Gamma_\eta$.
As the main difference from the centralized algorithm analysis, the constant $c_\eta$ characterizes if stochasticity will cause the policies of agents to diverge over PMA epochs.

By symmetry, we only prove the theorem for the first agent $(i=1)$.
We will use the reference policy $\bar{\pi}_k := \pi^1_k$ at each iteration $k$.
For all $k=0,1,\ldots,K$ define the random variable $\Delta_k := \sum_{i=1}^N \| \pi^i_k - \pi^1_k\|_1$.
Clearly $\Delta_0 = 0$.
We will prove a bound for $\Delta_k$ throughout training.
Using Theorem~\ref{theorem:ctd_with_pop} on the CTD iterations of agent $i$, we obtain
\begin{align*}
    \Exop[ \|\widehat{q}_k^i - \widehat{q}_k^1\|_\infty | \{\pi_k^i\}_i] \leq &\frac{C^{TD}_1}{\sqrt{(M_{pg}+t_0)(M_{pg}+t_0+1)}} + \frac{C^{TD}_2 \sqrt{M_{pg}}}{\sqrt{(M_{pg}+t_0)(M_{pg}+t_0+1)}}  \\
    &+ C^{TD}_{pop,1} L_{pop,\mu} ^ {M_{td}}+ \frac{C^{TD}_{pop,2}}{\sqrt{N}} + C^{TD}_{pol,1} \Delta_k + C^{TD}_{pol,2} \|\pi^i_k - \pi^1_k \|_1.
\end{align*}
For simplicity, we denote the first four summands independent of $k$ as $\epsilon_{TD}$, yielding
\begin{align*}
    \Exop[ \|\widehat{q}_k^i - \widehat{q}_k^1\|_\infty | \{\pi_k^i\}_i] \leq \epsilon_{TD} + C^{TD}_{pol,1} \Delta_k + C^{TD}_{pol,2} \|\pi^i_k - \pi^1_k \|_1.
\end{align*}
Using the iterative expectations used to prove Theorem~\ref{theorem:convergence_centralized}, 
we have for all $i\neq 1$ with probability 1:
\begin{align*}
\Exop[ \| \pi_{k+1}^i - \pi^1_{k+1} \|_1 | \{\pi_k^i\}_i] =& \Exop[ \| \Gamma^{md}_\eta(\widehat{q}_k^i,\pi_{k}^i) -\Gamma^{md}_\eta(\widehat{q}_k^1,\pi_{k}^1)\|_1 | \{\pi_k^i\}_i] \\
\leq & L_{md,q}\Exop[ \|\widehat{q}_k^i - \widehat{q}_k^1\|_\infty | \{\pi_k^i\}_i] + L_{md,\pi} \| \pi_k^i -  \pi_k^1\|_1 \\
\leq & L_{md,q} \epsilon_{TD} + L_{md,q} C^{TD}_{pol,1} \Delta_k + (L_{md,q} C^{TD}_{pol,2} + L_{md,\pi}) \|\pi^i_k - \pi^1_k \|_1.
\end{align*}
In this last inequality, dividing by $N$, summing over all $i=1,\ldots,N$ and taking the expectation of both sides we obtain:
\begin{align}\label{eq:expincrease}
    \Exop[ \Delta_{k+1} ] \leq L_{md,q} (C^{TD}_{pol,1} + C^{TD}_{pol,2}) \Exop[\Delta_k] + L_{md,\pi}\Exop[\Delta_k] + L_{md,q} \epsilon_{TD}.
\end{align}
Hence we obtain inductively the bound (for all $k$) and using the definition of $c_\eta$ and $\Delta_0 = 0$,
\begin{align*}
    \Exop[ \Delta_{k} ] \leq \frac{( L_{md,q} (C^{TD}_{pol,1} + C^{TD}_{pol,2}) + L_{md,\pi})^{k+1} - 1}{L_{md,q} (C^{TD}_{pol,1} + C^{TD}_{pol,2}) + L_{md,\pi} - 1} L_{md,q} \epsilon_{TD} = \frac{c_\eta^{k+1} - 1}{c_\eta - 1} L_{md,q} \epsilon_{TD},
\end{align*}
if $c_\eta \neq 1$, otherwise $\Exop[ \Delta_{k} ] \leq k L_{md,q} \epsilon_{TD}$.
Applying Theorem~\ref{theorem:ctd_with_pop} once more on agent 1:
\begin{align*}
    \Exop[\| \widehat{\pi}^1_{k+1} - \pi^*\|_1] \leq L_{\Gamma_\eta}\Exop[\| \widehat{\pi}_{k} - \pi^*\|_1] + L_{md,q} \epsilon_{TD} + L_{md,q} C^{TD}_{pol,1} \Exop[\Delta_k],
\end{align*}
and summing over $k$ iteratively and denoting $c_\eta := L_{md,q} (C^{TD}_{pol,1} + C^{TD}_{pol,2}) + L_{md,\pi}$:
\begin{align*}
     \Exop[\| \widehat{\pi}^1_{K} - \pi^*\|_1] \leq 2 L_{\Gamma_\eta}^K + \frac{L_{md,q} \epsilon_{TD}}{1 - L_{\Gamma_\eta}} + L_{md,q} C^{TD}_{pol,1} \sum_{k=1}^{K-1} L_{\Gamma_\eta}^{K-k-1} \Exop[ \Delta_{k}].
\end{align*}
The result follows by placing the defined constants in the theorem.
\end{proof}

\end{document}